\documentclass{article}
\usepackage{graphicx}
\usepackage{amsmath,amsthm,amssymb,pifont,colortbl,amscd, wrapfig}
\usepackage{bigdelim,multirow}
\newtheorem{theorem}{Theorem}[section]
\newtheorem{lemma}[theorem]{Lemma}
\newtheorem{proposition}[theorem]{Proposition}

\newtheorem{corollary}[theorem]{Corollary}
\newtheorem{conjecture}[theorem]{Conjecture}

\theoremstyle{definition}

\newtheorem{remark}[theorem]{Remark}
\newtheorem*{acknowledgments}{Acknowledgments}
\newtheorem*{claim1}{Claim 1}
\newtheorem*{claim2}{Claim 2}

\def\ev{\mathrm{ev}}
\def\c{\mathop{\mathrm{cts}}\nolimits}

\def\deg{\mathop{\mathrm{deg}}\nolimits}
\def\inv{\mathop{\mathrm{Inv}}\nolimits}
\def\cl{\mathop{\mathrm{cl}}\nolimits}
\def\ad{\mathop{\mathrm{ad}}\nolimits}
\newcommand{\sad}{\overline{\ad}}

\def\id{\mathop{\mathrm{id}}\nolimits}
\def\ev{\mathop{\mathrm{ev}}\nolimits}

\def\Ob{\mathop{\mathrm{Ob}}\nolimits}
\def\Hom{\mathop{\mathrm{Hom}}\nolimits}

\def\Span{\mathop{\mathrm{Span}}\nolimits}

\def\co{\colon\thinspace}
\newcommand{\uqenh}[1]{ (\bar U_q^{\ev})\,\hat  {}^{\;\hat  \otimes #1}}
\newcommand{\uqen}[1]{ (\bar U_q^{\ev})\,\tilde {}^{\;\tilde \otimes #1}}
\newcommand{\uqe}{\bar U_q^{\ev}}
\newcommand{\uqz}{\bar U_q^0}
\newcommand{\uqze}{\bar U_q^{\ev 0}}
\newcommand{\uq}{\bar U_q}

\newcommand{\uqzq}{U_{\mathbb{Z},q}}

\newcommand{\f}[1]{\tilde F^{({#1})}}
\newcommand{\e}[1]{\tilde E^{({#1})}}
\newcommand{\Z}{\mathbb{Z}[q,q^{-1}]}

\newcommand{\uqzx}{\langle \uqz \rangle }
\newcommand{\uqzex}{\langle \uqze \rangle}
\newcommand{\qintx}[1]{\langle \{ #1\}_q!\rangle }
\newcommand{\mux}{\langle \mu \rangle }
\newcommand{\Dx}{\langle D^{\pm 1} \rangle }
\newcommand{\Thetax}[1]{\langle \Theta _{#1}\rangle }
\newcommand{\Deltax}{\langle \Delta \rangle }
\newcommand{\Kx}{\uqzex}
\newcommand{\ex}[1]{\langle \e{#1}\rangle }
\newcommand{\fx}[1]{\langle \f{#1}\rangle }
\newcommand{\xx}[2]{\langle \tilde X_{#2}^{({#1})}\rangle }
\newcommand{\Yx}{\langle \dot Y\rangle }
\newcommand{\bYx}{\langle \bar Y\rangle }
\newcommand{\adx}{\langle \ad\rangle}
\newcommand{\sadx}{\langle \sad \rangle }
\newcommand{\Rx}[1]{\langle \alpha ^{\pm}_{#1}\otimes \beta ^{\pm}_{#1}\rangle }
\newcommand{\varepsilonx}{\langle \varepsilon  \rangle }
\newcommand{\etax}{\langle \eta   \rangle }
\newcommand{\Ex}[1]{\langle \tilde E^{#1}\rangle}
\newcommand{\Fx}[1]{\langle \tilde F^{#1}\rangle }
\begin{document}
\title{On the universal $sl_2$ invariant of boundary bottom tangles}
\author{Sakie Suzuki\thanks{Research Institute for Mathematical Sciences, Kyoto
University, Kyoto, 606-8502, Japan. E-mail address: \texttt{sakie@kurims.kyoto-u.ac.jp}} }

\maketitle
\begin{center}
\textbf{Abstract}
\end{center}
The universal $sl_2$ invariant of bottom tangles has a universality  property for the colored Jones polynomial of links.
Habiro conjectured  that the universal $sl_2$ invariant of boundary bottom tangles takes values in  certain subalgebras of the completed tensor powers of 
the quantized enveloping algebra  $U_h(sl_2)$ of the Lie algebra $sl_2$.
In the present paper, we prove an improved version of Habiro's conjecture.
As an application, we prove a  divisibility property of  the colored Jones polynomial of boundary links.
\section{Introduction}
In the 80's, Jones \cite{Jo} constructed a  polynomial invariant of links.
After that, Reshetikhin and Turaev \cite{Re} defined  an invariant
of  framed links whose components are colored by finite dimensional representations of  a ribbon Hopf algebra.
The \textit{colored Jones polynomial} is the  Reshetikhin-Turaev invariant  of links  whose components are colored by finite dimensional  representations of 
the quantized enveloping algebra  $U_h(sl_2)$.

The \textit{universal invariant} associated with a ribbon Hopf algebra   is  an invariant of framed links and tangles which are not colored by any representations,
see Hennings \cite{He}, Lawrence \cite{R1,R2},  Reshetikhin \cite{Re}, Ohtsuki \cite{O},  Kauffman \cite{Ka}, and Kauffman and Radford \cite{KR}.
The universal  invariant has the universality property for the Reshetikhin-Turaev invariant.
By the \textit{universal $sl_2$ invariant}, we mean the universal   invariant associated with $U_h(sl_2)$.
In particular, one can  obtain the colored Jones polynomial from the  universal  $sl_2$ invariant.

A \textit{bottom tangle}  is a tangle  consisting of arc components in a cube such that each boundary point is on the bottom line, and the two boundary points of each component  are adjacent to each other,  see Figure \ref{fig:cl} (a) for example.
We can define the closure link  of a bottom tangle, see  Figure \ref{fig:cl} (b).
For each link $L$, there is  a bottom tangle whose closure is  $L$.
In \cite{H1}, Habiro studied the  universal invariant of  bottom tangles  associated with a ribbon Hopf algebra, and  in  \cite{H2}, he  studied the universal  $sl_2$ invariant in detail.
\begin{figure}
\centering
\includegraphics[width=8cm,clip]{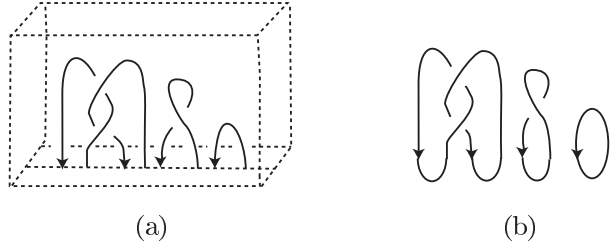}
\caption{(a) A bottom tangle $T$ (b) The closure link of $T$ }\label{fig:cl}
\end{figure}

The universal $sl_2$ invariant of  $n$-component bottom tangles takes values in the completed $n$-fold tensor power $U_h(sl_2)^{\hat \otimes n}$ of $U_h(sl_2)$.
By using  bottom tangles, we can restate the universality of the universal  $sl_2$ invariant:
the colored Jones polynomial  of a link $L$ is obtained from the  universal $sl_2$ invariant of a bottom tangle  whose closure is  $L$,  by taking the  quantum traces associated with the representations attached to the components of links (cf. \cite{H1}).

We are interested in relationships between the algebraic properties of the colored Jones polynomial and  the universal $sl_2$ invariant
and the topological properties of links and  bottom tangles.

Eisermann \cite{Ei} proved that the Jones polynomial of  an $n$-component
ribbon link is  divisible by the Jones polynomial of the  $n$-component unlink.
This result is generalized to links which are ribbon concordant to boundary links by Habiro  \cite{H4}.
Habiro \cite{H2} proved that the universal $sl_2$ invariant  of   $n$-component,
 algebraically-split, $0$-framed bottom tangles takes values in certain small subalgebras of the completed tensor powers of $U_h(sl_2)$,
and  gave a divisibility property of the colored Jones polynomial of algebraically-split, $0$-framed links.

In \cite{sakie}, the present author  proved  an improvement of  Habiro's result for algebraically-split, $0$-framed bottom tangles, in the special case of \textit{ribbon bottom tangles} and  ribbon links.

In the present paper, we study  the universal $sl_2$ invariant of \textit{ boundary bottom tangles}.
A bottom tangle is called \textit{boundary}  if its components admit mutually disjoint Seifert surfaces, 
see Figure \ref{fig:bo} for example. We can obtain each boundary link from a boundary bottom tangle by closing.
Habiro \cite{H2} conjectured  that the universal $sl_2$ invariant of boundary bottom tangles takes values in  certain subalgebras of the completed tensor powers of $U_h(sl_2)$.
We  prove an improved version of Habiro's conjecture (Theorem \ref{1}).
\begin{figure}
\centering
\includegraphics[width=4cm,clip]{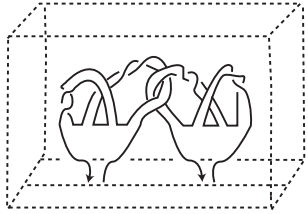}
\caption{A boundary bottom tangle }\label{fig:bo}
\end{figure}

\subsection{Main result}
The quantized enveloping algebra  $U_h=U_h(sl_2)$  is an $h$-adically completed $\mathbb{Q}[[h]]$-algebra (see Section \ref{preuni} for the details).
We set $q=\exp h$. 

Habiro \cite{H2}  proved that the universal $sl_2$ invariant $J_T$ of  an $n$-component,
 algebraically-split, $0$-framed bottom tangle $T$  is contained in the $\mathbb{Z}[q,q^{-1}]$-subalgebra 
$(\tilde {\mathcal{U}}_q^{ev})^{\tilde \otimes n}$ of $U_h^{\hat \otimes n}$.
In \cite{sakie}, we defined another $\Z$-subalgebra  $\uqenh{n}\subset (\tilde {\mathcal{U}}_q^{ev})^{\tilde \otimes n}$, and prove the following theorem.
(See Section \ref{preunis}   for the definition of $\uqe$,
and see Sections \ref{Comp0}--\ref{Comp} for the definition of the  completion $\uqenh{n}$ of $(\uqe)^{\otimes n}$.)
\begin{theorem}[\cite{sakie}]
Let  $T$ be an $n$-component ribbon bottom tangle with $0$-framing.
Then we have $J_T\in  \uqenh{n}$.
\end{theorem}
The main result of the present  paper is the following.
\begin{theorem} \label{1}
Let $T$ be an $n$-component boundary bottom tangle with $0$-framing.
Then we have $J_T\in  \uqenh{n}$.
\end{theorem}
\begin{remark}
Habiro \cite[Conjecture 8.9]{H2} conjectured Theorem \ref{1} with  $\uqenh{n}$ replaced with the $\Z$-subalgebra $\uqen{n}$,
which includes $\uqenh{n}$.
We do not know whether the inclusion $\uqenh{n}\subset \uqen{n}$ is proper or not.  The definition of  our algebra $\uqenh{n}$ appears to be  more natural.
\end{remark} 
Since every $1$-component bottom tangle is boundary, Theorem \ref{1}  for $n=1$ gives a possible improvement  of the following theorem.
\begin{theorem}[Habiro]\label{habiroo}
Let $T$ be an $1$-component bottom tangle with $0$-framing.
Then we have $J_T\in (\bar U_q^{\ev})\,\tilde {}$.
\end{theorem}
Theorem \ref{habiroo} follows from \cite[Theorem 4.1]{H2} and the equalities
\begin{align*}
\inv (\tilde {\mathcal{U}}_q^{ev})=Z(\tilde {\mathcal{U}}_q^{ev})=Z((\bar U_q^{\ev})\,\tilde {}\,),
\end{align*}
which is implicit in  \cite[Section 9]{H3}. Here, for a subset $X\subset U_h$, we denote by $\inv (X)$  the invariant part of $X$, and  by $Z(X)$ the center of  $X$.

If we use the one-to-one correspondence described in \cite[Section 13]{H1}  between the set of bottom tangles and the set of 
string links, then we can define the Milnor $\bar \mu $ invariants \cite{M1,M2} of a bottom tangle as that of the corresponding string link.
See \cite{Hab} for the Milnor $\bar \mu$ invariants of string links.
In fact, all Milnor $\bar \mu $ invariants vanish both for ribbon bottom tangles  and boundary bottom tangles.
It is natural to expect the following conjecture.
\begin{conjecture}\label{mil}
If $T$ be an $n$-component  bottom tangle with $0$-framing
with vanishing all Milnor $\bar \mu $ invariants, then we have
$J_T\in \uqenh{n}.$
\end{conjecture}
The converse of Conjecture \ref{mil} is also open.
\subsection{Application to the colored Jones polynomial}
We give an application (Theorem \ref{4}) of Theorem \ref{1} to the colored Jones polynomial of boundary links.
This result is parallel to the result for ribbon  links \cite{sakie}.

We use the following $q$-integer  notations:
\begin{align*}
&\{i\}_q = q^i-1,\quad  \{i\}_{q,n} = \{i\}_q\{i-1\}_q\cdots \{i-n+1\}_q,\quad  \{n\}_q! = \{n\}_{q,n},
\\
&[i]_q = \{i\}_q/\{1\}_q,\quad  [n]_q! = [n]_q[n-1]_q\cdots [1]_q, \quad \begin{bmatrix} i \\ n \end{bmatrix} _q  = \{i\}_{q,n}/\{n\}_q!,
\end{align*}
for $i\in \mathbb{Z}, n\geq 0$.

For $m\geq 1$, let  $V_m$ denote the $m$-dimensional irreducible representation of $U_h$.
Let $\mathcal{R}$  denote the representation ring  of $U_h$ over  $\mathbb{Q}(q^{\frac{1}{2}})$, i.e.,
$\mathcal{R}$ is the $\mathbb{Q}(q^{\frac{1}{2}})$-algebra 
\begin{align*}
\mathcal{R}= \Span _{\mathbb{Q}(q^{\frac{1}{2}})}\{V_m \  | \ m\geq 1\}
\end{align*}
with the multiplication induced by the tensor product.
It is well known that $\mathcal{R}=\mathbb{Q}(q^{\frac{1}{2}})[V_2].$ 

For $l\geq 0$, set
\begin{align*}
P_l&=\prod _{i=0}^{l-1}(V_2-q^{i+\frac{1}{2}}-q^{-i-\frac{1}{2}}) \in \mathcal{R},
\\
\tilde P'_l&=\frac{q^{\frac{1}{2}l}}{\{l\}_q!}P_l \in \mathcal{R},
\end{align*}
which are used  in \cite{H2} to construct the unified Witten-Reshetikhin-Turaev invariants for integral homology spheres.
We denote by  $J_{L; \tilde P'_{l_1},\ldots , \tilde P'_{l_n}}$ the colored Jones polynomial of $L$ with $i$th component $L_i$ colored by  $\tilde P'_{l_i}$.
Habiro  proved that Theorem \ref{1} implied the following result.
For $l\geq 0$, let $I_{l}$ denote the ideal in $\mathbb{Z}[q,q^{-1}]$ generated by  $\{l-k\}_q!\{k\}_q!$ for 
$k=0,\ldots, l$.
\begin{theorem}[{\cite[Conjecture 8.10]{H2}}]\label{4}
Let $L$ be an $n$-component boundary link with $0$-framing.
For $l_1,\ldots , l_n\geq 0$, we have
\begin{align*}
J_{L; \tilde P'_{l_1},\ldots , \tilde P'_{l_n}}\in \frac{\{ 2l_j+1\}_{q, l_j+1}}{\{1\} _q} I_{l_1}\cdots \hat I_{l_j}\cdots I_{l_n},
\end{align*} 
where $j$  is an integer such that   $l_j=\max\{l_i\}_{1\leq i\leq n}$, and   $\hat I_{l_j}$ denotes the omission of $I_{l_j}$.  
\end{theorem}
\begin{remark}
For $m\geq 1$, let $\Phi _m(q)=\prod _{d|m}(q^d-1)^{\mu (\frac{m}{d})}\in \mathbb{Z}[q]$ denote the $m$th cyclotomic polynomial, where $\prod _{d|m}$ denotes the
product  over all the positive divisors $d$ of $m$, and  $\mu$ is the M\"obius function.
It is not difficult to prove that for $l\geq 0$, $I_{l}$ is contained in the principal ideal in $\mathbb{Z}[q]$ generated by  $\prod_m\Phi _m(q)^{f(l,m)}$ with $
f(l,m)=\max \{0, \big\lfloor \frac{l+1}{m}\big\rfloor -1 \}$, where, for $r\in \mathbb{Q}$, we denote by $\lfloor r \rfloor$  the largest integer smaller than  or  equal to $r$.
\end{remark}
Theorem \ref{4} is an improvement in the special case of  boundary links of the following result.
\begin{theorem}[Habiro {\cite[Theorem 8.2]{H2}}]
Let $L$ be an $n$-component, algebraically-split link with $0$-framing. For $l_1,\ldots , l_n\geq 0$, we have
\begin{align*}
J_{L; \tilde P'_{l_1},\ldots , \tilde P'_{l_n}}\in \frac{\{ 2l_j+1\}_{q, l_j+1}}{\{1\} _q}\mathbb{Z}[q,q^{-1}].
\end{align*} 
\end{theorem}
\subsection{Examples}\label{ex}
Let $T_B$ be the Borromean bottom tangle depicted in Figure  \ref{fig:brunnianex} (a),
whose closure is the Borromean rings.
Since we have $J_{T_B}\notin  \uqenh{3}$ (cf. \cite{sakie}), it follows from  Theorem \ref{1} that  the Borromean rings is neither boundary  nor  ribbon, as is well known.

More generally, for $n\geq 3,$ let   $M_n$  be   Milnor's $n$-component Brunnian link  depicted in Figure \ref{fig:brunnianex} (b).
Note that $M_3$ is the Borromean rings. 
Since there is a non-trivial  Milnor $\bar \mu $ invariant of $M_n$ of length $n$ (cf. \cite{M1}),
$M_n$ is neither  boundary nor  ribbon.
We can prove this fact also from Theorem \ref{4} and 
\begin{align*}
J_{M_n;  \tilde P_1',\ldots, \tilde P_1'}&=(-1)^{n-2}q^{-2n+4}\Phi _1(q)^{n-2} \Phi _2(q)^{n-2}\Phi _3(q)\Phi_4(q)^{n-3}
\\
&\notin \Phi _1(q)^{n}\Phi _2(q)\Phi _3(q)\Z,
\end{align*}
which we will prove in a forthcoming paper \cite{sakie2}.
\begin{figure}
\centering
\includegraphics[width=7.5cm,clip]{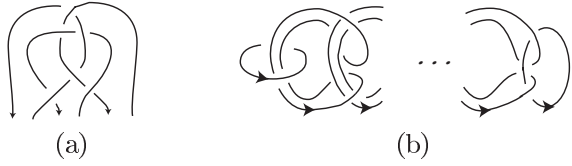}
\caption{(a) Borromean rings (b) Milnor's link $M_n$}\label{fig:brunnianex}
\end{figure}%

\subsection{Organization of paper}
The rest of the paper is organized as follows.
Section \ref{prel} contains preliminary results about bottom tangles, the quantized enveloping algebra $U_h$, and  the universal $sl_2$ invariant of bottom tangles.
In Section \ref{unib}, we recall from \cite{H2} Habiro's formula  for the universal $sl_2$ invariant of boundary bottom tangles, 
and then give a modification of his formula.
In Sections \ref{cat}, \ref{proof}, and \ref{completion}, we  prove Theorem \ref{1}.

\section{Preliminaries}\label{prel}
In this section, we recall basic things about bottom tangles, the universal enveloping algebra $U_h$, and  the universal $sl_2$ invariant of bottom tangles.
\subsection{Bottom tangles and boundary bottom tangles}\label{bottom}
A \textit{tangle} (cf. \cite{Kassel}) is the image of an embedding 
\begin{align*}
\Big(\coprod^m [0,1]\Big) \sqcup\Big(\coprod^n S^1\Big)\hookrightarrow [0,1]^3,
\end{align*}
with $m,n\geq 0,$
whose boundary  is on the two lines $[0,1]\times \{\frac{1}{2}\}\times \{0,1\}$ on the bottom and the top of the cube,
 see Figure \ref{fig:tanglek} (a) for example.
 We equip the image with both an orientation and  a framing.
Here, at each boundary point, the framing is fixed  on the lines $[0,1]\times \{\frac{1}{2}\}\times \{0,1\}$ as in Figure \ref{fig:tanglek} (b),
where the thin arrows represent the strands of the tangle, and the thick arrows represent the framing.
\begin{figure}
\centering
\includegraphics[width=9cm,clip]{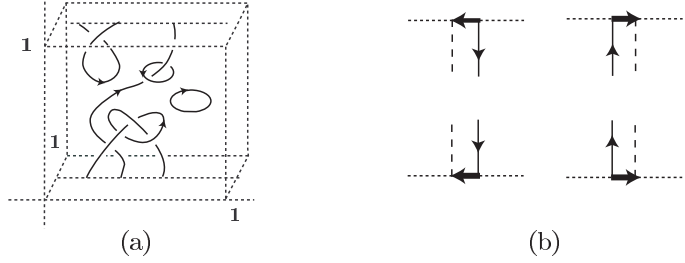}
\caption{(a) A tangle (b) The framing on the boundary}\label{fig:tanglek}
\end{figure}

A \textit{bottom tangle} (cf. \cite{H1,H2}) is a tangle consisting of arc components such that
 each boundary point is on the line $[0,1]\times \{\frac{1}{2}\}\times \{0\}$ on the bottom,
and  the two boundary points of each component  are adjacent to each other.
We give a preferred orientation of the tangle so that each component runs from its right boundary point to its left boundary point.
For example, see  Figure \ref{fig:closure} (a), where the dotted lines represent the framing. We draw a diagram of a bottom tangle in a rectangle assuming the blackboard framing, see Figure \ref{fig:closure} (b).

For each $n\geq 0$, let $BT_n$ denote the set of the ambient  isotopy classes, relative to boundary points,  of $n$-component
bottom tangles.

The \textit{closure link} $\cl(T)$ of a bottom tangle $T$ is defined as the link in $\mathbb{R}^3$ obtained from  $T$ by closing, see Figure \ref{fig:cl} again.
For each $n$-component link $L$, there is an $n$-component bottom tangle whose closure is  $L$.
For a bottom tangle, we can define its linking matrix as that of the closure link.

A \textit{Seifert surface} of knot $K$ is a compact, connected, orientable surface $F$ in $\mathbb{R}^3$ bounded by $K$.
An $n$-component link $L=L_1\cup\cdots \cup L_n$ is called  \textit{boundary }  if it has $n$ mutually  disjoint  Seifert surfaces $F_1,\ldots, F_n$ in $\mathbb{R}^3$ such  that $L_i$ bounds $F_i$ for $i=1,\ldots,n$.
\begin{figure}
\centering
\includegraphics[width=8cm,clip]{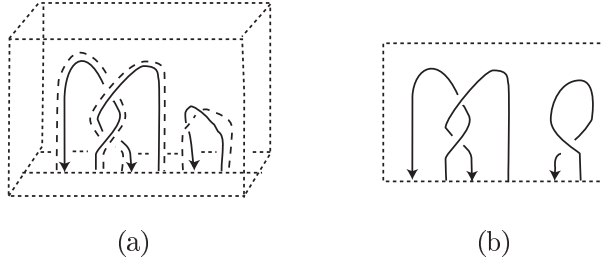}
\caption{ (a) A $3$-component bottom tangle $T$ (b) A diagram   of $T$  }\label{fig:closure}
\end{figure}%

For a $1$-component  bottom tangle $T\in BT_1$, there is a knot $K_T=T\cup \gamma \in [0,1]^3$, where
$\gamma $ is the line segment in the bottom $[0,1]\times  \{\frac{1}{2}\}\times \{0\}$ such that $\partial \gamma =\partial T$.
A \textit{Seifert surface} of a $1$-component bottom tangle $T$ is a Seifert surface of the knot $K_T$ contained  in $[0,1]^3$.
A bottom tangle $T=T_1\cup \cdots\cup T_n$ is called \textit{boundary } if it has  $n$ mutually  disjoint  Seifert surfaces $F_1,\ldots, F_n$ in $[0,1]^3$ such  that $K_{T_i}$ bounds $F_i$ for $i=1,\ldots,n$. For example,
see Figure \ref{fig:bo} again. 
Obviously, for each boundary link $L$, there is a boundary bottom tangle whose closure is  $L$.
\subsection{Quantized enveloping algebra $U_h$}\label{preuni}
We recall  the definition of the universal enveloping algebra $U_h(sl_2)$ of the Lie algebra $sl_2$, and  its ribbon Hopf algebra structure.
We follow the notations in \cite{H2}.

We denote by  $U_h=U_h(sl_2)$ the $h$-adically complete $\mathbb{Q}[[h]]$-algebra,
topologically generated by  $H, E,$ and $F$, defined by the relations
\begin{align*}
HE-EH=2E, \quad HF-FH=-2F, \quad EF-FE=\frac{K-K^{-1}}{q^{1/2}-q^{-1/2}},
\end{align*}
where we set 
\begin{align*}
q=\exp h,\quad K=q^{H/2}=\exp\frac{hH}{2}.
\end{align*}

We equip $U_h$  with the topological $\mathbb{Z}$-graded algebra structure such that  $\deg E=1$, $\deg F=-1$, and   $\deg H=0$.
For a homogeneous element $x$ of $U_h$, the degree of $x$ is denoted by $|x|$.

There is a   complete ribbon Hopf algebra  structure  on  $U_h$ as follows.
The comultiplication $\Delta \co U_h\rightarrow U_h\hat \otimes U_h$,
the counit $\varepsilon \co U_h\rightarrow \mathbb{Q}[[h]]$, and 
the antipode $S\co U_h\rightarrow U_h$ are given by
\begin{align*}
\Delta (H)&=H\otimes 1+1\otimes H, \quad  \varepsilon (H)=0, \quad S(H)=-H,
\\
\Delta (E)&=E\otimes 1+K\otimes E, \quad \varepsilon (E)=0, \quad  S(E)=-K^{-1}E,
\\
\Delta (F)&=F\otimes K^{-1}+1\otimes F, \quad  \varepsilon (F)=0, \quad  S(F)=-FK.
\end{align*}

Set 
\begin{align}
&D=q^{\frac{1}{4}H\otimes H} =\exp \big(\frac{h}{4}H\otimes H\big)\in U_h^{\hat {\otimes }2},\label{defd}
\\
&\tilde {F}^{(n)}=F^nK^n/[n]_q! \in U_h,\label{deff}
\\
&e=(q^{1/2}-q^{-1/2})E \in U_h, \label{defe}
\end{align}
for $n\geq 0$. The universal $R$-matrix and its inverse  $R^{\pm 1}\in U_h\hat \otimes U_h$  are given by
\begin{align*}
R =&D\sum _{n\geq 0}q^{\frac{1}{2}n(n-1)}\tilde {F}^{(n)}K^{-n}\otimes e^n,
\\
R^{-1} =&D^{-1}\sum_{n\geq 0}(-1)^{n}\tilde {F}^{(n)}\otimes K^{-n}e^n.
\end{align*} 
We have $R^{\pm 1}=\sum_{n\geq 0}\alpha ^{\pm}_n \otimes \beta^{\pm} _n $, where for $n\geq 0,$ we set  formally
\begin{align*}
\alpha _n \otimes \beta _n(&=\alpha^+ _n \otimes \beta^+ _n)=D\Big(q^{\frac{1}{2}n(n-1)}\tilde {F}^{(n)}K^{-n}\otimes e^n\Big),
\\
\alpha _n^- \otimes \beta _n^-&=D^{-1}\Big((-1)^{n}\tilde {F}^{(n)}\otimes K^{-n}e^n\Big).
\end{align*}
Note that the right hand sides are infinite sums of tensors of the form as $x\otimes y$ with $x,y\in U_h$.
We denote them by $\alpha ^{\pm}_n \otimes \beta^{\pm} _n$ for simplicity.

The ribbon element  and its inverse $r^{\pm 1}\in U_h$ are given by
\begin{align*}
r=\sum_{n\geq 0} \alpha _n^-K^{-1}\beta _n^-=\sum_{n\geq 0} \beta _n^- K\alpha _n^-, \quad  r^{-1}=\sum _{n\geq 0}\alpha _n K\beta _n =\sum_{n\geq 0} \beta _nK^{-1}\alpha _n.
\end{align*}

We use a notation $D=\sum D'\otimes D ''$.
We use the following formulas.
\begin{align}
&\sum D''\otimes D'=D,\label{d1}
\\
&(\Delta \otimes 1)D=D_{23}D_{13}, \quad (1 \otimes \Delta )D=D_{13}D_{12},\label{d2}
\\
&(\varepsilon \otimes 1)(D)=1=(1\otimes \varepsilon )(D),\label{d3}
\\
&(1\otimes S)D=D^{-1}=(S\otimes 1)D,\label{d4}
\\
&D(1\otimes x)=(K^{|x|}\otimes x)D, \quad D(x\otimes 1)=(x\otimes K^{ |x|})D, \label{exD}
\end{align}
where $D_{13}=\sum D'\otimes 1\otimes D''$, $D_{23}=1\otimes D$, $D_{12}=D\otimes 1$, and  $x\in U_h$  homogeneous. 
\subsection{Subalgebras of $U_h$}\label{preunis}
In this section, we recall from \cite{H2} subalgebras
$\uqzq, \uq$ and  $\uqe$  of $U_h$.
Recall from (\ref{deff}) and  (\ref{defe}) the definitions of $\f{n}\in U_h$ and $e\in U_h$, respectively.
Similarly, set 
\begin{align*}
 &\tilde E^{(n)}=(q^{-1/2}E)^n/[n]_q! \in U_h,
\\
 &f=(q-1)FK \in U_h,
\end{align*}
for $n\geq 0$.

Let $U_{\mathbb{Z}, q}$ denote the $\mathbb{Z}[q,q^{-1}]$-subalgebra of $U_h$ generated by
$K,K^{-1}, \tilde E^{(n)}$, and  $\tilde F^{(n)}$ for $n\geq 1$.

Let $\bar {U}_q$ denote the $\mathbb{Z}[q,q^{-1}]$-subalgebra of $U_{\mathbb{Z},q}$ generated by
 $K,K^{-1},e$ and $f$. Let $\bar {U}_q^{\ev}$ be the $\mathbb{Z}[q,q^{-1}]$-subalgebra of $\uq$ generated by
 $K^2,K^{-2},e$ and $f$.

\begin{remark}
Set $[i]=\frac{q^{i/2}-q^{-i/2}}{q^{1/2}-q^{-1/2}}$ for $i\in \mathbb{Z}$ 
and $[n]!=[n]\cdots[1]$ for $n\geq 0$.
Let $U_{\mathbb{Z}}$ be the $\mathbb{Z}[q^{1/2},q^{-1/2}]$-subalgebra of $U_h$ generated by $K,K^{-1},$ 
 $E^{(n)}=E^n/[n]!$, and  $F^{(n)}=F^n/[n]!$ for $n\geq 1$ (Lusztig's  integral form, cf. \cite{L}).  
 We have
 \begin{align*}
U_{\mathbb{Z}}=U_{\mathbb{Z}, q}\otimes _{\mathbb{Z}[q,q^{-1}]}\mathbb{Z}[q^{1/2},q^{-1/2}].
\end{align*}
Let $\bar U$ denote   the $\mathbb{Z}[q^{1/2},q^{-1/2}]$-subalgebra of $U_h$ 
generated by  $K,K^{-1}$, $(q^{1/2}-q^{-1/2})E$, and $(q^{1/2}-q^{-1/2})F$ (cf. \cite{De}).
We have
\begin{align*}
\bar U=\bar U_{ q}\otimes _{\mathbb{Z}[q,q^{-1}]}\mathbb{Z}[q^{1/2},q^{-1/2}].
\end{align*}
\end{remark}
There is a Hopf $\mathbb{Z}[q,q^{-1}]$-algebra structure  on $U_{\mathbb{Z}, q}$  inherited from $U_h$ (cf. \cite{L, sakie}). We have
\begin{align}
\Delta (\tilde E^{(m)})&=\sum_{j=0}^m
\tilde E^{(m-j)}K^j\otimes \tilde E^{(j)}, \label{De1}
\\
\Delta (\tilde {F}^{(m)})&=\sum_{j=0}^m 
\tilde {F}^{(m-j)}K^j\otimes \tilde {F}^{(j)}, \label{De2}
\\
S^{\pm 1}(\tilde E^{(m)})&=(-1)^mq^{\frac{1}{2}m(m\mp 1)}K^{-m}\tilde E^{(m)},
\\
S^{\pm 1}(\tilde {F}^{(m)})&=(-1)^mq^{-\frac{1}{2}m(m\mp 1)}K^{-m}\tilde {F}^{(m)},
\end{align}
for $i\in \mathbb{Z},  m\geq 0$.
Similarly, there is a Hopf $\mathbb{Z}[q,q^{-1}]$-algebra structure on $\uq$ inherited from $U_h$ (cf. \cite{De, H2}).

Let $U_h^{0}$ denote the Cartan part of $U_h$, i.e.,  the subalgebra of $U_h$ topologically generated by
 $H$.
Let $\bar {U}_q^{0}$ denote the $\mathbb{Z}[q,q^{-1}]$-subalgebra of $\uq$ generated by
 $K$ and $K^{-1}$. Let $\uqze$ be the $\mathbb{Z}[q,q^{-1}]$-subalgebra of $\uq$ generated by
 $K^2$ and $K^{-2}$. We have
\begin{align*}
\bar {U}_q^{0}=\uq \cap U_h^0, \quad \uqze=\uqe \cap U_h^0.
\end{align*}
\subsection{Adjoint action}\label{ad}
In what follows, we use the following notations.
For $m\geq 0$, let  $\Delta ^{[m]}\co U_h\rightarrow U_h^{\hat \otimes m}$
denote the $m$-output comultiplication defined by  $\Delta  ^{[0]}=\varepsilon, \Delta ^{[1]}=\id_{U_h},$ and 
\begin{align*}
\Delta ^{[m]}=(\Delta\otimes\id_{U_h}^{\otimes m-2  })\circ \Delta ^{[m-1]},
\end{align*}
for $m\geq 2$. For $x\in U_h,$ $m\geq 1$, we write  
\begin{align*}
\Delta ^{[m]}(x)=\sum x_{(1)}\otimes \cdots \otimes x_{(m)}.
\end{align*}
For  $m_1,\ldots , m_l\geq 0$, set
\begin{align}\label{Deg}
\Delta  ^{[m_1,\ldots, m_l]}=\Delta  ^{[m_1]}\otimes \cdots \otimes \Delta  ^{[m_l]} \co U_h^{\hat \otimes l}\rightarrow U_h^{\hat \otimes m_1+\cdots+m_l}.
\end{align}

We use  the left adjoint action  $\ad \co U_h\hat \otimes U_h\rightarrow  U_h$  defined  by
\begin{align}\label{Ad}
\ad (x\otimes y)=x\triangleright y:=\sum x_{(1)}yS(x_{(2)}), 
\end{align}
for $x,y\in U_h$. 
We use the following proposition.
\begin{proposition}[{\cite[Proposition 3.2]{sakie}}] \label{Habi}
We have
\begin{align*}
 U_{\mathbb{Z}, q}\triangleright \uqe \subset \uqe,
 \quad
 U_{\mathbb{Z}, q}\triangleright K\uqe \subset K\uqe.
\end{align*}
\end{proposition}

We  also use a right action $ {\sad }\co U_h\hat \otimes U_h\rightarrow  U_h$, which  is the  continuous $\mathbb{Q}[[h]]$-linear map  defined by
\begin{align*}
 {\sad}(y\otimes x)= y\triangleleft  x:&=\sum S^{-1}(x_{(2)})yx_{(1)},
 \\
 &=\sum S^{-1}(x)\triangleright y.
\end{align*}
for $x,y\in U_h$. 
Proposition \ref{Habi}  implies the following.
\begin{corollary}\label{Habii}
We have
\begin{align*}
\bar {U}_q^{\ev} \triangleleft  U_{\mathbb{Z}, q}\subset \bar {U}_q^{\ev},
\quad
K\bar {U}_q^{\ev} \triangleleft  U_{\mathbb{Z}, q} \subset K\uqe.
\end{align*}
\end{corollary}
\subsection{Universal $sl_2$ invariant of bottom tangles}\label{bottom inv}
For  an $n$-component bottom tangle $T=T_1\cup \cdots \cup T_n\in BT_n$, we define the universal $sl_2 $ invariant $J_T\in U_h^{\hat {\otimes }n}$  as follows (\cite{O, H1}).

We choose and fix a diagram of $T$ obtained from the copies of the fundamental tangles  depicted in Figure \ref{fig:fundamental},
by pasting horizontally and vertically. For example, for the bottom tangle $B$ depicted  in Figure \ref{fig:base} (a), we can take a diagram depicted in Figure \ref{fig:base} (b).
 We denote by $C(T)$ the set of the crossings of the diagram.
We call a map 
\begin{align*}
s\co C(T) \ \ \rightarrow \ \ \{0,1,2,\ldots\}
\end{align*}
a \textit{state}. We denote by  $\mathcal{S}(T)$ the set of states of the diagram.
\begin{figure}
\centering
\includegraphics[width=9cm,clip]{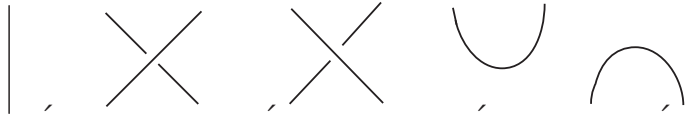}
\caption{Fundamental tangles, where the orientations of the strands are arbitrary
 }\label{fig:fundamental}
\end{figure}%
\begin{figure}
\centering
\includegraphics[width=12cm,clip]{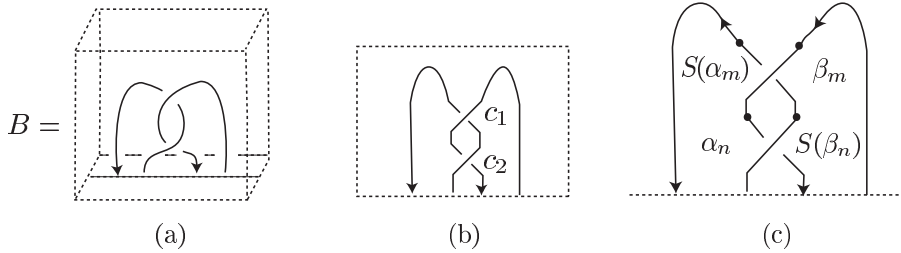}
\caption{(a) A bottom tangle $B\in BT_2$ (b) A  diagram  of $B$ (c) The labels which are put on the diagram of $B$
} \label{fig:base}
\end{figure}

Given a state $s\in \mathcal{S}(T)$, we attach labels  on the copies of the fundamental tangles in the diagram
following the rule described in Figure \ref{fig:cross}, where  ``$S'$'' should be replaced with $\id$ if 
the string is oriented downward, and with $S$ otherwise.
For example, for a state $t\in \mathcal{S}(B)$, we put  labels on the diagram of $B$ as in Figure  \ref{fig:base} (c), where we set $m=t(c_1)$ and $n=t(c_2)$
for the  upper and the lower crossings $c_1$ and $c_2$, respectively.
\begin{figure}
\centering
\includegraphics[width=12cm,clip]{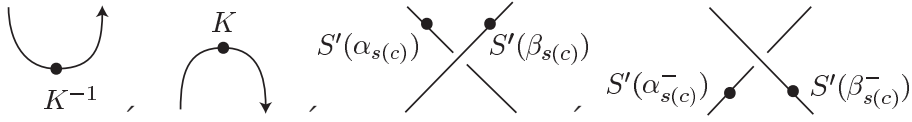}
\caption{How to place labels  on the fundamental tangles}\label{fig:cross}
\end{figure}

We define an element $J_{T,s}\in U_h^{\hat \otimes n}$ as follows. 
The $i$th tensorand of $J_{T,s}$ is defined to be the product 
of the labels put on the component corresponding to $T_i$, where the labels are read off along $T_i$
reversing the orientation, and   written from left to  right. 
We identify the labels $S'(\alpha^{\pm} _{i})$ and  $S'(\beta^{\pm} _{i})$
with the first and the second tensorands, respectively, of the element $S'(\alpha^{\pm} _{i})\otimes S'(\beta^{\pm} _{i})\in U_h^{\hat\otimes  2}$.
Also we identify the label $K^{\pm 1}$ with the element  $K^{\pm 1}\in U_h$.
Then, $J_{T,s}$ is a well-defined element in $U_h^{\hat \otimes n}$. 
For example,  we have
\begin{align*}
J_{B,t}&=S(\alpha  _m)S(\beta  _n)\otimes \alpha  _n\beta _m
\\
&=\sum q^{\frac{1}{2}m(m-1)} q^{\frac{1}{2}n(n-1)}S(D'_1\tilde {F}^{(m)}K^{-m})S(D''_2e^n)\otimes D'_2\tilde {F}^{(n)}K^{-n}D_1''e^m
\\
&=(-1)^{m+n}q^{-n+2mn}D^{-2}(\tilde F^{(m)}K^{-2n}e^n\otimes \tilde F^{(n)}K^{-2m}e^m)\in U_h^{\hat \otimes 2},
\end{align*}
where $D=\sum D'_1\otimes D''_1=\sum D'_2\otimes D''_2.$
Note that $J_{T,s}$  depends on the choice of the diagram.

Set 
\begin{align*}
J_T=\sum_{s\in \mathcal{S}(T)}J_{T, s}.
\end{align*}
For example, we have
\begin{align*}
J_B&=\sum_{t\in \mathcal{S}(B)}J_{B,t}=\sum_{m,n\geq 0} (-1)^{m+n}q^{-n+2mn}D^{-2}(\tilde F^{(m)}K^{-2n}e^n\otimes \tilde F^{(n)}K^{-2m}e^m).
\end{align*}
As is well known \cite{O},  $J_T$ does not depend on the choice of the diagram, and defines an isotopy invariant of bottom tangles.

\section{Universal invariant of boundary bottom tangles}\label{unib}
In this section, we recall  Habiro's formulas for boundary bottom tangles  at the topological level (Proposition \ref{yyuni}), and 
at the algebraic level on the universal $sl_2$ invariant (Proposition \ref{y3uni}).
Then, we modify these formulas into a form more convenient for our purpose.
After that, we study the commutator maps of $U_h$.
In the last section, we give an outline of the proof of Theorem \ref{1}.

In what follows, we use the following notations.
Let $\eta \co \mathbb{Q}[[h]]\rightarrow U_h$ be the unit morphism of $U_h$ and $\mu \co U_h^{\hat \otimes 2}\rightarrow U_h$  the multiplication of $U_h$. 
For $g\geq 0$, let $
\mu  ^{[g]}\co U_h^{\hat \otimes g}\rightarrow U_h$ denote the $g$-input multiplication defined by  $\mu ^{[0]}=\eta$, $\mu^{[1]}=\id_{U_h},$ and 
\begin{align*}
\mu ^{[g]}=\mu ^{[g-1]}\circ (\mu \otimes \id_{U_h}^{\otimes g-2}),
\end{align*} 
for $g\geq  2.$
For $g_1,\ldots , g_n\geq 0$, set
\begin{align}\label{mug}
\mu ^{[g_1,\ldots,g_n]}=\mu ^{[g_1]}\otimes \cdots \otimes \mu ^{[g_n]} \co U_h^{\hat \otimes  g_1+\cdots+g_n}\rightarrow U_h^{\hat \otimes n}.
\end{align}

\subsection{Habiro's formula (topological level)}
Let $T=T_1\cup \cdots\cup T_n\in BT_n$ be a  boundary bottom tangle  and $F_1,\ldots,F_n$ mutually disjoint Seifert surfaces
such that $\partial F_i=K_{T_i}$ for $i=1,\ldots, n.$
We can arrange the surfaces $F_1,\ldots ,F_n$ as depicted in  Figure \ref{fig:seifert},
where $\mathrm{Double}(T')$ is the tangle obtained from a bottom tangle $T'\in BT_{2g}$, where $g=g_1+\cdots+g_n$ with  $g_i=\mathrm{genus}(F_i)$,
 by first duplicating and then reversing the orientation of the inner component of each pair of duplicated components. 
 
The above arrangement of the  Seifert surfaces  implies the  following result, which appears in the proof of \cite[Theorem 9.9]{H1}.
\begin{figure}
\centering
\includegraphics[width=10cm,clip]{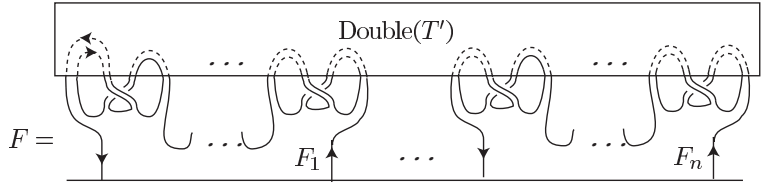}
\caption{How to arrange  Seifert surfaces }\label{fig:seifert}
\end{figure}
\begin{proposition}\label{yyuni}
For a bottom tangle $T\in BT_n$, the following conditions are equivalent.
\begin{itemize}
\item[\rm{(1)}]
$T$ is a boundary bottom tangle.
\item[\rm{(2)}]
There is  a bottom tangle  $T'\in BT_{2g},$ $g\geq 0$, and integers  $ g_1,\ldots , g_n\geq 0$ satisfying $g_1+\cdots+g_n=g$,  such that
\begin{align}
T=\mu _b^{[g_1,\ldots,g_n]}Y_b^{\otimes g}(T'),\label{date}
\end{align}
where 
$Y_b^{\otimes g}\co BT_{2g}\rightarrow BT_{g}$ and $
 \mu _b^{[g_1,\ldots,g_n]}\co BT_{g}\rightarrow BT_n $ are defined as depicted in Figure \ref{fig:boundary} (a) and (b), respectively.
\end{itemize}

\end{proposition}
\begin{figure}
\centering
\includegraphics[width=11.5cm,clip]{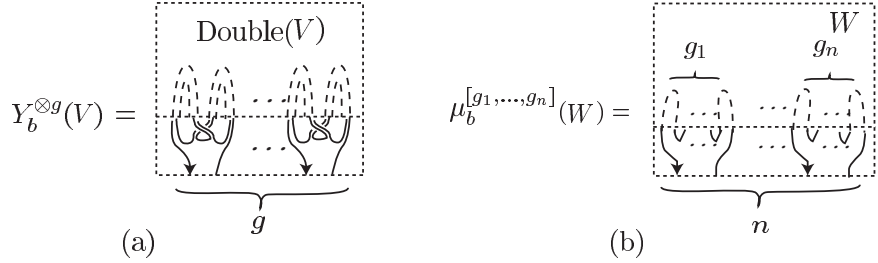}
\caption{(a) $Y_b^{\otimes g}(V)\in BT_{g}$ for $V\in BT_{2g}$  (b) $\mu _b^{[g_1,\ldots,g_n]}(W)\in BT_n$ for $W\in BT_g$ }\label{fig:boundary}
\end{figure}

\subsection{Habiro's formula (algebraic level)}
Recall from \cite[Proposition 9.7]{H1} the commutator morphism $Y_{\underline{H}}\co H\otimes H\rightarrow H$ for a ribbon Hopf algebra $H$.
In the present case  $H=U_h$, the morphism $Y_{\underline{U_h}}\co U_h\hat \otimes U_h\rightarrow U_h$ is the continuous $\mathbb{Q}[[h]]$-linear map  defined by
\begin{align*}
Y_{\underline{U_h}}( x\otimes y)&=\sum_{k\geq 0} x\triangleright \Big(\beta_k S\big((\alpha_k\triangleright y)_{(1)}\big)\Big)(\alpha_k\triangleright y)_{(2)}
\end{align*}
for $x,y\in U_h$. 
\begin{lemma}[Habiro \cite{H1}]\label{yuni}
For each bottom tangle $T\in BT_{2g}$,  $g\geq 0$, we have
\begin{align*}
&J_{Y_b^{\otimes g}(T)}=Y_{\underline{U_h}}^{\otimes g}(J_T).
\end{align*}
For each bottom tangle $T\in BT_{g_1+\cdots+g_n}$,  $g_1,\ldots,g_n\geq 0$, we have
\begin{align*}
&J_{ \mu _b^{[g_1,\ldots, g_n]}(T)}=\mu^{[g_1,\ldots, g_n]}(J_T). 
\end{align*}
\end{lemma}
Proposition \ref{yyuni} and Lemma \ref{yuni} imply  the following.
\begin{proposition}[Habiro \cite{H1}]\label{y3uni}
For a boundary bottom tangle $T\in BT_n$ and a bottom tangle $T'\in BT_{2g}$ satisfying (\ref{date}), we have
\begin{align*}
J_T=\mu^{[g_1,\ldots ,g_n]}(Y_{\underline{U_h}})^{\otimes g}(J_{T'}). 
 \end{align*}
\end{proposition}

\subsection{Modification of Habiro's formula (topological level) }
In this section, we modify Proposition \ref{yyuni}. 

Let $T\in BT_n$ be a boundary bottom tangle and $T'$ a  $2g$-component bottom tangle satisfying (\ref{date}).
We decompose the operation $Y_b^{\otimes g}$ into the two operations $\nu _b^{\otimes g}$ and $\bar Y_b^{\otimes g}$ as follows.
Let $\tilde T=\nu _b^{\otimes g}(T')$ be the $2g$-component (non-bottom) tangle  as depicted in Figure \ref{fig:boundary2} (a). 
Set $\bar Y_b^{\otimes g}(\tilde T)=Y_b^{\otimes g}(T')\in BT_g$, i.e., $\bar Y_b^{\otimes g}(\tilde T)$ is the bottom tangle  as depicted in Figure \ref{fig:boundary2} (b),
where $\mathrm{Double}(T')$ is defined in the  same way as that for bottom tangles.

We have 
\begin{figure}
\centering
\includegraphics[width=10cm,clip]{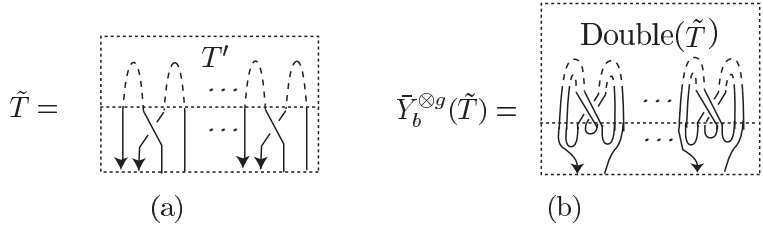}
\caption{(a) The tangle $\tilde T=\nu _b^{\otimes g}(T')$ (b)  The bottom tangle $\bar Y_b^{\otimes g}(\tilde T)\in BT_{g}$ }\label{fig:boundary2}
\end{figure}
\begin{align*}
T&=\mu _b^{[g_1,\ldots,g_n]}Y_b^{\otimes g}(T')
\\
&=\mu _b^{[g_1,\ldots,g_n]}(\bar Y_b^{\otimes g}\circ \nu _b^{\otimes g})(T')
\\
&=\mu _b^{[g_1,\ldots,g_n]}\bar Y_b^{\otimes g}\big(\nu _b^{\otimes g}(T')\big)
\\
&=\mu _b^{[g_1,\ldots,g_n]}\bar Y_b^{\otimes g}(\tilde T).
\end{align*}

Thus, we can modify Proposition \ref{yyuni} by replacing (2) with (2') as follows.
\begin{itemize}
\item[(2')]
There is  a $2g$-component tangle  $\tilde T=\nu _b^{\otimes g}(T')$ with $T'\in BT_{2g},$ $g\geq 0$, and integers  $ g_1,\ldots , g_n\geq 0$ satisfying $g_1+\cdots+g_n=g$,  such that
\begin{align*}
T=\mu _b^{[g_1,\ldots,g_n]}\bar Y_b^{\otimes g}(\tilde T).
\end{align*}
\end{itemize}
For   a boundary bottom tangle $T\in BT_n$, we call  $(\tilde T;g,g_1,\ldots,g_n)$  as in (2') a \textit{boundary data} for $T$. 
\subsection{Modification of Habiro's formula (algebraic level) }
Let $\bar Y\co U_h\hat \otimes U_h\rightarrow U_h$ be the continuous $\mathbb{Q}[[h]]$-linear map  defined by
\begin{align}
\bar Y(x\otimes y)&=\sum x_{(1)}KS(y_{(2)})KS(x_{(2)})y_{(1)}, \label{ytt}
\end{align}
for  $x,y\in U_h$. 

We modify Proposition \ref{y3uni} as follows.
\begin{proposition}\label{tu}
Let $T\in BT_n$ be a boundary bottom tangle  and $(\tilde T;g, g_1,\ldots,g_n)$ a boundary data for $T$.
We have
\begin{align*}
J_T=\mu^{[g_1,\ldots ,g_n]}\bar Y^{\otimes g}(J_{\tilde T}).
\end{align*}
Here, we can define the universal $sl_2$ invariant $J_{\tilde T}\in U_h^{\hat \otimes 2g}$ of the tangle $\tilde T$  in a similar way  to that of bottom tangles (cf. \cite{H1}).
\end{proposition}

Let $\nu\co U_h\hat \otimes U_h\rightarrow U_h\hat \otimes U_h$ be the continuous $\mathbb{Q}[[h]]$-linear map  defined by
 
\begin{align*}
\nu (x\otimes y)=\sum _{k\geq 0} x\beta _k\otimes \alpha _k y,
\end{align*}
for  $x,y\in U_h$.
We reduce Proposition \ref{tu} to the following lemma.
\begin{lemma}\label{nu}
We have \begin{align}
Y_{\underline{U_h}}=\bar Y\circ \nu.
\end{align}
\end{lemma}
\begin{proof}[Proof of Proposition \ref{tu} by assuming Lemma \ref{nu}.]
Let $T'\in BT_g$ the bottom tangle such that $\tilde T=\nu_b ^{\otimes g}(T')$.
We depict in Figure \ref{fig:skew}  the labels put on the new crossings $c_1,\ldots,c_g$ at the bottom of $\tilde T$ associated  a state $s\in \mathcal{S}(\tilde T)$. 
Since $(1\otimes S)(R^{-1})=R$, we have
\begin{align} \label{cd}
J_{\tilde T}=\nu ^{\otimes g}(J_{T'}).
\end{align}
By Proposition \ref{y3uni}, Lemma \ref{nu},  and (\ref{cd}),  we have
\begin {align*}
J_T&=\mu^{[g_1,\ldots ,g_n]}Y_{\underline{U_h}}^{\otimes g}(J_{T'})
\\
&=\mu^{[g_1,\ldots ,g_n]}(\bar Y\circ \nu)^{\otimes g}(J_{T'})
\\
&=\mu^{[g_1,\ldots ,g_n]}\bar Y^{\otimes g}\big(\nu ^{\otimes g}(J_{T'})\big)
\\
&=\mu^{[g_1,\ldots ,g_n]}\bar Y^{\otimes g}(J_{\tilde T}).
\end{align*}
\begin{figure}[!htb]
\centering
\includegraphics[width=10cm,clip]{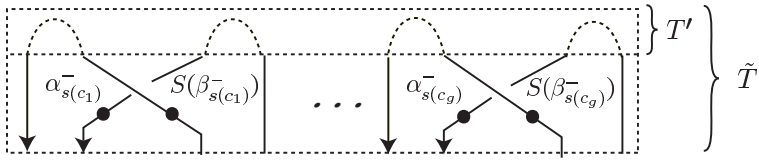}
\caption{The labels which are put on the crossings $c_1,\ldots,c_g$ at the bottom of $\tilde T$}\label{fig:skew}
\end{figure}
\end{proof}

\begin{proof}[Proof of Lemma \ref{nu}.]
For $x,y\in U_h,$ we have
\begin{align*}
(\bar Y&\circ \nu) (x\otimes y)
\\
&=\sum_{k\geq 0} (x\beta _k)_{(1)}KS((\alpha _ky)_{(2)})KS((x\beta _k)_{(2)})(\alpha _ky)_{(1)}
\\&=\sum_{k\geq 0} x_{(1)}\beta _{k(1)}KS(y_{(2)})S(\alpha _{k(2)})KS(\beta _{k(2)})S(x_{(2)})\alpha _{k(1)}y_{(1)}
\\&\stackrel{(\mathrm{i})}{=}\sum _{k_1,k_2,k_3,k_4\geq 0}x_{(1)}\beta _{k_1}\beta _{k_2}KS(y_{(2)})S(\alpha _{k_4}\alpha _{k_2})KS(\beta _{k_3}\beta _{k_4})S(x_{(2)})\alpha _{k_3}\alpha _{k_1}y_{(1)}
\\&=\sum_{k_1,k_2,k_3,k_4\geq 0} x_{(1)}\beta _{k_1}\beta _{k_2}KS(y_{(2)})S(\alpha _{k_2})S(\alpha _{k_4})KS(\beta _{k_4})S(\beta _{k_3})S(x_{(2)})\alpha _{k_3}\alpha _{k_1}y_{(1)}
\\&\stackrel{(\mathrm{ii})}{=}\sum_{k_1,k_2,k_3,k_4\geq 0} x_{(1)}\beta _{k_1}\beta _{k_2}KS(y_{(2)})S(\alpha _{k_2})\alpha _{k_4}K\beta _{k_4}S(\beta _{k_3})S(x_{(2)})\alpha _{k_3}\alpha _{k_1}y_{(1)}
\\&=\sum_{k_1,k_2,k_3\geq 0} x_{(1)}\beta _{k_1}\beta _{k_2}KS(y_{(2)})S(\alpha _{k_2})r^{-1}S(\beta _{k_3})S(x_{(2)})\alpha _{k_3}\alpha _{k_1}y_{(1)}
\\&\stackrel{(\mathrm{iii})}{=}\sum_{k_1,k_2,k_3,k_5,k_6\geq 0}  x_{(1)}\beta _{k_1}\beta _{k_2}KS(\beta _{k_5}^-y_{(1)}\beta _{k_6})S(\alpha _{k_2})r^{-1}S(\beta _{k_3})S(x_{(2)})\alpha _{k_3}\alpha _{k_1}\alpha _{k_5}^-y_{(2)}\alpha _{k_6}
\\&=\sum_{k_1,k_2,k_3,k_5,k_6\geq 0} x_{(1)}\beta _{k_1}\beta _{k_2}KS(\beta _{k_6})S(y_{(1)})S(\beta _{k_5}^-)S(\alpha _{k_2})r^{-1}S(\beta _{k_3})S(x_{(2)})\alpha _{k_3}\alpha _{k_1}\alpha _{k_5}^-y_{(2)}\alpha _{k_6}
\\&\stackrel{(\mathrm{iv})}{=}\sum_{k_1,k_2,k_3,k_5,k_6\geq 0} x_{(1)}\beta _{k_1}\beta _{k_2}S^{-1}(\beta _{k_6})KS(y_{(1)})S(\beta _{k_5}^-)S(\alpha _{k_2})r^{-1}S(\beta _{k_3})S(x_{(2)})\alpha _{k_3}\alpha _{k_1}\alpha _{k_5}^-y_{(2)}\alpha _{k_6}
\\&=\sum_{k_1,k_2,k_3,k_5,k_6\geq 0} x_{(1)}\beta _{k_1}\beta _{k_2}S^{-1}(\beta _{k_6})r^{-1}KS(y_{(1)})S(\beta _{k_5}^-)S(\alpha _{k_2})S(\beta _{k_3})S(x_{(2)})\alpha _{k_3}\alpha _{k_1}\alpha _{k_5}^-y_{(2)}\alpha _{k_6}
\\&\stackrel{(\mathrm{v})}{=}\sum_{k_1,k_2,k_3,k_4,k_5,k_6\geq 0} x_{(1)}\beta _{k_1}\beta _{k_2}\beta _{k_6}\beta _{k_4}K^{-1}\alpha _{k_4}KS(y_{(1)})S(\beta _{k_5}^-)S(\beta _{k_3}\alpha _{k_2})S(x_{(2)})\alpha _{k_3}\alpha _{k_1}\alpha _{k_5}^-y_{(2)}S(\alpha _{k_6})
\\&\stackrel{(\mathrm{vi})}{=}\sum_{k_1,k_2,k_3,k_4,k_5,k_6\geq 0} x_{(1)}\beta _{k_1}\beta _{k_2}\beta _{k_6}\beta _{k_4}S^2(\alpha _{k_4})S(y_{(1)})S(\beta _{k_5}^-)S(\beta _{k_3}\alpha _{k_2})S(x_{(2)})\alpha _{k_3}\alpha _{k_1}\alpha _{k_5}^-y_{(2)}S(\alpha _{k_6})
\\&\stackrel{(\mathrm{vii})}{=}\sum_{k,k_3,k_5\geq 0} x_{(1)}\beta _{k}S^2(\alpha _{k(4)})S(y_{(1)})S(\beta _{k_5}^-)S(\beta _{k_3}\alpha _{k(2)})S(x_{(2)})\alpha _{k_3}\alpha _{k(1)}\alpha _{k_5}^-y_{(2)}S(\alpha _{k(3)})
\\&\stackrel{(\mathrm{viii})}{=}\sum_{k,k_3,k_5\geq 0} x_{(1)}\beta _{k}S^2(\alpha _{k(4)})S(y_{(1)})S(\beta _{k_5}^-)S(\alpha _{k(1)}\beta _{k_3})S(x_{(2)})\alpha _{k(2)}\alpha _{k_3}\alpha _{k_5}^-y_{(2)}S(\alpha _{k(3)})
\\&=\sum_{k,k_3,k_5\geq 0} x_{(1)}\beta _{k}S^2(\alpha _{k(4)})S(y_{(1)})S(\beta _{k_3}\beta _{k_5}^-)S(\alpha _{k(1)})S(x_{(2)})\alpha _{k(2)}\alpha _{k_3}\alpha _{k_5}^-y_{(2)}S(\alpha _{k(3)})
\\&=\sum_{k\geq 0} x_{(1)}\beta _{k}S^2(\alpha _{k(4)})S(y_{(1)})S(\alpha _{k(1)})S(x_{(2)})\alpha _{k(2)}y_{(2)}S(\alpha _{k(3)})
\\&=\sum _{k\geq 0}x_{(1)}\beta _{k}S(\alpha _{k(1)}y_{(1)}S(\alpha _{k(4)}))S(x_{(2)})\alpha _{k(2)}y_{(2)}S(\alpha _{k(3)})
\\&\stackrel{(\mathrm{ix})}{=}\sum _{k\geq 0}x_{(1)}\beta_k S\Big(\big(\alpha_{k(1)}yS(\alpha _{k(2)})\big)_{(1)}\Big)S(x_{(2)})\big(\alpha_{k(1)}yS(\alpha _{k(2)})\big)_{(2)}=Y_{\underline{U_h}}( x\otimes y).
\end{align*}
Here,  recall that $r^{-1}=\sum _{k_4\geq 0} \alpha _{k_4}K\beta _{k_4}=\sum  _{k_4\geq 0}\beta _{k_4}K^{-1}\alpha _{k_4}$ is the inverse of  the ribbon element, which is central.
 
The  identity (i) follows from
\begin{align*}
&\sum _{k\geq 0} \alpha _{k(1)}\otimes \alpha _{k(2)}\otimes  \beta _{k(1)}\otimes  \beta_ {k(2)}=\sum _{k_1,k_2,k_3,k_4\geq 0}\alpha _{k_3}\alpha _{k_1}\otimes \alpha _{k_4}\alpha _{k_2}\otimes \beta _{k_1}\beta _{k_2}\otimes \beta _{k_3}\beta _{k_4}.
\end{align*}

(ii) and (v) follow from 
$(S\otimes S)R=R.$ 

(iii) and  (viii) follow from $R\big(\sum x_{(1)}\otimes x_{(2)}\big)=\big(\sum x_{(2)}\otimes x_{(1)}\big)R$ for $x\in U_h$.

(iv) and  (vi) follow from 
$KxK^{-1}=S^{-2}(x)$
for $x\in U_h$.

(vii) follows from
\begin{align*}
&\sum _{k_1,k_2,k_4,k_6\geq 0}\beta_{k_1} \beta_{k_2}\beta_{k_6} \beta_{k_4}  \otimes \alpha_{k_1} \otimes \alpha_{k_2} \otimes \alpha_{k_6} \otimes \alpha_{k_4}
=\sum  _{k\geq 0}\beta_k \otimes \alpha _{k(1)} \otimes \alpha _{k(2)} \otimes \alpha _{k(3)}\otimes  \alpha _{k(4)}.
\end{align*}

(ix) follows from 
$\sum S(x_{(1)})\otimes S(x_{(2)})=\sum S(x)_{(2)}\otimes S(x)_{(1)}$
for $x\in U_h.$
\end{proof}
\subsection{Commutator maps}
In this section, we study the commutator map $\bar Y$ of $U_h$.

Let $\dot Y\co U_h\hat \otimes U_h\rightarrow U_h$ be the  continuous $\mathbb{Q}[[h]]$-linear map defined  by
\begin{align*}
\dot Y( x\otimes y)&=\sum x_{(1)}S^{-1}(y_{(2)})S(x_{(2)})y_{(1)},
\end{align*}
for $x,y\in U_h$.
Note that 
\begin{align}
\dot Y( x\otimes y)&=\sum \Big( x\triangleright S^{-1}(y_{(2)})\Big) y_{(1)}\label{ad1}
\\
&=\sum x_{(1)}\Big(S(x_{(2)}) \triangleleft y\Big)
\\&=\sum x_{(1)}\Big(S^{-1}(y)\triangleright S(x_{(2)}) \Big).\label{ad2}
\end{align}
By the following lemma, we can study $\bar Y$ by using  $\dot Y$, $\triangleright $ and $\triangleleft$.
\begin{lemma}\label{ydyb}
For  $x, y\in U_h$, we have 
\begin{align*}
\bar Y(x\otimes y)&=\sum\dot Y( x_{(1)}\otimes y_{(2)})\big((x_{(2)}\triangleright K^2)\triangleleft y_{(1)}\big).
\end{align*}
\end{lemma}
\begin{proof}
We have 
\begin{align*}
\bar Y(x\otimes y)&=\sum x_{(1)}KS(y_{(2)})KS(x_{(2)})y_{(1)}
\\
&=
\sum x_{(1)}S^{-1}(y_{(2)})K^2S(x_{(2)})y_{(1)}
\\
&=\sum x_{(1)}S^{-1}(y_{(2)})S(x_{(2)})x_{(3)}K^2S(x_{(4)})y_{(1)}
\\
&=\sum x_{(1)}S^{-1}(y_{(4)})S(x_{(2)})y_{(3)}S^{-1}(y_{(2)})x_{(3)}K^2S(x_{(4)})y_{(1)}
\\
&=\sum\dot Y( x_{(1)}\otimes y_{(2)})\big((x_{(2)}\triangleright K^2)\triangleleft y_{(1)}\big),
\end{align*}
where the second identity follows from $Kz=S^{-2}(z)K$ for $z\in U_h$.
\end{proof}
The rest of this section is devoted to studying the map $\dot Y$.
\begin{lemma}\label{y2}
For $x,y,z\in U_h,$ we have
\begin{align}
\dot Y(xy\otimes z)&=\sum \big(x_{(1)}\triangleright \dot Y(y\otimes z_{(2)})\big)\dot Y(x_{(2)}\otimes z_{(1)}), \label{y21}
\\
\dot Y(x\otimes yz)&=\sum\dot Y( x_{(1)}\otimes z_{(2)})\big(\dot Y(x_{(2)}\otimes y)\triangleleft z_{(1)}\big).\label{y22}
\end{align}
\end{lemma}
\begin{proof}We have
\begin{align*}
\dot Y(xy\otimes z)&=\sum (xy)_{(1)}S^{-1}(z_{(2)})S\big((xy)_{(2)}\big)z_{(1)}
\\
&=\sum x_{(1)}y_{(1)}S^{-1}(z_{(2)})S(y_{(2)})S(x_{(2)})z_{(1)}
\\
&=\sum x_{(1)}y_{(1)}S^{-1}(z_{(4)})S(y_{(2)})z_{(3)}S^{-1}(z_{(2)})S(x_{(2)})z_{(1)}
\\
&=\sum x_{(1)}y_{(1)}S^{-1}(z_{(4)})S(y_{(2)})z_{(3)}S(x_{(2)})x_{(3)}S^{-1}(z_{(2)})S(x_{(4)})z_{(1)}
\\
&=\sum \big(x_{(1)}\triangleright \dot Y(y\otimes z_{(2)})\big)\dot Y(x_{(2)}\otimes z_{(1)}).
\end{align*}
Similarly, we have
\begin{align*}
\dot Y(x\otimes yz)&=\sum x_{(1)}S^{-1}\big((yz)_{(2)}\big)S(x_{(2)})(yz)_{(1)}
\\
&=\sum x_{(1)}S^{-1}(z_{(4)})S(x_{(2)})z_{(3)}S^{-1}(z_{(2)})x_{(3)}S^{-1}(y_{(2)})S(x_{(4)})y_{(1)}z_{(1)}
\\
&=\sum\dot Y( x_{(1)}\otimes z_{(2)})\big(\dot Y(x_{(2)}\otimes y)\triangleleft z_{(1)}\big).
\end{align*}
\end{proof}

\begin{lemma}\label{y3}
For $x,y\in U_h^0$, we have
\begin{align*}
\dot Y(x\otimes y)=\varepsilon (x)\varepsilon (y).
\end{align*}

\end{lemma}
\begin{proof}
It is enough to prove
\begin{align*}
\dot Y(H^m\otimes H^n)&=\delta _{m,0}\delta _{n,0},
\end{align*}
for $m,n\geq 0$.
By using the formula
\begin{align*}
\Delta (H^m)=\sum_{i=0}^{m} \begin{pmatrix}m\\i \end{pmatrix}H^i\otimes H^{m-i},
\end{align*}
for $m\geq 0,$
we have
\begin{align*}
\dot Y(H^m\otimes H^n)&=
\sum_{i=0}^{m}\sum_{j=0}^{n} \begin{pmatrix}m\\i \end{pmatrix} \begin{pmatrix}n\\j \end{pmatrix}H^{i}(-H)^j(-H)^{m-i}H^{n-j}
\\&=\Bigg(\sum_{i=0}^{m} (-1) ^{i}\begin{pmatrix}m\\i \end{pmatrix}\Bigg) \Bigg( \sum_{j=0}^{n} (-1)^j \begin{pmatrix}n\\j \end{pmatrix}\Bigg)(-1)^mH^{n+m}
\\
&=\delta _{n,0}\delta _{m,0}.
\end{align*}
\end{proof}

\begin{lemma}\label{qq}
We have
\begin{align}
&\dot Y(\uqzq \otimes \uq)\subset \uqe, \label{qq1}
\\
&\dot Y(\uq \otimes \uqzq)\subset \uqe.\label{qq2}
\end{align}
\end{lemma}
\begin{proof}
We prove (\ref{qq1}). Then (\ref{qq2}) is similar. Note that
\begin{align}
(1\otimes S^{\pm 1})\Delta  (\uq)\subset \bigoplus_{i=0,1} \Big(K^i\uqe\otimes K^i\uqe\Big),\label{bibi}
\end{align}
since we have
\begin{align}
&(1\otimes S^{\pm 1})\Delta  (K^{\pm 1})=K^{\pm 1}\otimes K^{\mp 1},\label{bibi1}
\\
&(1\otimes S^{\pm 1})\Delta  (\f{n})=\sum_{j=0}^n(-1)^nq^{-\frac{1}{2}n(n\mp 1)}\f{n-j}K^j\otimes K^{-j}\f{j},
\\
&(1\otimes S^{\pm 1})\Delta  (e)=e\otimes 1-K\otimes K^{-1}e.\label{bibi3}
\end{align}
Then, (\ref{ad1}) and  (\ref{bibi}) imply
\begin{align*}
\dot Y(\uqzq  \otimes \uq) &\subset  \sum_{i=0,1} \Big( \uqzq \triangleright K^i\uqe \Big) K^i\uqe.
\end{align*} 
By Proposition \ref{Habi}, we have
\begin{align*} 
\sum_{i=0,1} \Big( \uqzq \triangleright K^i\uqe \Big) K^i\uqe &\subset \sum_{i=0,1} \big(K^i\uqe\big)\cdot \big(K^i\uqe\big)\subset \uqe.
\end{align*}
This completes the proof.
\end{proof}
In what follows, we use the notations $D^{\pm 1}=\sum D'_{\pm }\otimes D''_{\pm} $.
\begin{lemma}\label{qd}
We have
\begin{align}
&\sum \dot Y( \uqzq \otimes \uqz D_{\pm }')\otimes \dot Y( \uq \otimes \uqz D_{\pm}'') \subset (\uqe)^{\otimes 2},\label{qd1}
\\
&\sum \dot Y( \uqzq \otimes \uqz D_{\pm}')\otimes \dot Y(\uqz D_{\pm}''\otimes \uq )\subset  (\uqe)^{\otimes 2},\label{qd4}
\\
&\sum \dot Y(\uqz D_{\pm}'\otimes \uqzq )\otimes \dot Y(\uq \otimes \uqz D_{\pm}'')\subset  (\uqe)^{\otimes 2},\label{qd3}
\\
&\sum \dot Y(\uqz D_{\pm}'\otimes \uqzq )\otimes \dot Y(\uqz D_{\pm}''\otimes \uq )\subset   (\uqe)^{\otimes 2}.\label{qd2}
\end{align}
\end{lemma}
\begin{proof}
First, we  prove (\ref{qd4}) with $D$. 
Let us  assume a weaker inclusion
\begin{align}
\sum \dot Y( \uqzq \otimes D')\otimes \dot Y(  D''\otimes  \uq) \subset (\uqe)^{\otimes 2},\label{eno}
\end{align}
which we prove later.
We have
\begin{align}
\begin{split}\label{Dd}
&\sum \dot Y( \uqzq \otimes \uqz D')\otimes \dot Y( \uqz D'' \otimes \uq)
\\
 =&\sum \dot Y( \uqzq \otimes D'\uqz)\otimes \dot Y( \uqz D'' \otimes \uq)
\\
\subset &\sum\dot Y( \uqzq \otimes \uqz)\big(\dot Y(\uqzq \otimes D')\triangleleft \uqz \big)\otimes\big(\uqz\triangleright \dot Y( D''\otimes \uq)\big)
\dot Y( \uqz\otimes \uq)
\\
\subset &\sum\dot Y( \uqzq \otimes \uqz)\big(\uqe \triangleleft \uqz \big)\otimes\big( \uqz\triangleright \uqe \big)\dot Y( \uqz \otimes \uq)
\\
\subset &\sum\dot Y( \uqzq \otimes \uqz)\cdot \uqe \otimes\uqe \cdot \dot Y( \uqz \otimes \uq)
\\
\subset& (\uqe)^{\otimes 2},
\end{split}
\end{align}
where the identity follows from  (\ref{exD}), 
the first inclusion follows from Lemma \ref{y2}, $\Delta  (X)\subset X^{\otimes 2}$, for $X=\uq,\uqz,\uqzq,$ and the last inclusion follows from   Lemma \ref{qq}.

Now, we prove (\ref{eno}). By (\ref{d2}) and (\ref{d4}), we have
\begin{align*}
( 1\otimes S^{\pm 1}\otimes 1\otimes S^{\pm 1})(\Delta \otimes \Delta  )(D)=\sum D'_{1,-}D'_{1}\otimes D'_{2}D'_{2,-}\otimes D''_{2,-}D''_{1}\otimes D''_{2}D''_{1,-},
\end{align*}
where  $D=\sum D_{1}'\otimes D_{1}''=\sum D_{2}'\otimes D_{2}''$ and $D^{-1}=\sum  D_{1,-}'\otimes D_{1,-}''=\sum D_{2,-}'\otimes D_{2,-}''$.

For  $a\in \uqzq$  and $b\in \uq$ homogeneous, we have
\begin{align}
\begin{split}\label{ad3}
&\sum \dot Y(a\otimes D')\otimes \dot Y(D''\otimes b)
\\
&=\sum a_{(1)}D'_2D'_{2,-}S(a_{(2)}) D'_{1,-}D'_{1}\otimes  D''_{2,-}D''_{1}S^{-1}(b_{(2)}) D''_{2}D''_{1,-}b_{(1)}
\\
&=\sum a_{(1)}D'_2D'_{2,-}K^{-|b_{(2)}|}S(a_{(2)}) K^{|b_{(2)}|}D'_{1,-}D'_{1}\otimes  S^{-1}(b_{(2)})D''_{2,-}D''_{1} D''_{2}D''_{1,-}b_{(1)}
\\
&=\sum a_{(1)}K^{-|b_{(2)}|}S(a_{(2)}) K^{|b_{(2)}|}\otimes  S^{-1}(b_{(2)})b_{(1)}
\\
&=\sum (a\triangleright K^{-|b_{(2)}|})K^{|b_{(2)}|}\otimes   S^{-1}(b_{(2)})b_{(1)},
\end{split}
\end{align}
where by (\ref{bibi1})--(\ref{bibi3}), we can assume that  $ S^{-1}(b_{(2)})b_{(1)}\in \uqe$, with $b_{(1)},b_{(2)}\in \uq$ homogeneous. By Corollary \ref{Habii}, we have $a\triangleright K^{-|b_{(2)}|} \in K^{|b_{(2)}|}\uqe$.
Hence we have
\begin{align*}
\sum (a\triangleright K^{-|b_{(2)}|})K^{|b_{(2)}|}\otimes   S^{-1}(b_{(2)})b_{(1)}\subset (K^{|b_{(2)}|}\uqe)K^{|b_{(2)}|}\otimes \uqe\subset (\uqe)^{\otimes 2},
\end{align*}
which completes the proof of  (\ref{eno}).

We can prove  (\ref{qd1}), (\ref{qd4}) with $D^{-1}$, (\ref{qd3}), and  (\ref{qd2})  almost in the same way   by using 
\begin{align*}
&\sum \dot Y(a\otimes D_{\pm}')\otimes \dot Y(b\otimes D_{\pm}'')
=\sum (a\triangleright K^{\pm |b_{(2)}|})K^{\mp |b_{(2)}|}\otimes b_{(1)}S(b_{(2)}),
\\
&\sum \dot Y(a\otimes D_{-}')\otimes \dot Y(D_{-}''\otimes b)
=\sum (a\triangleright K^{|b_{(2)}|})K^{-|b_{(2)}|}\otimes S^{-1}(b_{(2)})b_{(1)},
\\
&\sum \dot Y(D_{\pm}'\otimes a)\otimes \dot Y(b\otimes D_{\pm}'')
=\sum  K^{\mp |b_{(2)}|}(K^{\pm |b_{(2)}|}\triangleleft a)\otimes    b_{(1)}S(b_{(2)}),
\\
&\sum \dot  Y(D_{\pm}'\otimes a)\otimes \dot Y(D_{\pm}''\otimes b)
=\sum K^{\pm |b_{(2)}|}(K^{\mp |b_{(2)}|}\triangleleft a)\otimes S^{-1}(b_{(2)})b_{(1)},
\end{align*}
for $a\in \uqzq$  and $b\in \uq$ homogeneous.
\end{proof}

\subsection{Outline of the proof of Theorem \ref{1}}
We give an outline of the proof of Theorem \ref{1}.
There are two steps.
The first step  is in Section \ref{proof}. We prove the following proposition.
\begin{proposition}\label{2}
Let $T\in BT_n$ be a boundary bottom tangle and $(\tilde T;g,g_1,\ldots, g_n)$ a boundary data for $T$.
For each state  $s\in \mathcal{S}(\tilde T)$ we have
\begin{align*}
\mu ^{[g_1,\ldots ,g_n]}\bar  Y^{\otimes g}(J_{\tilde T,s}) \in (\bar U_q^{\ev})^{\otimes n}.
\end{align*}
\end{proposition}
The second step is  in Section \ref{completion}. We  define a completion $\uqenh{n}$ of $(\bar U_q^{\ev})^{\otimes n}$ and prove Theorem \ref{1}, i.e., 
\begin{align*}
J_T=\sum_{s\in \mathcal{S}(\tilde T)}\mu ^{[g_1,\ldots ,g_n]}\bar  Y^{\otimes g}(J_{\tilde T,s}) \in \uqenh{n}.
\end{align*}

In the above two steps, we use ``graphical calculus'' because the proof is too complicated to be written down by using expressions.
In order to do so, in Section  \ref{cat}, we define two symmetric monoidal categories   $\mathcal{A}$, $\mathcal{M}$, and a functor $\mathcal{F}\co \mathcal{A} \rightarrow \mathcal{M}$.
\section{The  categories $\mathcal{M}, \mathcal{A}$ and  the  functor $\mathcal{F} \colon  \mathcal{A}\rightarrow  \mathcal{M}$}\label{cat}
In what follows, we use strict symmetric monoidal categories and strict symmetric monoidal functors.
Since we use only strict ones, we omit the word ``strict''.
For the definition of symmetric monoidal categories, see \cite{Ka}, \cite{Mc}.
\subsection{The  category $\mathcal{M}$}
We define the symmetric monoidal category  $\mathcal{M}$.
The  objects in $\mathcal{M}$ are non-negative integers.
For $k,l\geq 0,$ the morphisms from $k$ to $l$ in $\mathcal{M}$ are  $\mathbb{Z}[q,q^{-1}]$-submodules of  the $\mathbb{Q}[[h]]$-module $\Hom^{\c}_{\mathbb{Q}[[h]]}(U_h^{\hat \otimes k}, U_h^{\hat \otimes l})$ of continuous $\mathbb{Q}[[h]]$-linear maps from $U_h^{\hat \otimes k}$ to $U_h^{\hat \otimes l}$.

We  equip $\mathcal{M}$ with a symmetric monoidal category structure as follows.
\begin{itemize}
\item The identity of an object $k$ in $\mathcal{M}$ is defined by $\id_k=\mathbb{Z}[q,q^{-1}]\id_{U_h^{\hat \otimes k}}$.

 The composition of morphisms $k\stackrel{X}{\longrightarrow }l\stackrel{Y}{\longrightarrow}m$ in  $\mathcal{M}$ is defined by
\begin{align*}
Y\circ X=\Span_{\mathbb{Z}[q,q^{-1}]}\{ y\circ x\ | \ x\in X, \ y\in Y\}.
\end{align*}
\item The unit object is $0$, and the tensor product of objects $k$ and $l$ in $\mathcal{M}$ is define by $k+l$.

The tensor product of morphisms $F\co k\rightarrow l$ and $F'\co k'\rightarrow l'$ in $\mathcal{M}$ is defined by
\begin{align*}
Z\otimes Z'=\Span_{\mathbb{Z}[q,q^{-1}]}\{ \varphi (z\otimes  z')\ | \ z\in Z, \ z'\in Z'\},
\end{align*}
where  $\varphi $ is the natural  $ \mathbb{Q}[[h]]$-linear map 
\begin{align*}
\varphi \co \Hom^{\c}_{\mathbb{Q}[[h]]}(U_h^{\hat \otimes k}, U_h^{\hat \otimes l})\otimes 
\Hom^{\c}_{\mathbb{Q}[[h]]}(U_h^{\hat \otimes k'}, U_h^{\hat \otimes l'})\rightarrow 
\Hom^{\c}_{\mathbb{Q}[[h]]}(U_h^{\hat \otimes k+k'}, U_h^{\hat \otimes l+l'}).
\end{align*}
\item The symmetry $c_{k,l} \co k\otimes l\rightarrow l\otimes k$ of  objects $k$ and $l$ in $\mathcal{M}$ is  defined by
\begin{align*}
c_{k,l}=\mathbb{Z}[q,q^{-1}]\tau_{U_h^{\hat \otimes k},U_h^{\hat \otimes l}},
\end{align*}
where $\tau_{U_h^{\hat \otimes k},U_h^{\hat \otimes l}}\co U_h^{\hat \otimes k+l}\rightarrow U_h^{\hat \otimes l+k}$ is the continuous $\mathbb{Q}[[h]]$-linear map defined by
 \begin{align*}
\tau_{U_h^{\hat \otimes k},U_h^{\hat \otimes l}}(x\otimes y)= y\otimes x,
\end{align*}
for $x\in U_h^{\hat \otimes k}$ and $ y\in U_h^{\hat \otimes l}$.
\end{itemize}
It is straightforward to check the axioms of a symmetric monoidal category.
\subsection{The category $\mathcal{A}$ and the functor $\mathcal{F}\co \mathcal{A}\rightarrow \mathcal{M}$}\label{gene}
Let  $\mathcal{A}$ be the symmetric monoidal category with the unit object $I$,  freely
generated by an object $A$ and morphisms 
\begin{align*}
& \qintx{i} \in \Hom_{\mathcal{A}}(I,I),
\\
&\etax, \ex{i}, \fx{i},  \uqzx,  \uqzex \in \Hom_{\mathcal{A}}(I,A),
\\
&\Dx \in \Hom_{\mathcal{A}}(I,A^{\otimes 2}),
\\
& \varepsilonx \in \Hom_{\mathcal{A}}(A,I),
\\
&\Deltax \in \Hom_{\mathcal{A}}(A,A^{\otimes 2}),
\\
&\mux,  \Yx, \adx, \sadx \in \Hom_{\mathcal{A}}(A^{\otimes 2},A),
\end{align*}
 for $i\geq 0$. (Here $\Dx$ is one morphism, not two morphisms $\langle D^{+1}\rangle$ and $\langle D^{-1}\rangle$. )
We denote by $c_{X,Y}\co X\otimes Y\rightarrow Y\otimes X$ the symmetry of  objects $X,Y$ in $\mathcal{A}$.

We define the symmetric monoidal subcategories $\mathcal{A}_{ \mathfrak{S}}, \mathcal{A}_{\mu }, \mathcal{A}_{\Delta },$ and $ \mathcal{A}_{\mu, \Delta }$ of $ \mathcal{A}$
as follows.
On objects, we define $ \Ob(\mathcal{A}_{ \mathfrak{S}})=\Ob(\mathcal{A}_{\mu})=\Ob( \mathcal{A}_{\Delta})=\Ob(\mathcal{A}_{\mu,\Delta})=\Ob(\mathcal{A})$.
On morphisms, $\mathcal{A}_{\mathfrak{S}}$ is generated by no morphism as a symmetric monoidal category, i.e., for $k,l\geq 0,$ $k\neq l$, 
we have $\Hom_{\mathcal{A}_{ \mathfrak{S}}}(A^{\otimes k},A^{\otimes l})=\emptyset $, and  for $l\geq 0,$
 the monoid $\Hom_{\mathcal{A}_{ \mathfrak{S}}}(A^{\otimes l},A^{\otimes l})$ is isomorphic to the symmetric group $\mathfrak{S}(l)$ in a natural way.
On morphisms, $\mathcal{A}_{\mu}$ is generated by $\mux$, $\mathcal{A}_{\Delta}$ is generated by $\Deltax$,
and $\mathcal{A}_{\mu,\Delta}$ is generated by $\mux$ and $\Deltax$, as symmetric monoidal categories.

Let $\mathcal{F}\co \mathcal{A}\rightarrow \mathcal{M}$ be the  symmetric monoidal functor defined by $\mathcal{F}(A)=1$ on the objects and  
\begin{align*}
\mathcal{F}(\qintx{i} )&=\Z \{i\}_q!,
\\
\mathcal{F}(\etax)&=\Z \eta,
\\
\mathcal{F}(\varepsilonx)&=\Z \varepsilon,
\\
\mathcal{F}(\mux)&=\Z \mu,
\\
\mathcal{F}(\Deltax)&=\Z \Delta,
\\
\mathcal{F}(\Yx )&=\Z \dot Y,
\\
\mathcal{F}(\adx)&=\Z \ad,
\\
\mathcal{F}(\sadx)&=\Z \sad,
\\
\mathcal{F}(\ex{i} )&=\uqz \e{i},
\\
\mathcal{F}(\fx{i} )&=\uqz \f{i},
\\
\mathcal{F}(\uqzx )&=\uqz,
\\
\mathcal{F}(\uqzex )&=\uqze,
\\
\mathcal{F}(\Dx )&= (\uqz)^{\otimes 2} D+(\uqz)^{\otimes 2} D^{-1},
\end{align*}
for $i\geq 0$, on the morphisms.
Here, for a $\Z$- submodule $X\subset U_h^{\hat \otimes n}$, we identify $X$ with a $\Z$-submodule of $\Hom^{\c}_{\mathbb{Q}[[h]]}(\mathbb{Q}[[h]], U_h^{\hat \otimes n})$ by  identifying each $x\in X$ 
with  the map $f_x\co \mathbb{Q}[[h]]\rightarrow U_h^{\hat \otimes n}$ such that $f_x(a)= ax$ for $a\in \mathbb{Q}[[h]]$.

In what follows, we  use  diagrams of  morphisms in $\mathcal{A}$ as follows.
The generating morphisms in $\mathcal{A}$ are depicted as in Figure \ref{fig:obx}.
The composition $y\circ x$ of morphisms $x$ and $y$ in $\mathcal{A}$ is represented  as the diagram obtained by placing  the diagram of $x$ on the top  of  the diagram of $y$, see Figure \ref{fig:comp} (a).
The tensor product $z\otimes z'$ of morphisms $z$ and $z'$ in $\mathcal{A}$ is represented  as the diagram  obtained 
by placing the diagram of $z'$ to the right of  the diagram of $z$, see Figure \ref{fig:comp} (b).
\begin{figure}
\centering
\includegraphics[width=12cm,clip]{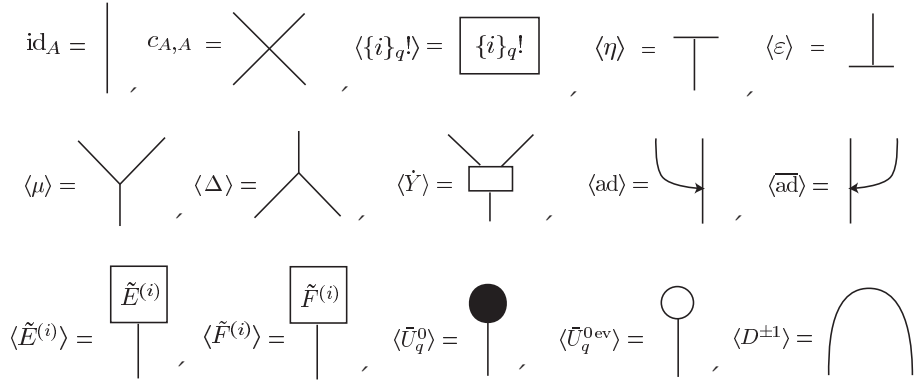}
\caption{The diagrams of  the generator morphisms in $\mathcal{A}$ }\label{fig:obx}
\end{figure}
\begin{figure}
\centering
\includegraphics[width=10cm,clip]{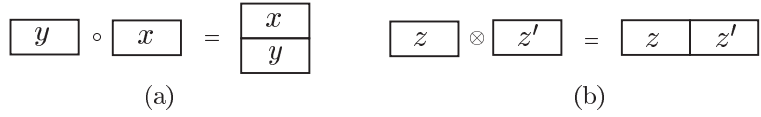}
\caption{(a) Composition  (b)  Tensor product }\label{fig:comp}.
\end{figure}

For a  diagram of a morphism $b\co A^{\otimes k}\rightarrow A^{\otimes l}$ in $\mathcal{A}$, we call the $k$ edges at the top  of  the diagram  the \textit{input edges} of $b$,
and the $l$ edges at the bottom of the diagram the \textit{output edges} of $b$.

For simplicity, a copy of a  generating morphism $f$ of $\mathcal{A}$ appearing in a diagram will be called  `an $f$' in the diagram.

 \subsection{Some morphisms in $\mathcal{A}$}
In this section, we define  morphisms $\mux^{[g_1,\ldots, g_n]}
,\Deltax^{[m_1,\ldots, m_l]}, \Rx{i},$ and $\bYx$ in $\mathcal{A}$.

For $ g_1,\ldots,g_n\geq 0$, we  define 
\begin{align*}
&\mux^{[g_1,\ldots, g_n]}\in \Hom_{\mathcal{A}}(A^{\otimes g_1+\cdots+g_n},A^{\otimes n})
\end{align*}
in a  similar way to (\ref{mug}), and 
for $m_1,\ldots , m_l\geq 0$, we define
\begin{align*}
&\Deltax^{[m_1,\ldots, m_l]}\in \Hom_{\mathcal{A}}(A^{\otimes l},A^{\otimes m_1+\cdots+m_l})
\end{align*}
in a  similar way to (\ref{Deg}), see Figure  \ref{fig:mD}. 
Clearly we have
\begin{figure}
\centering
\includegraphics[width=11cm,clip]{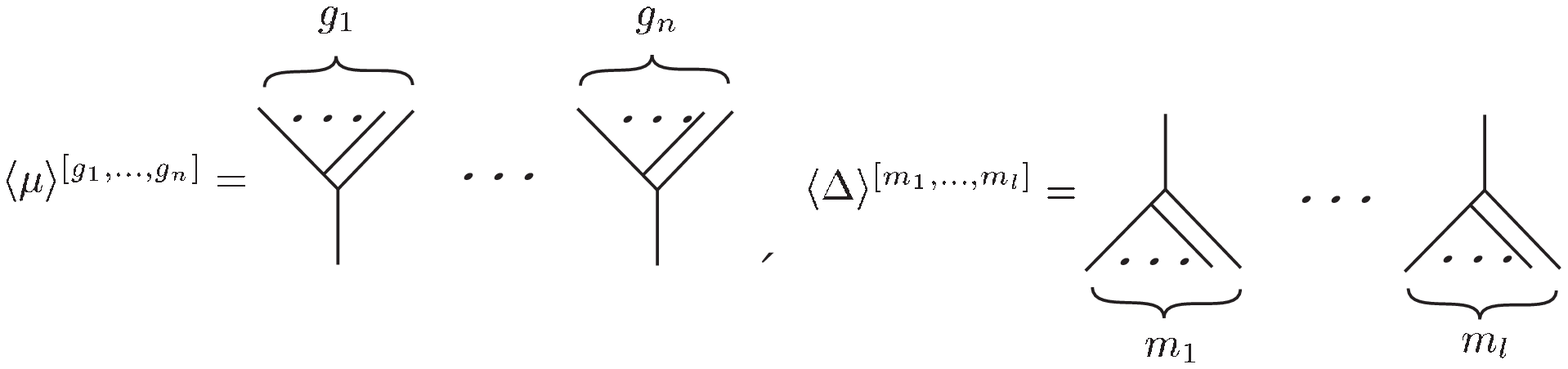}
\caption{$\mux^{[g_1,\ldots, g_n]}$ and $\Deltax^{[m_1,\ldots, m_l]}$}\label{fig:mD}
\end{figure}
\begin{align}
&\mathcal{F}\big(\mux^{[g_1,\ldots, g_n ]}\big)=\Z \mu ^{[g_1,\ldots, g_n ]},\label{F1}
\\
&\mathcal{F}\big(\Deltax^{[m_1,\ldots, m_l ]}\big)=\Z \Delta  ^{[m_1,\ldots, m_l ]}.
\end{align} 

For $i\geq 0,$ set 
\begin{align*}
\Thetax{i}=\qintx{i}\otimes \fx{i}\otimes \ex{i}\in \Hom_{\mathcal{A}}(I ,A^{\otimes 2}).
\end{align*}
We represent $\Thetax{i}$ as in Figure \ref{fig:theta} (a).
We define 
 \begin{figure}
\centering
\includegraphics[width=11cm,clip]{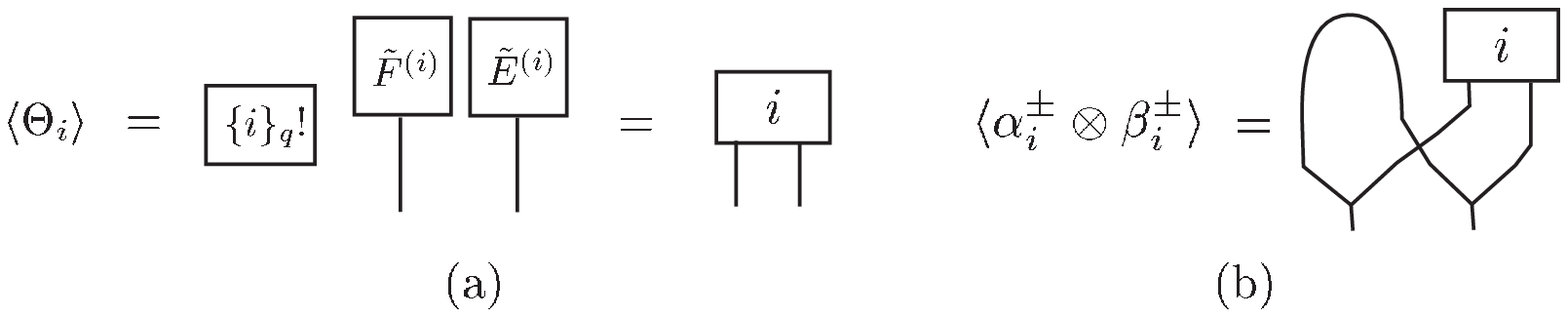}
\caption{(a) A diagram  of  $\Thetax{i}$ (b) A diagram of  $\Rx{i}$}\label{fig:theta}
\end{figure} 
\begin{align*}
\Rx{i}\in \Hom_{\mathcal{A}}(I ,A^{\otimes 2})
\end{align*}
as  in Figure \ref{fig:theta} (b),
i.e.,
\begin{align*}
\Rx{i}=&\big(\mux\otimes \mux\big)\circ  \big(\id_A\otimes c_{A,A} \otimes \id_A\big) \circ \big(\Dx \otimes \Thetax {i}\big).
\end{align*}
In $U_h^{\hat \otimes 2}$, we have
\begin{align*}
&\mathcal{F}\big(\Rx{i}\big)(1)=\big(\uqz\otimes \uqz\big)D\big( \f{i}\otimes e^i\big)+\big(\uqz\otimes \uqz\big)D^{-1}\big(\f{i}\otimes e^i\big),
\end{align*}
which implies 
\begin{align*}
\alpha _i\otimes \beta _i,\ \alpha ^-_i\otimes \beta ^-_i\in \mathcal{F}\big(\Rx{i}\big)(1).
\end{align*}  
Since  we have
\begin{align*}
(S^j\otimes S^k)\Big(\mathcal{F}\big(\Rx{i})\big(1)\Big)=\mathcal{F}\big(\Rx{i}\big)(1),
\end{align*} 
for $j,k\in \mathbb{Z}$, it follows that
\begin{align}\label{F4}
S^j(\alpha _i)\otimes S^k(\beta _i), \ S^j(\alpha ^-_i)\otimes S^k(\beta ^-_i)\in \mathcal{F}\big(\Rx{i}\big)(1).
\end{align}

We define 
\begin{align*}
\bYx \in \Hom_{\mathcal{A}}(A^{\otimes 2} ,A^{\otimes 1})
\end{align*}
as  in Figure \ref{fig:RbY}, i.e.,
\begin{figure}
\centering
\includegraphics[width=3cm,clip]{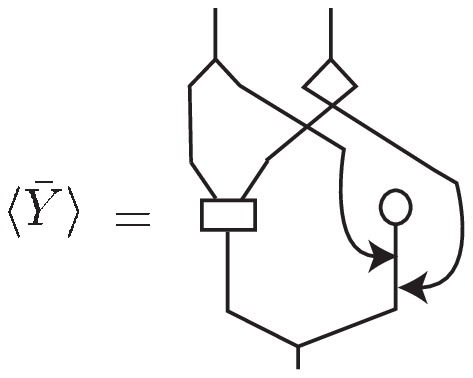}
\caption{A diagram of  $\bYx$ }\label{fig:RbY}
\end{figure}
\begin{align*}
\bYx=&\mux\circ \big(\id _A\otimes \sadx\big)\circ \big(\id_A\otimes \adx\otimes \id_A\big)
\\
&\circ \big(\Yx\otimes \id_A\otimes \Kx\otimes \id_A\big)\circ \big(\id\otimes c_{A^{\otimes 2},A}\big)\circ  \big(\Deltax\otimes \Deltax\big).
\end{align*}
By Lemma \ref{ydyb}, we have
\begin{align*}
\bar Y=&\mu\circ (\id_{U_h}\otimes \sad)\circ (\id_{U_h}\otimes \ad\otimes \id_{U_h})\circ (\dot Y\otimes \id_{U_h}\otimes K^2\otimes \id_{U_h})\circ (\id_{U_h}\otimes \tau_{U_h^{\hat \otimes 2},U_h})\circ  (\Delta\otimes \Delta).
\end{align*}
Hence we have 
\begin{align}
&\bar Y\in \mathcal{F}\big(\bYx\big).\label{F3}
\end{align}
\section{Proof of Proposition \ref{2}}\label{proof}
The goal of this section is to prove Proposition \ref{2}.
In Section \ref{gr}, we define a subset $\Gamma _g\subset \Hom_{\mathcal{A}}(I ,A^{\otimes g})$, and prove 
$\bar  Y^{\otimes g}(J_{\tilde T,s})\in \mathcal{F}(\Gamma _g)(1)$.
Here, for a subset $X\subset \Hom_{\mathcal{A}}(I ,A^{\otimes g})$, we set 
\begin{align*}
\mathcal{F}(X)(1)=\bigcup_{B\in X} \mathcal{F}(B)(1) \subset U_h^{\hat \otimes g}.
\end{align*}
In Section \ref{gr2}--\ref{prBB}, we prove $\mathcal{F}(\Gamma _g)(1)\subset (\uqe)^{\otimes g}$.
Thus we have
\begin{align*}
\bar  Y^{\otimes g}(J_{\tilde T,s}) \in (\bar U_q^{\ev})^{\otimes g},
\end{align*}
which implies Proposition \ref{2}. 
\subsection{The set $\Gamma _g\subset \Hom_{\mathcal{A}}(I ,A^{\otimes g})$}\label{gr}
Let  $\mathcal{W}$ be the symmetric monoidal category  freely generated by two objects  $W_+$ and $W_-$ and three morphisms 
\begin{align*}
\mu _{\mathcal{W}}\co (W_+)^{\otimes 2}\rightarrow W_+,
\quad
\ad_{\mathcal{W}}\co W_- \otimes W_+\rightarrow W_+,
\quad 
\sad_{\mathcal{ W}}\co W_+ \otimes W_- \rightarrow W_+,
\end{align*}
see Figure \ref{fig:b4} for example.  
We define the symmetric monoidal functor $\mathcal{F}_{\mathcal{W}}\co \mathcal{W}\rightarrow \mathcal{A}$ by $\mathcal{F}_{\mathcal{W}} (W_+)=\mathcal{F}_{\mathcal{ W} }(W_-)=A$ on objects, and 
\begin{align*}
\mathcal{F}_{\mathcal{W}}(\mu_{\mathcal{W}})=\mux,
\quad
\mathcal{F}_{\mathcal{W}}(\ad_{\mathcal{W}})=\adx,
\quad
\mathcal{F}_{\mathcal{W}}(\sad_{\mathcal{W}})=\sadx,
\end{align*}
on  morphisms.
\begin{figure}
\centering
\includegraphics[width=13cm,clip]{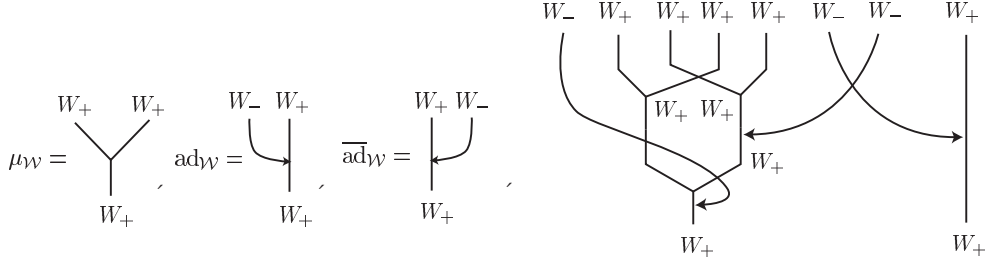}
\caption{A morphism in $\mathcal{W}$}\label{fig:b4}
\end{figure}

For $g\geq 0$, let $\tilde \Gamma_g$ be the set of  quadruples $b=(b_1,b_2,b_3,b_4)\in \Hom({\mathcal{A}})^{\times 4}$ of 
composable morphisms
\begin{align*}
I\stackrel{b_1}{\longrightarrow }A^{\otimes 2l_1+2l_2+l_3}\stackrel{b_2}{\longrightarrow }A^{\otimes l_4+2l_5}\stackrel{b_3}{\longrightarrow }A^{\otimes l_4+l_5+l_6}\stackrel{b_4}{\longrightarrow }A^{\otimes g}
\end{align*}
in $\mathcal{A}$ such that
\begin{align*}
b_1&=\Dx ^{\otimes l_1}\otimes\Thetax{s_1}\otimes \cdots \otimes \Thetax{s_{l_2}}\otimes \uqzx^{\otimes l_3},
\\
b_2&\in   \Hom_{\mathcal {A}_{\mu,\Delta}}(A^{\otimes 2l_1+2l_2+l_3},A^{\otimes  l_4+2l_5}), 
\\
b_3&=\id_A^{\otimes l_4}\otimes \Yx^{\otimes l_5}\otimes \Kx^{\otimes l_6},
\\
b_4&\in \mathcal{F}_{\mathcal{W}}\big(\Hom_{\mathcal{W}}(W_-^{\otimes l_4}\otimes W_+^{\otimes l_5+l_6}, W_+^{\otimes g})\big),
\end{align*}
for  $l_1,\ldots,l_6, s_1,\ldots, s_{l_2}\geq 0,$
satisfying  Condition A  below.
\begin{center}
\textbf{Condition A}:  On a diagram of $b_4\circ b_3\circ b_2\circ b_1$, from  each output edge of  $\Thetax{s_p}$ for  $p=1,\ldots, l_2$, we can find a descending path to an input edge  of  a $\Yx$.
\end{center}
For example, see Figure \ref{fig:b}, where the dotted arrow  denotes a path as in Condition A from the right output edge of  $\Thetax{s_1}$.
\begin{figure}
\centering
\includegraphics[width=7.5cm,clip]{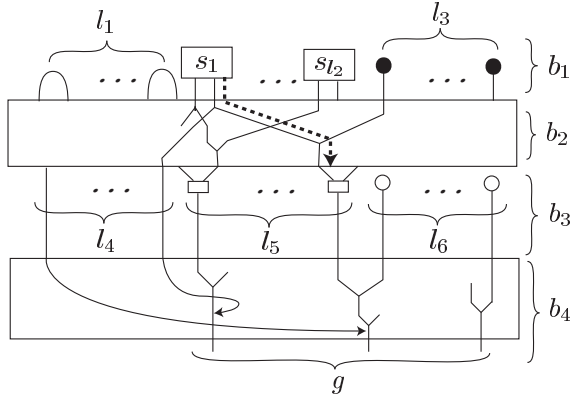}
\caption{An example of $b_4\circ b_3\circ b_2\circ b_1$ with  $(b_1,b_2,b_3,b_4)\in \tilde  \Gamma_g$
}\label{fig:b}
\end{figure}

Let $\lambda \co \tilde \Gamma_g\rightarrow \Hom_{ \mathcal{A}}(I,A^{\otimes g})$ be the composition map defined by 
\begin{align*}
\lambda (b_1,b_2,b_3,b_4)= b_4\circ b_3\circ b_2\circ b_1.
\end{align*}  Set
\begin{align*}
\Gamma_g=\lambda (\tilde \Gamma _g)\subset \Hom_{ \mathcal{A}}(I,A^{\otimes g}).
\end{align*}
 
We consider the following sequence of   maps
\begin{align*}
\tilde \Gamma_g\stackrel{\lambda }{\longrightarrow}  \Gamma_g\subset \Hom_{ \mathcal{A}}(I,A^{\otimes g})\stackrel{\mathcal{F}}{\longrightarrow} \Hom_{\mathcal{M}}(0,g).
\end{align*}

\begin{lemma}\label{A}
Let $T\in BT_n$ be a boundary bottom tangle and $(\tilde T;g,g_1,\ldots, g_n)$ a boundary data for $T$.
For each state $s\in \mathcal{S}(\tilde T)$, we have $\bar Y^{\otimes g}(J_{\tilde T,s})\in \mathcal{F}(\Gamma _g)(1).$
\end{lemma}
\begin{proof}
It is enough to construct an element $B\in \Gamma _g$ such that $\bar Y^{\otimes g}(J_{\tilde T,s})\in \mathcal{F}(B)(1).$

Recall  from Section \ref{bottom inv} the definition of $J_{\tilde T,s}$ associated with a fixed diagram of $\tilde T$ with the crossings  $c_1,\ldots,c_l.$
 We put the labels $S'(\alpha_{s(c_i)} ^{\pm })$ and $S'(\beta_{s(c_i)}^{\pm})$ on the 
crossing $c_i$ for $i=1,\ldots,l$, and 
put the labels  $K$ and $K^{-1}$ on the maximal and minimal  points, respectively, on strands  running from left to right.
After that, we multiply the labels on each strand, and take the tensor product.
Thus, there is $k\geq 0$ and a permutation $\sigma \in \mathfrak{S}(2l+k)$ such that
\begin{align*}
J_{\tilde T,s}\in \big(\mu ^{[N_1,\ldots,N_{2g}]}\circ \sigma \big)\Big(S'(\alpha_{s(c_1)} ^{\pm })\otimes S'(\beta_{s(c_1)}^{\pm})\otimes \cdots \otimes S'(\alpha_{s(c_l)} ^{\pm })\otimes S'(\beta_{s(c_l)}^{\pm})\otimes (\uqz)^{\otimes k}\Big),
\end{align*}
where, for each $i=1,\ldots,2g$, $N_i\geq 0$ is the number of labels put on $i$th strand of  $\tilde T$.

By (\ref{F1}) and (\ref{F4}), we have  $J_{\tilde T,s}\in \mathcal{F}(u)(1)$, where
\begin{align*}
&u=\mux ^{[N_1,\ldots,N_{2g}]}\circ  \sigma \circ \Big(\Rx{s(c_1)}\otimes \cdots \otimes \Rx{s(c_l)}\otimes \uqzx^{\otimes k}\Big).
\end{align*}
Here we identify $\sigma \in \mathfrak{S}(2l+k)$ with the  corresponding morphism in $\Hom_{\mathcal{A}_{\mathfrak{S}}}(A^{\otimes 2l+k}, A^{\otimes 2l+k})$.

By (\ref{F3}), we have
\begin{align*}
\bar Y^{\otimes g}(J_{\tilde T,s})\in \mathcal{F}\big(\bYx^{\otimes g}\circ u\big)(1).
\end{align*}

Set $B=\bYx^{\otimes g}\circ  u\in \Hom_{\mathcal{A}}(I,A^{\otimes g})$. It is not difficult to check that $B\in \Gamma _g$ as in Figure \ref{fig:jt2}.
In particular,  $B$ satisfies Condition A, since for each $i=1,\ldots,l,$ the  output edges of  $\Thetax{s(c_l)}$ go down  to the output edges of $u$, and there is  a descending path from each output edge of $u$  to an  input edge of a $\Yx$, see the dotted lines in Figure \ref{fig:jt2} for example.
\end{proof}
\begin{figure}
\centering
\includegraphics[width=13cm,clip]{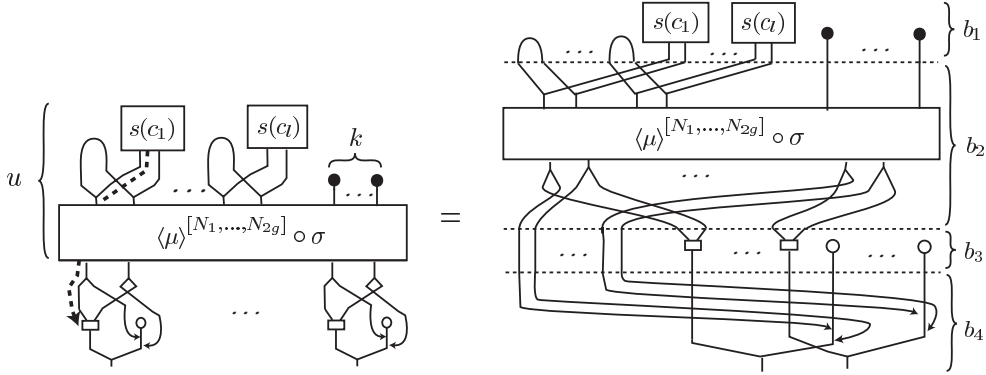}
\caption{ $\bYx^{\otimes g}\circ  u=b_4\circ b_3\circ b_2\circ b_1$ with $(b_1,b_2,b_3,b_4)\in \tilde \Gamma _g$}\label{fig:jt2}
\end{figure}

Proposition \ref{2}  follows  from Lemma \ref{A} and the following lemma.
\begin{lemma}\label{B''}
For $g\geq 0,$ we have $\mathcal{F}(\Gamma _g)(1)\subset (\uqe)^{\otimes g}$.
\end{lemma}

The outline of  the proof of Lemma \ref{B''} is as follows.
We define two subsets $\Gamma _g',\Gamma _g''\subset \Gamma _g$ such that
$\Gamma _g''\subset \Gamma _g'\subset \Gamma _g$, and 
prove the following inclusions
\begin{align}
\mathcal{F}(\Gamma _g)(1)\subset \mathcal{F}(\Gamma' _g)(1)\subset  \mathcal{F}'\big(\Gamma _g'\big)(1)\subset  \mathcal{F}'\big(\Gamma _g''\big)(1)\subset (\uqe)^{\otimes g}\label{naga},
\end{align}
where $\mathcal{F}'$ is a modification of the functor $\mathcal{F}$, which is  defined in Section \ref{f7}.
\subsection{The subset $\Gamma _g'\subset \Gamma _g$}\label{gr2}
In this section, we define  the subset $\Gamma _g'\subset \Gamma _g$.

For $g\geq 0$, let $\dot \Gamma'_g$ be the set of   $7$-tuples $(b_1,d,w,k, \sigma  ,b_3,b_4)$ of 
morphisms in $\mathcal{A}$, such that $(b_1, \sigma \circ (d\otimes w\otimes k) ,b_3,b_4)\in \tilde \Gamma _g$
is well-defined,  $\sigma \in \Hom(\mathcal{A}_{\mathfrak{S}})$ and 
\begin{align*}
b_1&= \Dx^{\otimes l_1}\otimes \Thetax{s_1}\otimes \cdots \otimes \Thetax{s_{l_2}}\otimes \uqzx ^{\otimes l_3},
\\
d \in \Hom _{\mathcal{A}_{\mu}}(&A^{\otimes 2l_1},A^{\otimes l_4}),
\quad
w =\bigotimes_{p=1}^{l_2} \big(\Deltax^{[m_p,n_p]}\big),
\quad
k =\Deltax^{[r_1,\ldots,r_{l_3}]},
\end{align*}
for $l_1,\ldots,l_4,s_1,\ldots, s_{l_2}\geq 0$, $m_1,\ldots,m_{l_2},n_1,\ldots,n_{l_2}, r_1,\ldots,r_{l_3}\geq 1$.
See Figure \ref{fig:dwk} for an example of  $\sigma \circ(d\otimes w\otimes k)\circ b_1$.

Let $\kappa \co \dot \Gamma _g\rightarrow \tilde \Gamma _g$ be the map defined by
\begin{align*}
\kappa (b_1,d,w,k, \sigma  ,b_3,b_4)=(b_1, \sigma \circ (d\otimes w\otimes k) ,b_3,b_4).
\end{align*}
Set
\begin{align*}
\tilde \Gamma'_g &=\kappa(\dot  \Gamma'_g)\subset \tilde \Gamma _g,
\\
\Gamma'_g &=\lambda (\tilde \Gamma'_g)\subset \Gamma _g.
\end{align*}
\begin{figure}
\centering
\includegraphics[width=7cm,clip]{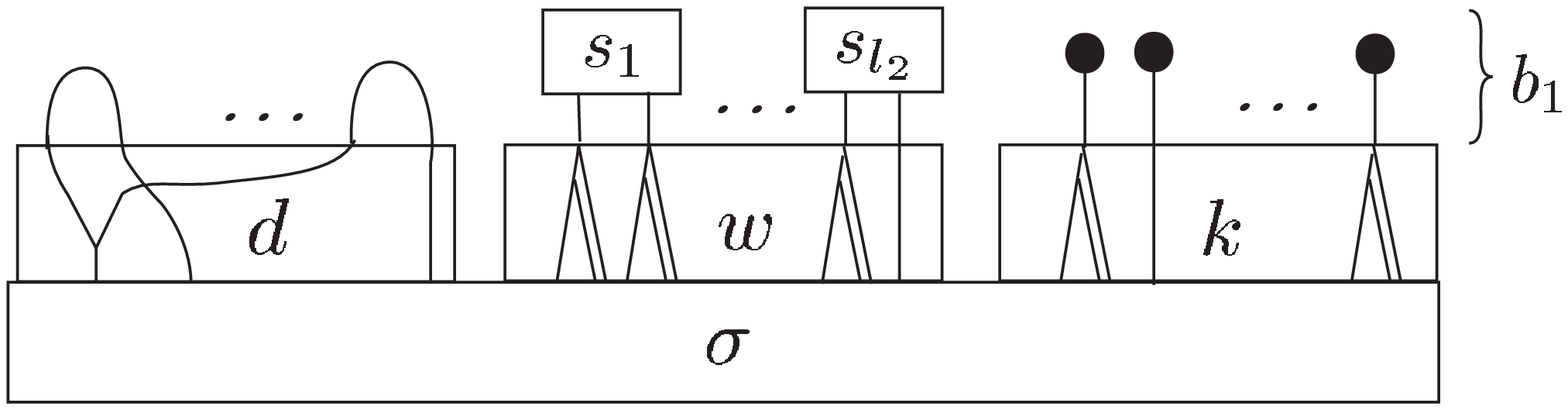}
\caption{An example  of $\sigma \circ(d\otimes w\otimes k)\circ b_1$ for $(b_1,d,w,k, \sigma  ,b_3,b_4)\in \dot  \Gamma _g'$ }\label{fig:dwk}
\end{figure} 

\subsection{Proof of $\mathcal{F}(\Gamma _g)(1)\subset \mathcal{F}(\Gamma' _g)(1)$}\label{SB}
In this section, we  define a preorder $\preceq $ on $\Gamma _g$,
and prove the following two lemmas, which imply $\mathcal{F}(\Gamma _g)(1)\subset \mathcal{F}(\Gamma' _g)(1)$.

\begin{lemma}\label{co''}
For  $B\preceq B'$ in $ \Gamma _g$, we have $\mathcal{F}(B)(1)\subset  \mathcal{F}(B')(1)$.
\end{lemma}
\begin{lemma}\label{c'}
For each $B\in \Gamma_g$, there exists $B'\in \Gamma'_g$ such that 
$B\preceq B'$.
\end{lemma}
The preorder  $\preceq $ on $\Gamma _g$ is generated by the binary relations $\stackrel{i}{\Rightarrow}$ for $i=1,\ldots,8$ on $\Hom(\mathcal{A})$
 defined by the local moves on diagrams as depicted in Figure \ref{fig:move}, where in each relation, the outsides of the two rectangles are the same.
 Note that  $\Gamma _g$ is \textit{closed} under $\stackrel{i}{\Rightarrow}$ for $i=1,\ldots,8$, i.e., for $B\stackrel{i}{\Rightarrow} B'$ in $\Hom(\mathcal{A})$, 
if $B\in \Gamma _g$, then $B'\in \Gamma _g$.
 In particular, we can check that each $\stackrel{i}{\Rightarrow}$ preserves Condition A.
\begin{figure}
\centering
\includegraphics[width=9cm,clip]{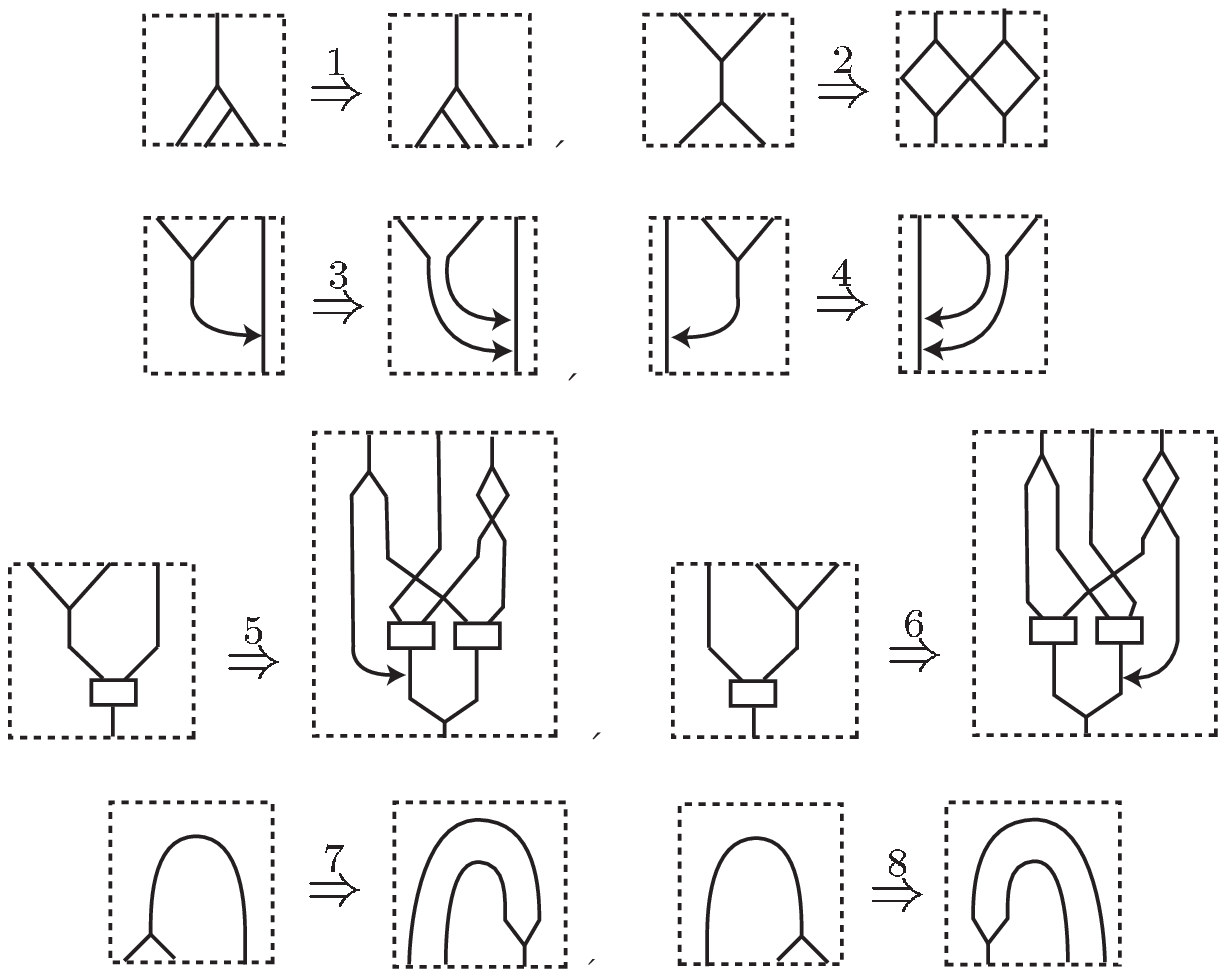}
\caption{Local moves $\stackrel{i}{\Rightarrow}$ for $i=1,\ldots,8$
 }\label{fig:move}
\end{figure}

\begin{proof}[Proof of Lemma \ref{co''}]
It is enough to prove that, for $B\stackrel{i}{\Rightarrow} B'$ in $\Gamma _g$ with $i\in \{1,\ldots,8\}$, 
 we have $\mathcal{F}(B)(1)\subset \mathcal{F}(B')(1)$.
 
The cases $i=1,2,3,4$ are clear.
The cases $i=5,6$ follow from Lemma \ref{y2}.
The cases $i=7,8$ follow from (\ref{d2}), $\Delta  (\uqz)\subset (\uqz)^{\otimes 2}$ and $\uqz\subset \mu((\uqz)^{\otimes 2})$.
\end{proof}

The rest of this section is devoted to the proof of Lemma \ref{c'}.
We divide  Lemma \ref{c'}  to the following two claims.
\begin{claim1} For $b=(b_1,b_2,b_3,b_4)\in \tilde \Gamma _g$,
there exists $b'=(b'_1,b'_2,b'_3,b'_4)\in \tilde \Gamma_g$  with 
$b'_2\in  \Hom(\mathcal{A}_{\Delta})$
 such that $\lambda (b)\preceq  \lambda (b')$.
 \end{claim1}
 \begin{claim2} For $b'=(b'_1,b'_2,b'_3,b'_4)\in \tilde \Gamma_g$  with 
$b'_2\in  \Hom(\mathcal{A}_{\Delta})$, there exists $b''=(b''_1,b''_2,b''_3,b''_4)\in \tilde \Gamma_g'$
such that $\lambda (b')\preceq   \lambda (b'')$.
\end{claim2}

Roughly speaking, we prove Claim 1 by reducing the number of the $\mux$'s of $b_2$ by using $\stackrel{i}{\Rightarrow}$ for $i=3,\ldots,6$.
For that purpose, we define ``$\mux$-complexity'' functions
\begin{align*}
|\ |,m&\co \Hom(\mathcal{A}_{\mu,\Delta})\rightarrow  \mathbb{Z}_{\geq 0}
\end{align*}
as follows. Given an element $b\in \Hom(\mathcal{A}_{\mu,\Delta})$, we color each edge of a diagram of $b$ with an non-negative integer.
First, we color each edge on the top with $0$.
Then, we color the  edges below inductively as in Figure \ref{fig:coloremu} (a).
We define $|b|$ as the maximal integer on the edges on the bottom.  We define $m(b)$ as the number of  the edges on the bottom colored with $|b|$.
For example, for the colored morphism $f\in \Hom(\mathcal{A}_{\mu,\Delta})$ in Figure \ref{fig:coloremu} (b), we have  $|f|=3$, and 
$m(f)=2$.
\begin{figure}
\centering
\includegraphics[width=12cm,clip]{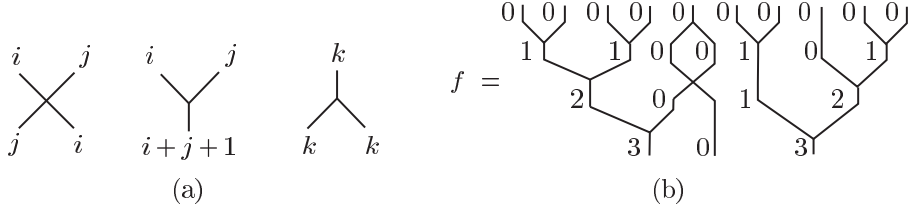}
\caption{(a) How to color the edges (b) An example of the coloring}\label{fig:coloremu}
\end{figure}

We use the following lemma.
\begin{lemma}\label{ca}
For every $B\in \Hom(\mathcal{A}_{\mu,\Delta})$, there exists $B_{\mu}\in\Hom (\mathcal{X}_{\mu})$ and $B_{\Delta}\in \Hom (\mathcal{X}_{\Delta})$ such that $B\preceq  B_{\mu}\circ B_{\Delta}$
and $|B|=|B_{\mu}\circ B_{\Delta}|$.
\end{lemma}
\begin{proof}
We can realize a path from $B$ to  $B_{\mu}\circ B_{\Delta}$ with some $B_{\mu}\in\Hom (\mathcal{X}_{\mu})$ and $ B_{\Delta}\in \Hom (\mathcal{X}_{\Delta})$ by using $\stackrel{2}{\Rightarrow}$,
which preserves $|\ |$ as in Figure \ref{fig:mud}.
\begin{figure}
\centering
\includegraphics[width=7cm,clip]{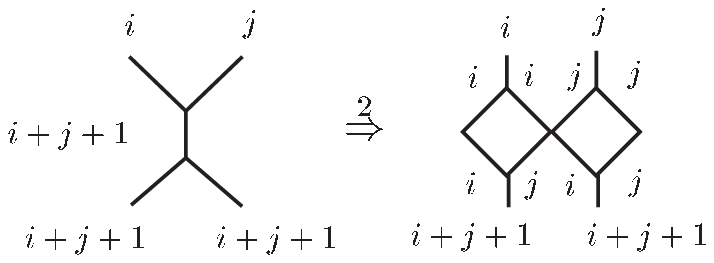}
\caption{The coloring before and after we apply  $\stackrel{2}{\Rightarrow}$}\label{fig:mud}
\end{figure}
\end{proof}
\begin{proof}[Proof of Claim 1]
We use  double induction  on $|b_2|$ and $m(b_2)$.
If $|b_2|=0$, then we have $b_2\in \Hom(\mathcal{A}_{\Delta})$. 
We assume $|b_2|>0$. It is enough to prove that there exists  $a=(a_1,a_2,a_3,a_4) \in \tilde \Gamma_g$ such that $\lambda (b)\preceq  \lambda (a)$ satisfying 
either $|b_2|>|a_2|$, or $|b_2|=|a_2|$ and $m(b_2)>m(a_2)$.

By Lemma \ref{ca}, we can assume  $b_2=b_{2,\mu}\circ b_{2,\Delta}$ with $b_{2,\mu}\in \Hom(\mathcal{A}_{\mu})$ and $ b_{2,\Delta}\in  \Hom(\mathcal{A}_{\Delta})$.
Since  $|b_{2,\mu}|=|b_2|>0$,  there is  a $\mux$ at the bottom of $b_2$ whose  output edge is  colored by $|b_2|$.
We define $a=(a_1,a_2,a_3,a_4) \in \tilde \Gamma_g$ as follows. 
\begin{itemize}
\item[(1)]If  the  output edge of the $\mux$  is connected to the left input edge of an  $\adx$ (resp.  the right input edge of an $\sadx$), then let $a$ be the element obtained from $b$ by applying $\stackrel{3}{\Rightarrow}$ to the  $\adx$ (resp.  $\stackrel{4}{\Rightarrow}$ to the $\sadx$) in $\lambda (b)$ as in Figure \ref{fig:ba2} (a) (resp. (b)).
\begin{figure}
\centering
\includegraphics[width=11cm,clip]{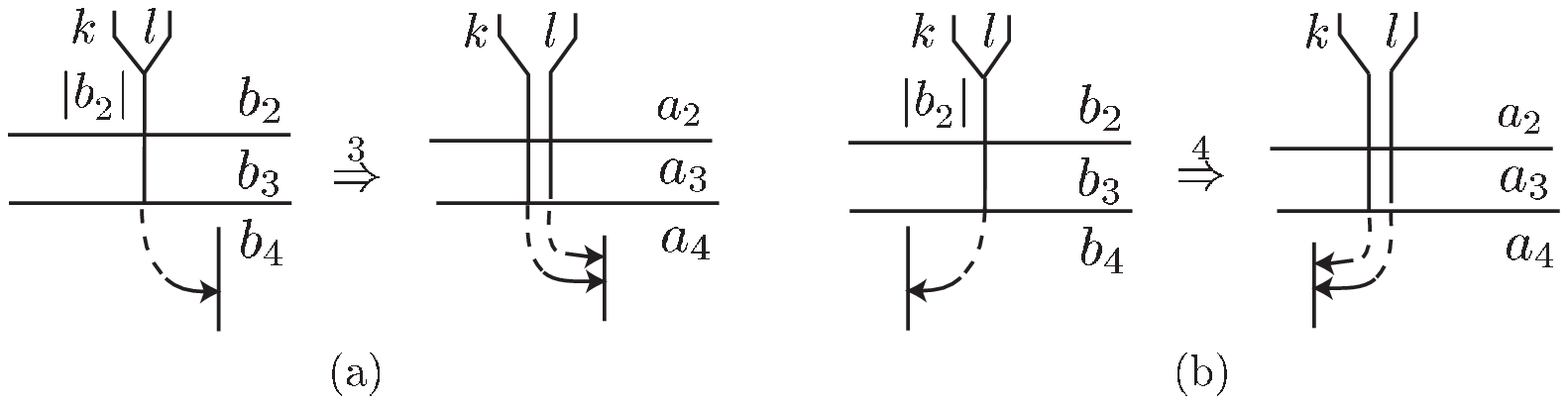}
\caption{How to obtain $a=(a_1,a_2,a_3,a_4) \in \tilde \Gamma_g$ from $b$, where   $k,l\geq 0$, $k+l+1=|b_2|$ }\label{fig:ba2}
\end{figure}
\item[(2)]If  the  output edge of the $\mux$    is connected to  the   left (resp.  right) input edge of a $\Yx$, then let $a$ be
the element obtained from $b$ by applying $\stackrel{5}{\Rightarrow}$ (resp. $\stackrel{6}{\Rightarrow}$) on the  $\Yx$  in $\lambda (b)$ as in  Figure \ref{fig:ba2b} (a) (resp. (b)).
\begin{figure}
\centering
\includegraphics[width=13cm,clip]{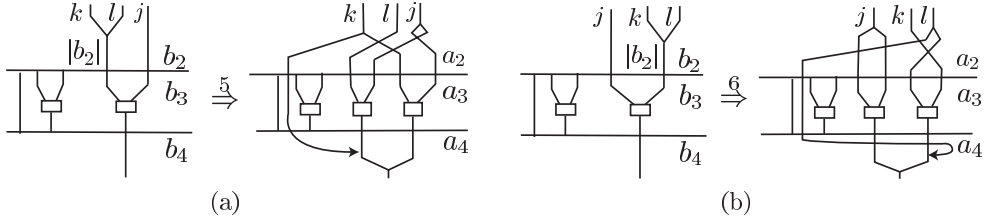}
\caption{How to obtain $a=(a_1,a_2,a_3,a_4) \in \tilde \Gamma_g$ from $b$, where  $j,k,l\geq 0$, $k+l+1=|b_2|$}\label{fig:ba2b}
\end{figure}
\end{itemize}
If $m(b_2)=1,$ then we have $|b_2|>|a_2|$. 
If $m(b_2)>1,$ then we have $|b_2|=|a_2|$ and $m(b_2)>m(a_2)$. 
This completes the proof.
\end{proof}

\begin{proof}[Proof of Claim 2]
We transform $b_2'\circ b_1'$  into   $b_2''\circ b_1''$ such that  $b''=(b''_1,b''_2,b_3',b_4')\in \tilde \Gamma_g'$ by  the two steps as in Figure \ref{fig:jt3}.
That is, 
\begin{itemize}
\item[(i)]we can  transform $b'_2$ into $\sigma \circ \Deltax^{[m_1,\ldots,m_l]}$ for some $ \sigma \in \Hom(\mathcal{A}_{\mathfrak{S}})$, $l\geq 0$, 
$m_1,\ldots,m_l\geq 1$ by using  $\stackrel{1}{\Rightarrow}$, and 
\item[(ii)] we can transform  $(\Deltax^{[m]}\otimes \Deltax^{[n]})\circ \Dx$,  $m,n\geq 1$,  into $ a\circ \Dx^{\otimes mn},$
for some $a\in \Hom_{\mathcal{A}_{\mu}}(2mn,m+n)$, by using $\stackrel{2}{\Rightarrow}$, $\stackrel{7}{\Rightarrow}$,  and $\stackrel{8}{\Rightarrow}$ as depicted in Figure \ref{fig:deld6}.
\begin{figure}
\centering
\includegraphics[width=13cm,clip]{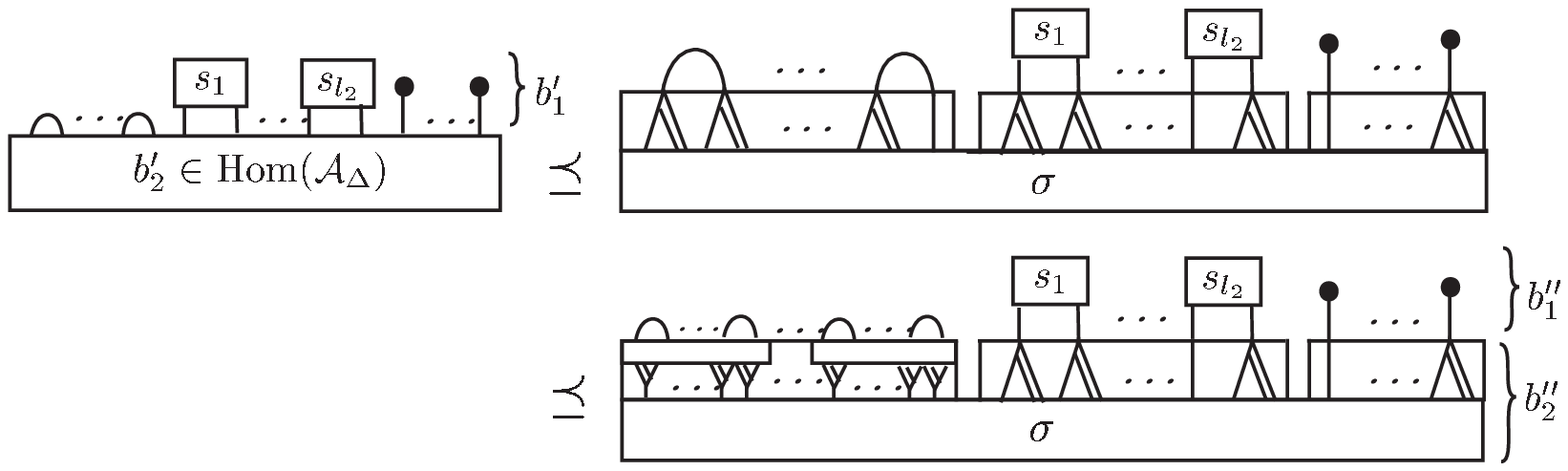}
\caption{How to transform $b_2'\circ b_1'$ to $b_2''\circ b_1''$}\label{fig:jt3}
\end{figure}
\end{itemize}
Hence we have the assertion.
\end{proof}
\begin{figure}
\centering
\includegraphics[width=12cm,clip]{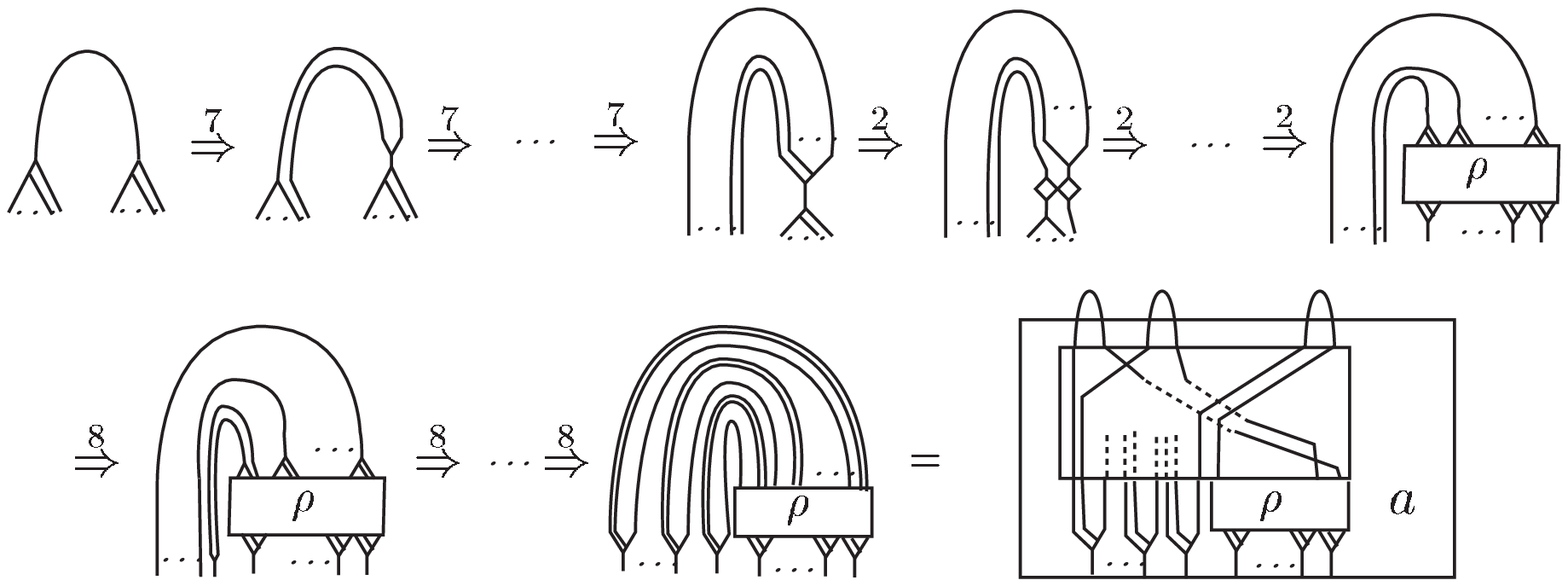}
\caption{How to transform $(\Deltax^{[m]}\otimes \Deltax^{[n]})\circ \Dx$ into $a\circ \Dx^{\otimes mn}$, where $ \rho\in \Hom(\mathcal{A}_{\mathfrak{S}})$ }\label{fig:deld6}
\end{figure}
\subsection{The functor $\mathcal{F}'$ and proof of $\mathcal{F}\big(\Gamma _g'\big)(1)\subset \mathcal{F}'\big(\Gamma _g'\big)(1)$}\label{f7}
In this section, we define the symmetric monoidal functor $\mathcal{F}'\co \mathcal{A}\rightarrow \mathcal{M}$ and prove $\mathcal{F}\big(\Gamma _g'\big)(1)\subset \mathcal{F}'\big(\Gamma _g'\big)(1)$.

For $n\geq 0$, we equip $U_h^{\hat \otimes n}$  with the topological $\mathbb{Z}^n$-graded algebra structure such that  
\begin{align*}
\deg(x_1\otimes \cdots\otimes x_n)=\big(|x_1|,\ldots,|x_n|\big),
\end{align*}
for  homogeneous elements  $x_1,\ldots,x_n\in U_h$ with respect to the topological $\mathbb{Z}$-grading of $U_h$ defined in Section \ref{preuni}.

For $k,l\geq 0$, we call a map $f\co U_h^{\hat \otimes k}\rightarrow U_h^{\hat \otimes l}$ \textit{homogeneous} if it sends each homogeneous element to a homogeneous element.
We call an object in $\mathcal{M}$ \textit{homogeneous} if it is generated by  homogeneous maps as a $\mathbb{Z}[q,q^{-1}]$-submodule of  $\Hom^{\c}_{\mathbb{Q}[[h]]}(U_h^{\hat \otimes k}, U_h^{\hat \otimes l})$.
Note that the image by the functor $\mathcal{F}$ of each generator morphism  in $\mathcal{A}$ in Section \ref{gene}, except for $\Deltax$,  is homogeneous.  

We define  $\mathcal{F}'$ in the same way as $\mathcal{F}$ except that we set
$\mathcal{F}'(\Deltax)=\sum_{j\in \mathbb{Z}}\Z\Delta_j$  instead of  $\mathcal{F}(\Deltax)=\Z\Delta$.
Here, for $j\in \mathbb{Z}$, $\Delta_j\co U_h\rightarrow U_h\hat \otimes U_h$  is  the continuous $\mathbb{Q}[[h]]$-linear map defined by
\begin{align*}
\Delta_j(x)= \sum x_{(1)}\otimes p _j (x_{(2)}),
\end{align*}
for $x\in U_h$,
where $p_j\co U_h\rightarrow U_h$ is the projection map defined by
\begin{align*}
p _j(y)=
\begin{cases}
y \quad  \text{if }|y|=j,
\\
0 \quad  \text{otherwise},
\end{cases}
\end{align*}
for $y\in U_h$ homogeneous.
Since  $\mathcal{F}'(\Deltax)$ is homogeneous, $\mathcal{F}'$ sends each generator morphism in $\mathcal{A}$ to a homogeneous module.
Moreover, since the compositions and  the tensor products of homogeneous objects in $\mathcal{M}$ are also homogeneous,
 the image by $\mathcal{F}'$ of  each morphism in $\mathcal{A}$ is homogeneous.

We prove $\mathcal{F}\big(\Gamma _g'\big)(1)\subset \mathcal{F}'\big(\Gamma _g'\big)(1).$
For  $x\in U_h$  a finite linear combination of homogeneous elements,
it is easy to check that  
\begin{align}
\Delta(x)&=\sum_{j\in \mathbb{Z}}\Delta_j(x),\label{homo1}
\\
\mathcal{F}(\Deltax)(x) &=\big(\Z\Delta\big)(x)\subset \Big(\sum_{j\in \mathbb{Z}} \Z \Delta_j\Big)(x)=\mathcal{F}'(\Deltax)(x). \label{homo2}
\end{align}
(In fact, (\ref{homo1}) is true for all $x\in U_h$.
However, (\ref{homo2})  is not,
since  $\sum_{j\in \mathbb{Z}} \Z \Delta_j$  consists of  finite linear combinations of $\Delta_j$ for $j\in \mathbb{Z}$.)

Note that each $\Deltax$ in a diagram of $B\in \Gamma' _g$ is contained in a $\Deltax^{[n]}\ex{m}$, in a $\Deltax^{[n]}\fx{m}$,
or in a $\Deltax^{[n]}\uqzx$  for $m,n\geq 0$. 
By (\ref{homo2}), we can prove that
\begin{align*}
\mathcal{F}(\Deltax^{[n]}\ex{m})(1)&\subset \mathcal{F}'(\Deltax^{[n]}\ex{m})(1),
\\
\mathcal{F}(\Deltax^{[n]}\fx{m})(1)&\subset \mathcal{F}'(\Deltax^{[n]}\fx{m})(1),
\end{align*}
for $m,n\geq 0$, by using induction on $n$. 
For $y\in U_h^0$, we have 
\begin{align*}
\mathcal{F}(\Deltax)(y)=\big(\Z\Delta_0\big)(y)=\mathcal{F}'(\Deltax)(y).
\end{align*}
Thus,  we have $\mathcal{F}\big(B\big)(1)\subset \mathcal{F}'\big(B\big)(1)$,
which completes the proof.

\subsection{The Subset $\Gamma''_g\subset \Gamma _g'$}\label{gr3}
In this section, we define  the subset $\Gamma _g''\subset \Gamma _g'$.

In what follows, we color each edge of a diagram of $B\in \Gamma _g'$ with $d,w,k$ or $\emptyset $ as follows.
First, we color the output edges of $\Dx$'s, $\Thetax{i}$'s, and $\uqzx$'s with $d,w$, and $k$, respectively.
Then, we color the  edges below as in Figure \ref{fig:colore} (a).
See Figure \ref{fig:colore}  (b) for an example of $G\in \Gamma _2'$ with the coloring.

\begin{figure}
\centering
\includegraphics[width=12cm,clip]{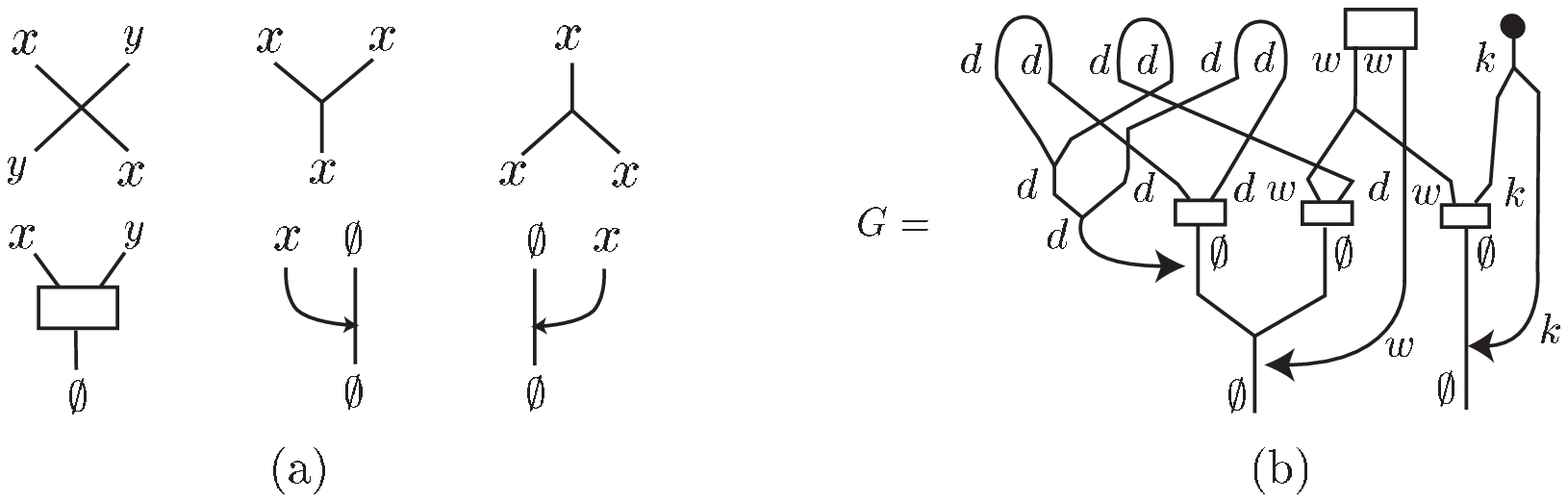}
\caption{(a) How to color the edges (b) An example of the coloring}\label{fig:colore}
\end{figure}

For $g\geq 0$, let  $\dot \Gamma''_g\subset \dot  \Gamma _g'$ be the subset consisting of  $b=(b_1,d,w,k, \sigma  ,b_3,b_4)$  such that
\begin{itemize}
\item[($C_{dk}$)] $d$ and $k$ are the identity morphisms in $\mathcal{A}$, 
\item[($C_{\ad}$)]  in $(\lambda \circ \kappa )(b)$, there is no  $\adx$ (resp. $\sadx$) with the $d$-colored left (resp. right) input edge, 
i.e.,  the first $l$ input edges of  $b_3=\id_A^{\otimes l}\otimes \Yx^{\otimes m}\otimes \uqzex^{\otimes n}$  are not  colored by $d$,
\item[($C_{Y}$)] there is no $\Yx$ with the left and right input edges both colored by $d$.
\end{itemize}
See Figure \ref{fig:dwk2} for an example of  $b_3\circ \sigma \circ(d\otimes w\otimes k)\circ b_1$.

Set
\begin{align*}
\tilde \Gamma''_g &=\kappa(\dot  \Gamma''_g)\subset \tilde \Gamma '_g,
\\
\Gamma''_g &=\lambda (\tilde \Gamma''_g)\subset \Gamma' _g.
\end{align*}
\begin{figure}
\centering
\includegraphics[width=7cm,clip]{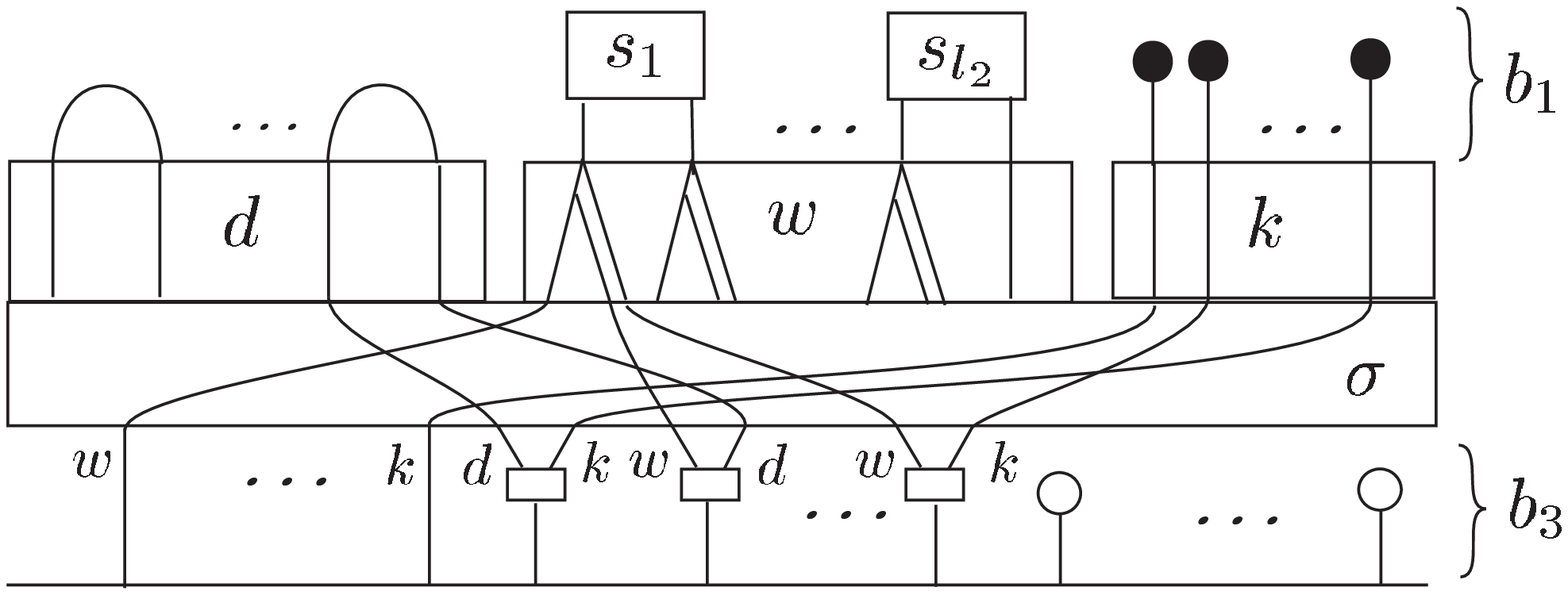}
\caption{An example  of $b_3\circ \sigma \circ(d\otimes w\otimes k)\circ b_1$ for $(b_1,d,w,k, \sigma  ,b_3,b_4)\in \dot  \Gamma _g''$}\label{fig:dwk2}
\end{figure}

\subsection{Proof of $\mathcal{F}'\big(\Gamma' _g\big)(1)\subset \mathcal{F}'\big(\Gamma'' _g\big)(1)$}\label{SB2}
Similarly to  Section \ref{SB}, we define a preorder $\preceq '$ on $\Gamma _g'$,
 and prove the following two lemmas, which imply  $\mathcal{F}'\big(\Gamma' _g\big)(1)\subset \mathcal{F}'\big(\Gamma'' _g\big)(1)$.
\begin{lemma}\label{co'}
For $B\preceq 'B'$ in $ \Gamma '_g$,  we have $\mathcal{F}'(B)(1)\subset \mathcal{F}'(B')(1)$.
\end{lemma}
\begin{lemma}\label{C}
For each element $B\in \Gamma'_g$, there exists $B'\in \Gamma''_g$, such that 
$B\preceq' B'$.
\end{lemma}
The preorder  $\preceq' $ on $\Gamma _g'$ is  generated by binary relations $\stackrel{i}{\Rightarrow}$ for $i=9,\ldots, 13$ on $\Gamma _g'$.
In the present case, we divide the definitions of  the binary relations into three.
Correspondingly, the proof of Lemma \ref{co'} is divided into that of Lemmas \ref{om1}, \ref{om2}, and \ref{om3}.

For $B\in \Gamma _g'$, let $N_{dk}(B)\geq 0$ be the number of the $\mux$'s colored by $d$ and the $\Deltax$'s colored by $k$,
$N_{\ad}(B)\geq 0$ the number of the $\adx$'s with $d$-colored left input edges and the $\sadx$'s with $d$-colored right input edges,
and $N_{Y}(B)\geq 0$ the number of the $\Yx$'s with  the left and right input edges both $d$-colored.
For example, for $G\in \Gamma _2'$ as in Figure \ref{fig:colore} (b), we have $N_{dk}(G)=3$, $N_{\ad}(G)=1$, and $N_{Y}(G)=1$.

Note that for $B\in \Gamma _g'$, we have $B\in \Gamma _g''$ if and only if $N_{dk}(B)=N_{\ad}(B)=N_{Y}(B)=0$.
By using inductions on $N_{\ad}(B),N_{Y}(B)$ and $N_{dk}(B)$,  Lemma \ref{C} follows from Lemmas \ref{ata1}, \ref{ata2}, and \ref{ata3}. 
\subsubsection{Binary relation $\stackrel{9}{\Rightarrow}$} 
Let $\stackrel{i}{\rightsquigarrow}$ for $i=1,\ldots,8$ be the local moves on diagrams of 
morphisms in $\mathcal{A}$ as depicted  in Figure \ref{fig:move21}, where in each relation,  the outsides of the two rectangles are the same.
For $B,B''\in \Gamma _g',$ we write $B\stackrel{9}{\Rightarrow} B'$ if there exists $B''\in \Hom(\mathcal{A})$ such that 
either $B\stackrel{1}{\rightsquigarrow}B''$ or $B\stackrel{2}{\rightsquigarrow}B''$, and there exists a sequence from $B''$ to $B'$ in $\Hom(\mathcal{A})$ of moves $\stackrel{i}{\rightsquigarrow}$ for $i=3,\ldots,8$.
\begin{figure}
\centering
\includegraphics[width=12cm,clip]{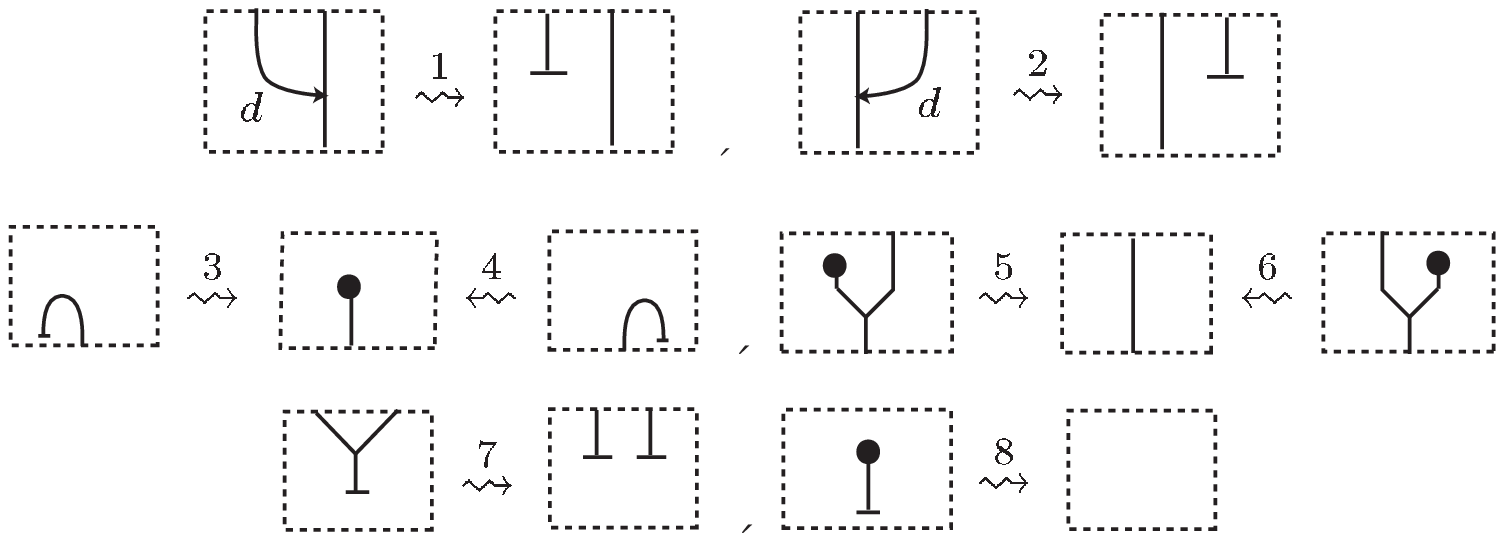}
\caption{Local moves $\stackrel{i}{\rightsquigarrow}$ for $i=1,\ldots,8$ }\label{fig:move21}
\end{figure}

\begin{lemma}\label{ata1}
For $B\in \Gamma _g'$ with $N_{\ad}(B)>0$, there exists $B'\in \Gamma _g'$ such that $N_{\ad}(B)>N_{\ad}(B')$ and  $B\stackrel{9}{\Rightarrow} B'$.
\end{lemma}
\begin{proof}
We can transform $B$ into  $B'\in \Gamma _g'$  as in Lemma \ref{ata1}  as follows.
Since $N_{\ad}(B)>0$, there exists $B''$ obtained from $B$ by applying  $\stackrel{1}{\rightsquigarrow}$ or $\stackrel{2}{\rightsquigarrow}$.
There is an   $\varepsilonx$ in $B''$, and we continue the transformation as follows.
\begin{itemize}
\item[(1)] If the $\varepsilonx $  is connected to the left (resp. right) output edge of a $\Dx$, then we apply $\stackrel{3}{\rightsquigarrow}$
 (resp. $\stackrel{4}{\rightsquigarrow}$).
If the new $\uqzx$ is connected to  the  left  (resp. right) input edge of a $\mux$, then we apply $\stackrel{5}{\rightsquigarrow}$ (resp. $\stackrel{6}{\rightsquigarrow}$),
otherwise  we put its  edge into the $k$-part.
\item[(2)] If the $\varepsilonx $  is connected to  an output edge of a $\mux$, then we apply $\stackrel{7}{\rightsquigarrow}$.
Then, for each new  $\varepsilonx$, we  continue the transformation starting from (1).
If there appears $(\varepsilonx\otimes \varepsilonx ) \circ \Dx$, then we apply $\stackrel{3}{\rightsquigarrow}$ or $\stackrel{4}{\rightsquigarrow}$,
and then we apply $\stackrel{8}{\rightsquigarrow}$.
\end{itemize}
For example, see Figure \ref{fig:lmove}, where a dotted circle with a number $i$ attached is a place to where we apply $\stackrel{i}{\rightsquigarrow}$.
It is easy to check that the procedure terminates, and the result $B'$ is contained in $\Gamma _g'$.
One can check that $N_{\ad}(B')=N_{\ad}(B)-1$.
\end{proof}

\begin{figure}
\centering
\includegraphics[width=11cm,clip]{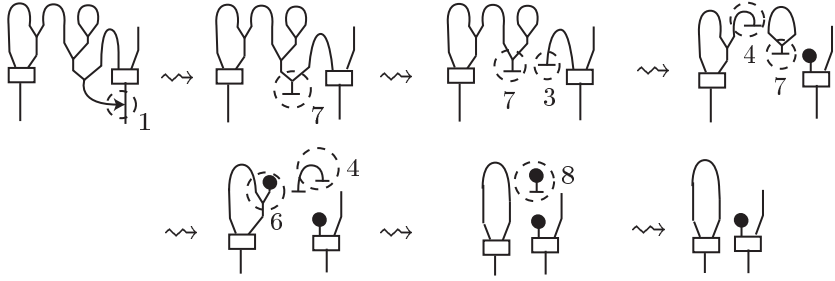}
\caption{Binary relation $\stackrel{10}{\Rightarrow}$}\label{fig:lmove}
\end{figure}

\begin{lemma}\label{om1}
For $B\stackrel{9}{\Rightarrow} B'$ in $\Gamma _g'$ , we have  $\mathcal{F}'\big(B\big)(1)\subset \mathcal{F}'\big(B'\big)(1)$.
\end{lemma}
\begin{proof}
It is enough to prove that, for $C\stackrel{j}{\rightsquigarrow}C'$  with  $j\in\{1,\ldots,8\} $ in the sequence of the local moves  from $B$ to $B'$, we have $\mathcal{F}'(C)(1)\subset \mathcal{F}'(C')(1)$.

Consider the case  $j=1$. The case  $j=2$ is similar.
Recall from Section \ref{f7} that the image by $\mathcal{F}'$ of each morphism in $\mathcal{A}$ is homogeneous.
This implies that, for each $b\in \Hom_{\mathcal{A}}(I,A^{\otimes l})$, $l\geq 0$, the $\Z$-submodule $\mathcal{F}'(b)(1)$ of $U_h^{\hat \otimes l}$ is generated by homogeneous elements of $U_h^{\hat \otimes l}$. 
Thus, the case  $j=1$  follows from 
\begin{align*}
&\sum \ad (\uqz D'_{1,\pm}\cdots D'_{n,\pm}(D'_{\pm}D''_{\pm})^m\otimes x)\otimes \uqz D_{1,\pm}''\otimes \cdots \otimes \uqz D_{n,\pm}''
\\
\subset &x\otimes (\uqz)^{\otimes n}
\\
\subset &\sum (\varepsilon \otimes \id_{U_h})(\uqz D'_{1,\pm}\cdots D'_{n,\pm}(D'_{\pm}D''_{\pm})^m\otimes  x)\otimes \uqz D_{1,\pm}''\otimes \cdots \otimes \uqz D_{n,\pm}'',
\end{align*}
for $m,n\geq 0$ and $x\in U_h$ homogeneous,  where we set  $D^{\pm 1}=\sum D'_{i,\pm}\otimes D''_{i,\pm}$ for $1\leq i\leq n$.
Here, we use  from \cite[Lemma 5.2]{sakie} the identities 
\begin{align*}
&\sum \ad(D'_{\pm}\otimes x)\otimes D_{\pm}''=x\otimes K^{\pm |x|},
\\
& \sum \ad(D'_{\pm}D_{\pm}''\otimes  x)=q^{|x|^2}x,
\end{align*}
for $x\in U_h$ homogeneous.

The cases $j=3,4$ follow from
\begin{align*}
(\varepsilon \otimes \id_{U_h})\circ \Big((\uqz)^{\otimes 2} D^{\pm 1} \Big)=\uqz =(\id_{U_h}\otimes \varepsilon )\circ \Big((\uqz)^{\otimes 2}D^{\pm 1}\Big).
\end{align*} 

The other cases $j=5,6,7,8$ are clear.
Hence we have the assertion.
\end{proof}

\subsubsection{Binary relation $\stackrel{10}{\Rightarrow}$} 
Let $\stackrel{i}{\rightsquigarrow}$  for  $i=9,\ldots,16$ be the local moves as depicted in Figure \ref{fig:move212}, where in each relation,  the outsides of the two rectangles are the same.
Here, the bottom lines in $\stackrel{12}{\rightsquigarrow}$ is the bottom lines of the morphisms.
For $B,B'\in \Gamma _g',$ we write $B\stackrel{10}{\Rightarrow} B'$ if there exist $B''\in \Hom(\mathcal{A})$ such that 
$B\stackrel{9}{\rightsquigarrow}B''$ and there is a sequence from $B''$ to $B'$ in $\Hom(\mathcal{A})$ of moves $\stackrel{i}{\rightsquigarrow}$ for $i=3,\ldots,8, 10,\ldots,16$.
\begin{figure}
\centering
\includegraphics[width=11cm,clip]{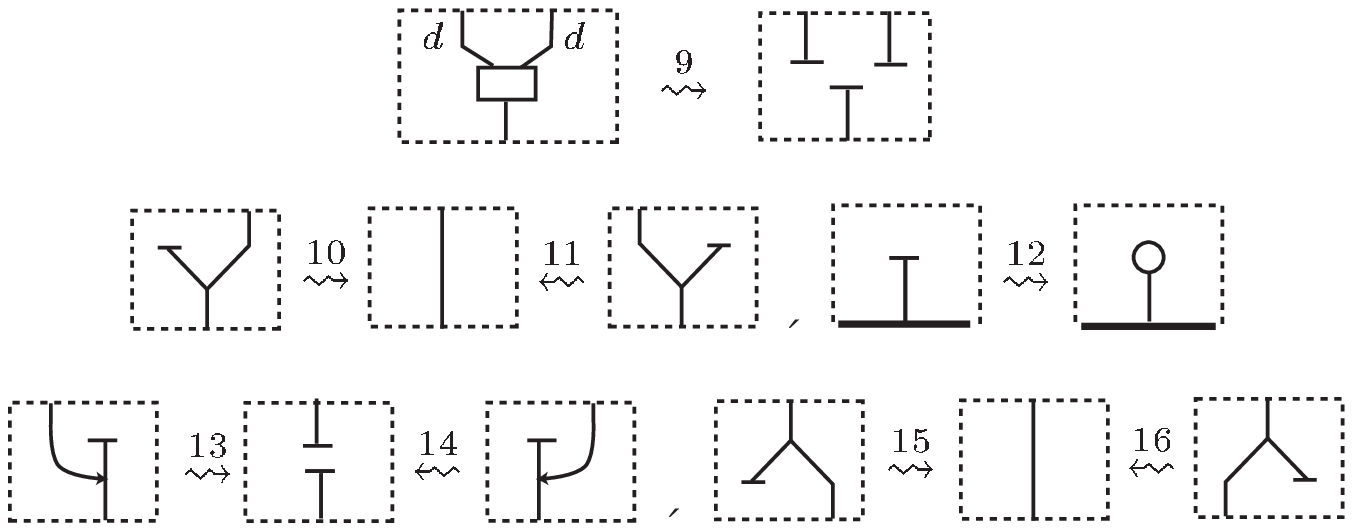}
\caption{Local moves $\stackrel{i}{\rightsquigarrow}$ for $i=9,\ldots,16$}\label{fig:move212}
\end{figure}

\begin{lemma}\label{ata2}
For $B\in \Gamma _g'$ with $N_{Y}(B)>0$ and $N_{\ad}(B)=0$, there exists $B'\in \Gamma _g'$ such that  $N_{Y}(B)>N_{Y}(B')$, $N_{\ad}(B')=0$,  and $B\stackrel{10}{\Rightarrow} B'$.
\end{lemma}
\begin{proof}
We can transform $B$ into  $B'\in \Gamma _g'$  as in Lemma \ref{ata2}  as follows.
Since $N_{Y}(B)>0$, there exists  $B''$ obtained from $B$ by applying  $\stackrel{9}{\rightsquigarrow}$.
For the two $\varepsilonx$ in $B''$, we apply  the local moves $\stackrel{i}{\rightsquigarrow}$ for $i=3,\ldots,8$ as in the proof of Lemma \ref{ata1}.
For the  $\etax$ in $B''$, we continue the transformation as follows.
\begin{itemize}
\item[(1)] If the $\etax $ is connected to the left (resp. right) input edge of  a $\mux$, then we apply $\stackrel{10}{\rightsquigarrow}$
 (resp. $\stackrel{11}{\rightsquigarrow}$).
\item[(2)] If the $\etax $ is connected to the bottom of the diagram, then we  replace the $\etax $ with $\uqzex$ by using  $\stackrel{12}{\rightsquigarrow}$.
\item[(3)] If the $\etax $ is connected to the right (resp. left) input edge of an $\adx$ (resp. $\sadx$),  then we apply $\stackrel{13}{\rightsquigarrow}$ 
(resp. $\stackrel{14}{\rightsquigarrow}$).
Then, there appears an $\etax $ and  an $\varepsilonx$. For the $\etax $, we continue the transformation starting from (1).  
If the $\varepsilonx $  is  colored by $d$,  then we apply $\stackrel{i}{\rightsquigarrow}$ for $i=3,\ldots, 8$ as in the proof of Lemma \ref{ata1}.
Let us assume that the $\varepsilonx $ is colored by  $w$ or $k$. 
By Condition A in the definition of $\tilde \Gamma _g$, the $\varepsilonx $ cannot be connected  directly to any output edge of the $\Thetax{i}$'s.
Hence the $\varepsilonx $ is connected to either  an output edge  of  a $\Deltax$ or the output edge of a $\uqzx$. 
If the $\varepsilonx $ is connected to  the  left (resp. right)  output edge a $\Deltax$, then we apply $\stackrel{15}{\rightsquigarrow}$ 
 (resp. $\stackrel{16}{\rightsquigarrow}$).
If the $\varepsilonx $ is connected to a $\uqzx$, we apply $\stackrel{8}{\rightsquigarrow}$.
\end{itemize}
For example, see Figure \ref{fig:lmove2}.
It is easy to check that the procedure terminates, and the result $B'$ is contained in $\Gamma _g'$.
One can check that $N_{Y}(B')=N_{Y}(B)-1$ and $N_{\ad}(B')=N_{\ad}(B)=0$.
\end{proof}

\begin{figure}
\centering
\includegraphics[width=11cm,clip]{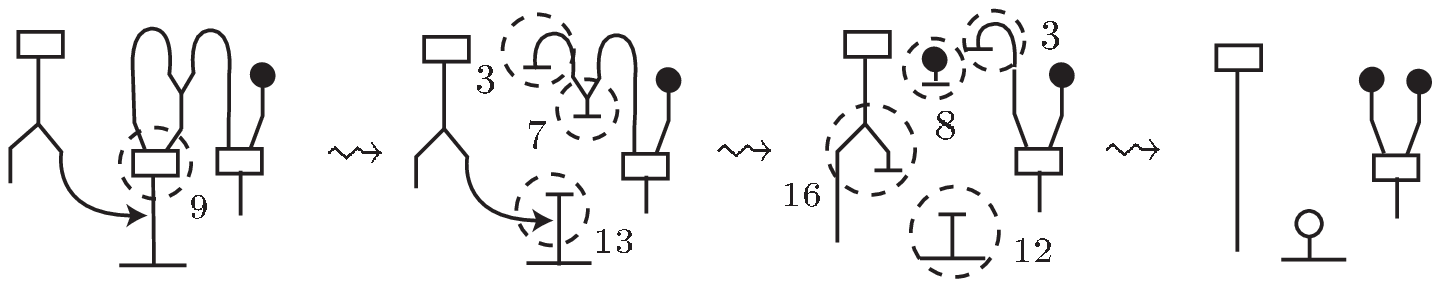}
\caption{Binary relation $\stackrel{10}{\Rightarrow}$}\label{fig:lmove2}
\end{figure}

\begin{lemma}\label{om2}
For $B\stackrel{10}{\Rightarrow} B'$ in $\Gamma _g'$, we have  $\mathcal{F}'\big(B\big)(1)\subset \mathcal{F}'\big(B'\big)(1)$.
\end{lemma}
\begin{proof}
It is enough to prove that, for $C\stackrel{j}{\rightsquigarrow}C'$  with  $j\in\{9,\ldots,16\} $ in the sequence of the local moves  from $B$ to $B'$, we have $\mathcal{F}'(C)(1)\subset \mathcal{F}'(C')(1)$.

The case $j=9$ follows from  Lemma \ref{y3}.

The case $j=15,16$ follows from 
\begin{align*}
\Big((\varepsilon \otimes \id_{U_h})\circ \Delta _k\Big)(\f{l}\uqz \e{m})=&
\begin{cases}
\f{l}\uqz \e{m} \quad \text{if }k=m-l,
\\
0 \quad \text{otherwise},
\end{cases}
\\
\Big((\id_{U_h}\otimes \varepsilon )\circ \Delta _k\Big)(\f{l}\uqz \e{m})=&
\begin{cases}
\f{l}\uqz \e{m}\quad \text{if }k=0,
\\
0 \quad \text{otherwise},
\end{cases}
\end{align*}
respectively, for $k,l,m\geq 0$.

The other cases $j=10,\ldots,14$ are clear. Hence we have the assertion.
\end{proof}
\subsubsection{Binary relation $\stackrel{i}{\Rightarrow}$ for $i=11,12,13$} 
Let $\stackrel{i}{\Rightarrow}$ for $i=11,12,13$ be the binary relations on $\Gamma' _g$ defined by the local moves on diagrams  as in Figure \ref{fig:move33},
where   in each relation,  the outsides of the two rectangles are the same.
It is easy to check that $\Gamma' _g$ is closed under $\stackrel{i}{\Rightarrow}$  for $i=11,12,13$. 

\begin{figure}
\centering
\includegraphics[width=13cm,clip]{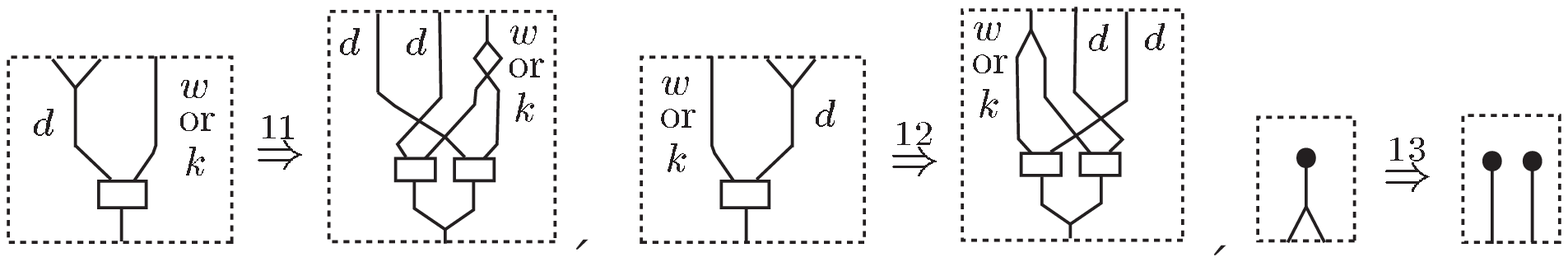}
\caption{Local moves $\stackrel{i}{\Rightarrow}$ for $i=11, 12,13$}\label{fig:move33}
\end{figure}

\begin{lemma}\label{ata3}
For $B\in \Gamma _g'$ with $N_{dk}(B)>0, N_{\ad}(B)=N_{Y}(B)=0$, there exists $B'\in \Gamma _g'$ such that $N_{dk}(B)>N_{dk}(B')$, $N_{\ad}(B')=N_{Y}(B')=0$, and   $B\stackrel{i}{\Rightarrow} B'$ with $i\in \{11,12,13\}$.
\end{lemma}
\begin{proof}
Since $N_{dk}(B)>0 $ and $N_{Y}(B)=0$, there is  a part in $B$ as in the left hand side of $\stackrel{i}{\Rightarrow}$ with $i\in \{11,12,13\}$.
We can obtain $B'$  as in Lemma \ref{ata3} from $B$ as follows.
If there is a $\Deltax\circ \uqzx$, then we obtain $B'$ from $B$ by applying $\stackrel{13}{\Rightarrow}$.
If there is no $\Deltax\circ \uqzx$, then we obtain $B'$ from $B$ by applying  $\stackrel{11}{\Rightarrow}$ or $\stackrel{12}{\Rightarrow}$,
and then  applying  $\stackrel{13}{\Rightarrow}$ if necessary.
\end{proof}
\begin{lemma}\label{om3}
For $B\stackrel{i}{\Rightarrow} B'$ in $\Gamma _g'$ with $i\in \{11,12,13\}$, we have  $\mathcal{F}'\big(B\big)(1)\subset \mathcal{F}'\big(B'\big)(1)$.
\end{lemma}
\begin{proof}
Consider the case $i=11$. The case $i=12$ is similar.
We can prove  the assertion by  two steps as in Figure \ref{fig:move11}, i.e., we have $F'(C)(1)\subset F'(C')(1)$ for  $C\stackrel{5}{\Rightarrow }C'$ in $\Gamma _g'$ by
(\ref{y21}) and (\ref{homo2}), and we have
\begin{figure}
\centering
\includegraphics[width=8cm,clip]{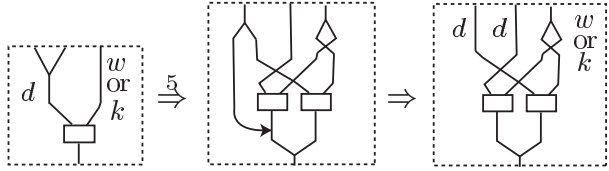}
\caption{Graphical proof of the case $i=11$}\label{fig:move11}
\end{figure}
\begin{align*}
\sum \ad \Big((\uqz D'_{1,\pm}&\cdots D'_{n,\pm}(D'_{\pm}D''_{\pm})^m)_{(1)}\otimes x\Big)
\\ &\otimes (\uqz D'_{1,\pm}\cdots D'_{n,\pm}(D'_{\pm}D''_{\pm})^m)_{(2)} \otimes \uqz D_{1,\pm}''\otimes \cdots \otimes \uqz D_{n,\pm}''
\\
\subset x\otimes \uqz D'_{1,\pm}\cdots &D'_{n,\pm}(D'_{\pm}D''_{\pm})^m \otimes \uqz D_{1,\pm}''\otimes \cdots \otimes \uqz D_{n,\pm}'',
\end{align*}
for $m,n\geq 0$.

The case $\stackrel{13}{\Rightarrow }$ follows from $\Delta(\uqz)\subset (\uqz)^{\otimes 2}$. Hence we have the assertion.
\end{proof}
\subsection{Modification of the elements in $\Gamma _g''$}\label{ep}
In the next section, we prove  $\mathcal{F}'\big(\Gamma'' _g\big)(1)\subset (\uqe)^{\otimes g}$,
which completes the proof of the sequence (\ref{naga}).
Before that, we modify the elements in $\Gamma _g''$.

First of all, we define notations.
For $m\geq 0,  n\geq 1$, set
\begin{align*}
\mathcal{I}(m,n)&=\{ (i_1,\ldots,i_n) \ | \ i_1,\ldots,i_n\geq 0, i_1+\cdots +i_n=m \}.
\end{align*}
For $\mathbf{i}=(i_1,\ldots,i_n)\in \mathcal{I}(m,n)$,  set
\begin{align*}
&\tilde E^{\mathbf{i}}=(\uqz)^{\otimes n}(\e{i_1}\otimes \cdots\otimes  \e{i_n})\subset (\uqzq)^{\otimes n},
\\
&\tilde F^{\mathbf{i}}=(\uqz)^{\otimes n}(\f{i_1}\otimes \cdots\otimes  \f{i_n})\subset (\uqzq)^{\otimes n},
\\
&\Ex{\mathbf{i}}=\ex{i_1}\otimes \cdots\otimes  \ex{i_n}\in \Hom_{\mathcal{A}}(I,A^{\otimes n}),
\\
&\Fx{\mathbf{i}}=\fx{i_1}\otimes \cdots\otimes  \fx{i_n} \in \Hom_{\mathcal{A}}(I,A^{\otimes n}).
\end{align*}
Clearly, we have 
\begin{align*}
\tilde E^{\mathbf{i}}=\mathcal{F}'(\Ex{\mathbf{i}})(1),\quad \tilde F^{\mathbf{i}}=\mathcal{F}'(\Fx{\mathbf{i}})(1).
\end{align*}

We use the following lemma.
\begin{lemma}\label{df1}
For $m\geq 0$, $n\geq 1$, we have
\begin{align*}
\mathcal{F}'\big(\Deltax ^{[n]}\circ \ex{m}\big)(1)\subset
\sum_{\mathbf{i}\in \mathcal{I}(m,n)} \mathcal{F}'\big(\Ex{\mathbf{i}}\big)(1),
\\
\mathcal{F}'\big(\Deltax ^{[n]}\circ \fx{m}\big)(1)\subset 
\sum_{\mathbf{i}\in \mathcal{I}(m,n)} \mathcal{F}'\big(\Fx{\mathbf{i}}\big)(1).
\end{align*}
\end{lemma}
\begin{proof}
The assertion follows from (\ref{De1}) and (\ref{De2}), by using  induction on $n\geq 1$. 
\end{proof}

Let $B=b_4\circ b_3\circ b_2\circ b_1$ with $b=(b_1,b_2,b_3,b_4)\in \tilde \Gamma _g''$.
By the condition ($C_{dk}$) in the definition of $\Gamma _g''$, we can
write $b_2\circ b_1=\sigma \circ \tilde b_1$ with $\sigma \in \Hom(\mathcal{A}_{\mathfrak{S}})$ and 
\begin{align} 
\begin{split}\label{bdel3}
\tilde b_1&=\Dx ^{\otimes l_1}\otimes \Big(\bigotimes_{p=1}^{l_2} \big(\Deltax^{[m_p, n_p]}\big)\circ \Thetax{s_p}\Big)\otimes \uqzx ^{\otimes l_3},
\end{split}
\end{align}
for   $l_1,l_2,l_3\geq 0$,  $ s_1,\ldots, s_{l_2}\geq 0$, $m_1,\ldots,m_{l_2},n_1,\ldots, n_{l_2}\geq 1$.
Note that 
\begin{align*}
 \Big(\bigotimes_{p=1}^{l_2}\big(\Deltax^{[m_p, n_p]}\big)\circ \Thetax{s_p}\Big)=\bigotimes_{p=1}^{l_2}\Big( \qintx{s_p}\otimes \big(\Deltax^{[m_p]}\circ \f{s_p}\big)\otimes\big( \Deltax^{[n_p]}\circ \e{s_p}\big)\Big).
\end{align*}

Let us start the modification of $B$.
For $\mathbf{i}_p\in \mathcal{I}(s_p,m_p)$ and $\bar {\mathbf{i}}_p\in \mathcal{I}(s_p,n_p) $, $p=1,\ldots, l_2,$ set
\begin{align}
\begin{split}\label{bdel}
\tilde{b}_1{(\mathbf{i}_1, \bar{ \mathbf{i}}_1, \ldots , \mathbf{i}_{l_2}, \bar{ \mathbf{i}}_{l_2})}&= \Dx ^{\otimes l_1}\otimes \Bigg( \bigotimes_{p=1}^{l_2} \Big(\qintx{s_p}\otimes 
\Fx{\mathbf{i}_p}\otimes \Ex{\bar{ \mathbf{i}}_p}\Big)\Bigg)\otimes \uqzx ^{\otimes l_3}.
\end{split}
\end{align}
In other words,  $\tilde{b}_1{(\mathbf{i}_1, \bar{ \mathbf{i}}_1, \ldots , \mathbf{i}_{l_2}, \bar{ \mathbf{i}}_{l_2})}$ is obtained from $\tilde b_1$
 by replacing the $\Deltax^{[m_p]}\circ \f{s_p}$ with a $\Fx{\mathbf{i}_p}$, and the $\Deltax^{[n_p]}\circ \e{s_p}$  with a $\Ex{\bar {\mathbf{i}}_p}$, 
 for  $p=1,\ldots, l_2$, see Figure \ref{fig:dwk5}.
\begin{figure}
\centering
\includegraphics[width=12.5cm,clip]{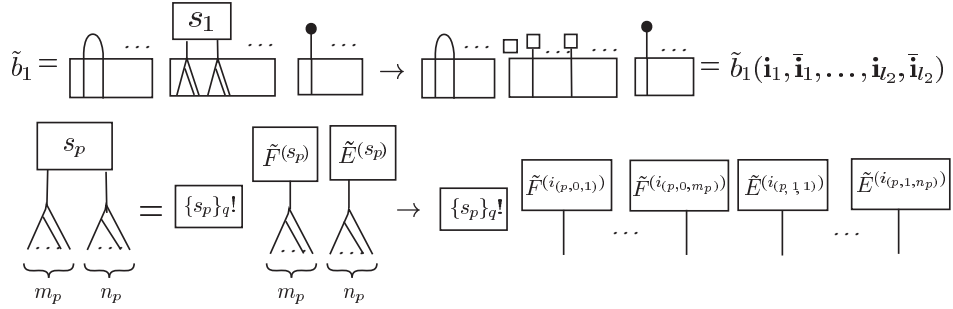}
\caption{How to modify $\tilde b_1$ with respect to $\mathbf{i}_p=\big(i_{(p,0,1)},\ldots, i_{(p,0,m_p)}\big)\in \mathcal{I}(s_p,m_p), \bar {\mathbf{i}}_p=\big(i_{(p,1,1)},\ldots, i_{(p,1,n_p)}\big)\in \mathcal{I}(s_p,n_p)$  for  $p=1,\ldots, l_2$
}\label{fig:dwk5}
\end{figure}

Thus, we obtain the modification $b_4\circ b_3\circ b_2\circ \tilde{b}_1{(\mathbf{i}_1, \bar{ \mathbf{i}}_1, \ldots , \mathbf{i}_{l_2}, \bar{ \mathbf{i}}_{l_2})}$
of $B$ with respect to $b\in \tilde \Gamma _g''$ and $\mathbf{i}_p\in \mathcal{I}(s_p,m_p),$ $\bar {\mathbf{i}}_p\in \mathcal{I}(s_p,n_p),$ for  $p=1,\ldots,l_2.$

By Lemma \ref{df1}, we have
\begin{align}
\begin{split}\label{u3}
\mathcal{F}'(B)(1)\subset \sum_{\substack{\mathbf{i}_p\in \mathcal{I}(s_p,m_p)
\\
\bar {\mathbf{i}}_p\in \mathcal{I}(s_p,n_p)
\\
p=1,\ldots,l_2}}\mathcal{F}'\big(b_4\circ b_3\circ \sigma \circ  \tilde{b}_1{(\mathbf{i}_1, \bar{ \mathbf{i}}_1, \ldots , \mathbf{i}_{l_2}, \bar{ \mathbf{i}}_{l_2})}\big)(1).
\end{split}
\end{align}
\subsection{Proof of $\mathcal{F}'\big(\Gamma'' _g\big)(1)\subset (\uqe)^{\otimes g}$}\label{prBB}
We prove the  inclusion $\mathcal{F}'\big(\Gamma'' _g\big)(1)\subset (\uqe)^{\otimes g}$.
Let $B=b_4\circ b_3\circ b_2\circ b_1$ with $(b_1,b_2,b_3,b_4)\in \Gamma _g''$ such that $b_2\circ b_1=\sigma \circ \tilde b_1$ with $\sigma \in \Hom(\mathcal{A}_{\mathfrak{S}})$ and $\tilde b_1$ as in (\ref{bdel3}).
By (\ref{u3}), it is enough to prove 
\begin{align}\begin{split}\label{B2}
\mathcal{F}'\big(b_4\circ b_3\circ \sigma \circ \tilde{b}_1{(\mathbf{i}_1, \bar{ \mathbf{i}}_1, \ldots , \mathbf{i}_{l_2}, \bar{ \mathbf{i}}_{l_2})}\big)(1)\subset (\uqe)^{\otimes g},
\end{split}
\end{align}
for $\mathbf{i}_p\in \mathcal{I}(s_p,m_p),$ $\bar {\mathbf{i}}_p\in \mathcal{I}(s_p,n_p),$  $p=1,\ldots,l_2.$

We prove (\ref{B2}). 
First, we study $\mathcal{F}'\big(b_3\circ \sigma \circ \tilde{b}_1{(\mathbf{i}_1, \bar{ \mathbf{i}}_1, \ldots , \mathbf{i}_{l_2}, \bar{ \mathbf{i}}_{l_2})}\big)(1).$
Fix
\begin{align*}
\mathbf{i}_p=\big(i_{(p,0,1)},\ldots, i_{(p,0,m_p)}\big)\in \mathcal{I}(s_p,m_p),\quad \bar {\mathbf{i}}_p=\big(i_{(p,1,1)},\ldots, i_{(p,1,n_p)}\big)\in \mathcal{I}(s_p,n_p),
\end{align*}
for $p=1,\ldots, l_2.$
On a diagram of  $\tilde{b}_1{(\mathbf{i}_1, \bar{ \mathbf{i}}_1, \ldots , \mathbf{i}_{l_2}, \bar{ \mathbf{i}}_{l_2})}$, we color  the output edge of the  $\fx{i_{(p,0,t)}}$ 
(resp. $\ex{i_{(p,1,u)}}$) with a label $(p,0,t)$ (resp.  $(p,1,u)$) for $t=1,\ldots,m_p$ (resp. $ u=1,\ldots,n_p$),  $ p=1,\ldots,l_2$,
see Figure \ref{fig:EF} (a).
We also color the output edges of the $\uqzx$'s in $\tilde b_1(\mathbf{i}_1, \bar{ \mathbf{i}}_1, \ldots , \mathbf{i}_{l}, \bar{ \mathbf{i}}_{l})$ with symbols $k_1,\ldots,k_{l_3}$ from the left to right,
see Figure \ref{fig:EF} (b).
Let $\mathcal{P}$ be the set of the all labels, i.e., set
\begin{figure}
\centering
\includegraphics[width=11cm,clip]{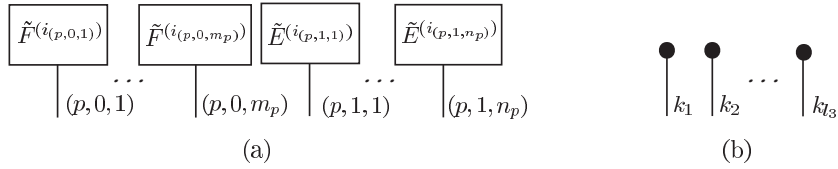}
\caption{How to color the output edges of  $\tilde b_1(\mathbf{i}_1, \bar{ \mathbf{i}}_1, \ldots , \mathbf{i}_{l_2}, \bar{ \mathbf{i}}_{l_2})$}\label{fig:EF}
\end{figure}
\begin{align*}
\mathcal{P}=&\{ (p,0,t) \ | \  1\leq t\leq m_p, 1\leq p\leq l_2\}
\sqcup \{(p,1,u) \ | \  1\leq u\leq n_p, 1\leq p\leq l_2 \}
\\
&\sqcup \{k_1,\ldots,k_{l_3}\}.
\end{align*}
In what follows, since $\mathcal{F}'(\uqzx)(1)=\mathcal{F}'(\ex{0})(1)=\uqz$, we identify $\uqzx$ with $\ex{0}$, and set $i_{k_j}=0$ for $j=1,\ldots,l_3$, see  figure \ref{fig:EF2}. We call the diagram of $\xx{i}{}$ for $i\geq 0$ with $X\in \{E,F\}$ an \textit{X-box}.
\begin{figure}
\centering
\includegraphics[width=4cm,clip]{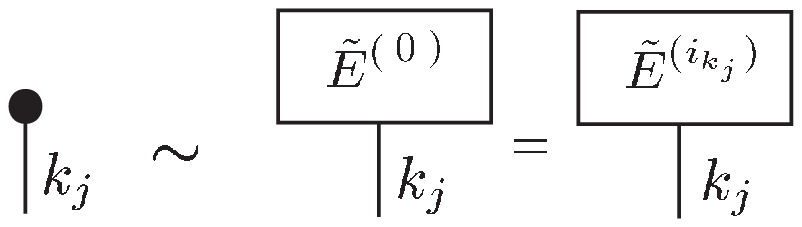}
\caption{How to treat the $j$th $\uqzx$ for $j=1,\ldots,l_3$}\label{fig:EF2}
\end{figure}
\begin{figure}
\centering
\includegraphics[width=12cm,clip]{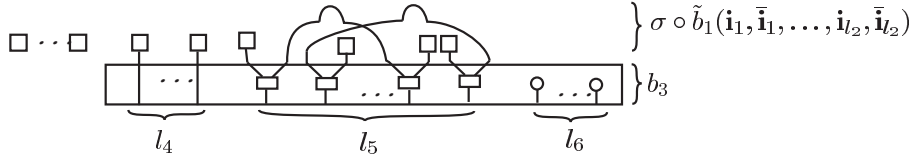}
\caption{How to arrange the diagram of $b_3\circ \sigma \circ \tilde b_1(\mathbf{i}_1, \bar{ \mathbf{i}}_1, \ldots , \mathbf{i}_{l}, \bar{ \mathbf{i}}_{l})$
}\label{fig:dwk212}
\end{figure}

 After we color the edges, we arrange the diagram of $b_3\circ \sigma \circ \tilde{b}_1{(\mathbf{i}_1, \bar{ \mathbf{i}}_1, \ldots , \mathbf{i}_{l_2}, \bar{ \mathbf{i}}_{l_2})}$
keeping $b_3$ so that each X-box is connected to $b_3$ directly without any crossings
 as in Figure  \ref{fig:dwk212}, where we set $b_3=\id_A^{\otimes l_4}\otimes \Yx^{\otimes l_5}\otimes \uqzex^{\otimes l_6}$, and  the floating boxes is the diagrams of $\qintx{s_p}$ for $p=1,\ldots,l_2$.
Here,  by the condition ($C_{\ad}$) in the definition of $\Gamma _g''$, the first $l_4$ input edges of $b_3$ are connected to X-boxes,
and  by the condition ($C_{Y}$), at least one of the input edges of each $\Yx$ in $b_3$ is connected to  an X-box.
Note that there are five cases as depicted in Figure \ref{fig:couple} (c1)--(c5), how a $\Yx$ is connected to the X-boxes and the $\Dx$'s.
 
Thus, we have
\begin{align}\label{b3}
b_3\circ \sigma \circ \tilde{b}_1{(\mathbf{i}_1, \bar{ \mathbf{i}}_1, \ldots , \mathbf{i}_{l_2}, \bar{ \mathbf{i}}_{l_2})}=
\bigotimes_{p=1}^{l_2}\qintx{s_p}\otimes  \xx{i_{a(1)}}{1}\otimes \cdots\otimes \xx{i_{a(l_4)}}{l_4}\otimes Z\otimes \uqzex^{\otimes l_6},
\end{align}
for  $a(1),\ldots, a(l_4)\in  \mathcal{P}$, $X_1,\ldots,X_{l_4}\in \{E,F\}$, and  $Z\in \Yx^{\otimes l_5}\circ \Hom_{\mathcal{A}}(I,A^{\otimes 2l_5})$.

For $j=1,\ldots,l_4$, we call the label $a(j)$  \textit{isolated}.
We say the labels $a$ and $b$ as in  Figure \ref{fig:couple} (c1)--(c5) are \textit{adjacent} to each other.
\begin{figure}
\centering
\includegraphics[width=10cm,clip]{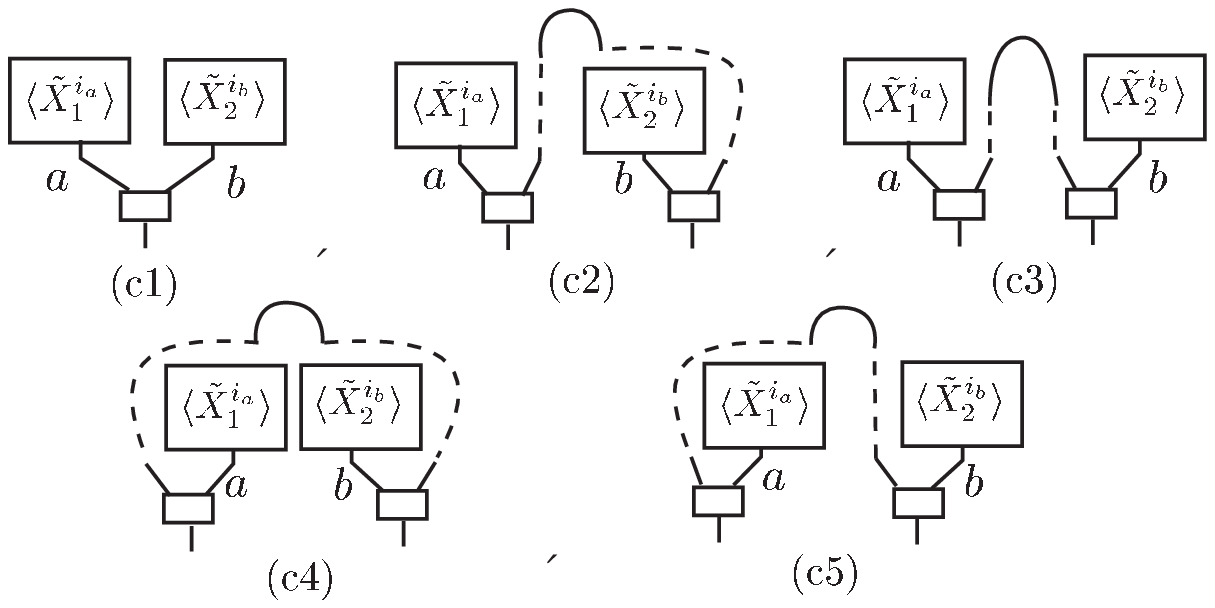}
\caption{The  $\Yx$'s in $b_3$, where $X_1,X_2\in \{E,F\}$
}\label{fig:couple}
\end{figure}

Note that the identity (\ref{b3}) implies 
\begin{align}\label{b31}
\mathcal{F}'\big(b_3\circ \sigma \circ \tilde{b}_1{(\mathbf{i}_1, \bar{ \mathbf{i}}_1, \ldots , \mathbf{i}_{l_2}, \bar{ \mathbf{i}}_{l_2})}\big)(1)\subset 
\Big(\prod_{p=1}^{l_2}\{s_p\}_q!\Big)\cdot \Big(\uqzq^{\otimes l_4}\otimes \mathcal{F}'(Z)(1)\otimes (\uqze)^{\otimes l_6}\Big).
\end{align}

Let us consider $\sum z_1\otimes\cdots \otimes  z_{l_5}\in \mathcal{F}'(Z)(1)$.  
If  the $m$th $\Yx$ (from the left in $b_3$) is as in (c1), then we can assume $z_m\in \dot Y(\uqz \tilde X_1^{(i_a)} \otimes  \uqz \tilde X_2^{(i_b)})$.
By Proposition \ref{qq}, we have
\begin{align}
&\dot Y(\uqz \tilde X_1^{(i_a)} \otimes  \uqz \tilde X_2^{(i_b)}) \subset (\{\min(i_a,i_b)\}_q!)^{-1}\cdot \uqe,\label{K1}
\end{align}
where  $(\{\min(i,j)\}_q!)^{-1}\cdot \uqe\subset \uqe\otimes _{\Z}\mathbb{Q}(q)$.
For example, we have
\begin{align*}
\dot Y(\uqz \tilde E^{(2)} \otimes  \uqz \tilde F^{(3)})&=  (\{2\}_q!)^{-1}\dot Y(\uqz e^2 \otimes  \uqz \tilde F^{(3)})
\\
&\subset (\{2\}_q!)^{-1} \uqe.
\end{align*}

If  the $m$th and $n$th $\Yx$'s  are as in (c2), then we can assume 
\begin{align*}
\sum z_{m}\otimes z_{n}\in\sum  \dot Y(\uqz\tilde X_1^{(i_a)} \otimes  \uqz D_{\pm}')\otimes \dot Y(\uqz \tilde X_2^{(i_b)} \otimes \uqz D_{\pm}'').
\end{align*}
By Lemma \ref{qd}, we have
\begin{align}
&\sum \dot Y(\uqz\tilde X_1^{(i_a)} \otimes  \uqz D_{\pm}')\otimes \dot Y(\uqz \tilde X_2^{(i_b)} \otimes \uqz D_{\pm}'')\subset  (\{\min(i_a,i_b)\}_q!)^{-1}\cdot (\uqe)^{\otimes 2}.
\end{align}
Similarly, for  (c3), (c4),  (c5), we have 
\begin{align}
&\sum \dot Y(\uqz \tilde X_1^{(i_a)} \otimes\uqz D_{\pm}')\otimes \dot Y(  \uqz D_{\pm}''\otimes  \uqz \tilde X_2^{(i_b)}) \subset (\{\min(i_a,i_b)\}_q!)^{-1}\cdot (\uqe)^{\otimes 2},
\\
&\sum \dot Y(\uqz D_{\pm}'\otimes \uqz \tilde X_1^{(i_a)})\otimes \dot Y(\uqz \tilde X_2^{(i_b)} \otimes \uqz  D_{\pm}'')\subset   (\{\min(i_a,i_b)\}_q!)^{-1}\cdot (\uqe)^{\otimes 2},\label{K2}
\\
&\sum \dot Y(\uqz D_{\pm}'\otimes \uqz \tilde X_1^{(i_a)} )\otimes \dot Y(\uqz  D_{\pm}''\otimes \uqz \tilde X_2^{(i_b)})\subset  (\{\min(i_a,i_b)\}_q!)^{-1}\cdot (\uqe)^{\otimes 2},
\end{align}
respectively. 

Let $\mathcal{P}_{\mathrm{A}}^2$ denote the set of unordered pairs $\{a,b\}$ of mutually adjacent elements $a,b\in  \mathcal{P}$.
By the above inclusions (\ref{K1})--(\ref{K2}), we have
\begin{align*}
\sum z_1\otimes \cdots\otimes z_{l_5 }\in \Big(\prod_{\{a,b\}\in \mathcal{P}_{\mathrm{A}}^2}\big(\big\{\min (i_{a},i_{b})\big\}_q!\big)^{-1} \Big)\cdot (\uqe)^{\otimes l_5}.
\end{align*}
Thus, by (\ref{b31}), we have
\begin{align}
\begin{split}\label{ab}
\mathcal{F}'\big(b_3\circ \sigma \circ \tilde b_1(\mathbf{i}_1, \bar{ \mathbf{i}}_1, \ldots , \mathbf{i}_{l}, \bar{ \mathbf{i}}_{l})\big)(1)&\subset I\cdot\Big( \uqzq^{\otimes l_4}\otimes (\uqe)^{\otimes l_5}\otimes (\uqze)^{\otimes l_6}\Big)
\\
&\subset I\cdot\Big( \uqzq^{\otimes l_4}\otimes (\uqe)^{\otimes l_5+l_6}\Big),
\end{split}
\end{align}
where we set 
\begin{align*}
I=\Big(\prod_{p=1}^{l_2}\{s_p\}_q! \Big)\cdot \Big(\prod_{\{a,b\}\in \mathcal{P}_{\mathrm{A}}^2}\big(\big\{\min (i_{a},i_{b})\big\}_q!\big)^{-1} \Big)\in \mathbb{Q}(q).
\end{align*}

Let us study $\mathcal{F}'(b_4)\Big( \uqzq^{\otimes l_4}\otimes (\uqe)^{\otimes l_5+l_6}\Big)$.
Since the first $l_4$ input edges of $b_4$ are connected to the left (resp. right) input edges of the $\adx$'s (resp. $\sadx$'s),
and the next  $l_5+l_6$ input edges of $b_4$ go down to the edges of the $\mux$'s and to the right (resp. left) input edges  of the $\adx$'s (resp. $\sadx$'s),
by  Proposition \ref{Habi}, (resp. Corollary \ref{Habii}) we have
\begin{align}\label{bb4}
\mathcal{F}'(b_4)\Big( \uqzq^{\otimes l_4}\otimes (\uqe)^{\otimes l_5+l_6}\Big)\subset  (\uqe)^{\otimes g}.
\end{align}
By (\ref{ab}) and  (\ref{bb4}),  we have
\begin{align*}
&\mathcal{F}'\big(b_4\circ b_3\circ \sigma \circ \tilde{b}_1{(\mathbf{i}_1, \bar{ \mathbf{i}}_1, \ldots , \mathbf{i}_{l_2}, \bar{ \mathbf{i}}_{l_2})}\big)(1) 
\subset
I\cdot (\uqe)^{\otimes g}.
\end{align*}
Thus, for the proof of (\ref{B2}), it is enough to  prove 
\begin{align}\label{u}
I\in \mathbb{Z}[q,q^{-1}].
\end{align}

For $k\geq 1$,  let $\Phi _k(q)$ be the $k$th cyclotomic polynomial in $q$.
For $f(q)\in \mathbb{Z}[q,q^{-1}]$, $f(q)\neq 0$, let $d_k\big(f(q)\big)$ be the largest integer $i$ such that $f(q)\in \Phi _k^i(q)\mathbb{Z}[q,q^{-1}]$.
Since both $\prod_{p=1}^{l_2} \{s_p\}_q!$ and $\prod_{\{a,b\}\in \mathcal{P}_{\mathrm{A}}^2 }\big\{\min (i_a,i_b)\big\}_q!$ are
products of the cyclotomic polynomials, in order to prove (\ref{u}), it is enough to prove
\begin{align}\label{uu}
d_k\Big(\prod_{p=1}^{l_2} \{s_p\}_q!\Big)\geq d_k\Big(\prod_{\{a,b\}\in \mathcal{P}_{\mathrm{A}}^2 } \big\{\min (i_a,i_b)\big\}_q!\Big),
\end{align}
for $k\geq 1$.

We prove (\ref{uu}). Fix $k\geq 1$.
Note that for $i\in \mathbb{Z}$, we have 
\begin{align*}
d_k\big( \{i\}_q\big)=d_k(q^i-1)=
\begin{cases}1\quad \text{if }k|i,
\\
0\quad \text{otherwise}.
\end{cases}
\end{align*}
This identity and  $s_p=\sum_{t=1}^{m_p}i_{(p,0,t)}=\sum_{u=1}^{n_p} i_{(p,1,u)}$ imply
\begin{align*}
d_k\big( \{s_p\}_q!\big)&\geq \sum_{t=1}^{m_p} d_k\big(\{i_{(p,0,t)}\}_q!\big),
\\
d_k\big( \{s_p\}_q!\big)&\geq \sum_{u=1}^{n_p}d_k\big(\{i_{(p,1,u)}\}_q!\big).
\end{align*}
Thus, we have
\begin{align*}
d_k\Big(\prod_{p=1}^{l_2} \{s_p\}_q!\Big)&\geq 
\sum_{p=1}^{l_2}\Big(\sum_{t=1}^{m_p} d_k\big(\{i_{(p,0,t)}\}_q!\big)+\sum_{u=1}^{n_p}d_k\big(\{i_{(p,1,u)}\}_q!\big)\Big)/2
\\
&=\sum_{a\in  \mathcal{P}} d_k\big( \{i_a\}_q!\big)/2
\\
&=\Big(\sum_{\{a,b\}\in  \mathcal{P}_{\mathrm{A}}^2}  d_k\big(\{i_a\}_q!\{i_b\}_q!\big)+\sum_{c\in  \mathcal{P}_{\mathrm{iso}}}d_k\big(  \{i_c\}_q!\big)\Big)/2
\\
&\geq \sum_{\{a,b\}\in  \mathcal{P}_{\mathrm{A}}^2}  d_k\big(\{i_a\}_q!\{i_b\}_q!\big)/2
\\
&= d_k\Big(\prod_{\{a,b\}\in  \mathcal{P}_{\mathrm{A}}^2}  \{i_a\}_q!\{i_b\}_q!\Big)/2
\\
&\geq d_k\Big(\prod_{\{a,b\}\in  \mathcal{P}_{\mathrm{A}}^2} \big\{\min (i_a,i_b)\big\}_q!\Big).
\end{align*}
Here $\mathcal{P}_{\mathrm{iso}}\subset \mathcal{P}$ denotes the subset consisting of isolated labels.
Hence we have (\ref{uu}), which completes the proof of $\mathcal{F}'\big(\Gamma'' _g\big)(1)\subset (\uqe)^{\otimes g}$.
\section{Completions}\label{completion}
In this section, we define the completion $\uqenh{n}$ of $(\uqe)^{\otimes n}$, and prove Theorem \ref{1}.
\subsection{Filtrations  of $\uqe$ with respect to $\ad$ and $\sad$}\label{Comp0}
For a subset $X\subset \bar  U_q^{\ev}$, let $\langle X \rangle_{\mathrm{ideal}} $ denote the two-sided ideal of $\uqe$ generated by $X$.
Set
\begin{align*}
&A_p=\langle U_{\mathbb{Z},q}\triangleright e^p  \rangle_{\mathrm{ideal}},
\\
&C_p=\langle 
\sum_{p'\geq p}(U_{\mathbb{Z},q}\tilde E^{(p')}\triangleright \bar U_q^{\ev}\big)
 \rangle_{\mathrm{ideal}},
\quad C'_p=\langle 
\sum_{p'\geq p}(U_{\mathbb{Z},q}\tilde F^{(p')}\triangleright \bar U_q^{\ev}\big)
 \rangle_{\mathrm{ideal}},
 \\
&\tilde C_p=\langle 
\sum_{p'\geq p}K(U_{\mathbb{Z},q}\tilde E^{(p')}\triangleright K\bar U_q^{\ev}\big)
 \rangle_{\mathrm{ideal}},
\quad \tilde C'_p=\langle 
\sum_{p'\geq p}K(U_{\mathbb{Z},q}\tilde F^{(p')}\triangleright K\bar U_q^{\ev}\big)
 \rangle_{\mathrm{ideal}},
\end{align*}
for $p\geq 0$. For $X=A,C,C',\tilde C,\tilde C'$,   the $X_p,p\geq 0$ form a decreasing filtration of $\uqe$, i.e., we have $X_p\supset X_{p+1}$ for $p\geq 0$.
\begin{lemma}[{\cite[Proposition 5.5]{sakie}}]\label{EQQ}
\begin{itemize}
\item[\rm{(i)}]
For $p\geq 0$, we have $
C_p=C'_p.$
\item[\rm{(ii)}]
For $p\geq 0$, we have $C_{2p}\subset A_p$.
\item[\rm{(iii)}]
If $p\geq 0$ is even, then we have $C_{2p}=A_p$.
\end{itemize}
\end{lemma}
\begin{lemma}\label{EQQ2}
\begin{itemize}
\item[\rm{(i)}]
For $p\geq 0$, we have $
\tilde C_p=\tilde C'_p.$
\item[\rm{(ii)}]
For $p\geq 0$, we have $\tilde C_{2p}\subset A_p$.
\item[\rm{(iii)}]
If $p\geq 0$ is odd, then we have $\tilde C_{2p}=A_p$.
\end{itemize}
\end{lemma}
\begin{proof}
The proof is almost the same as that of Lemma \ref{EQQ}.
\end{proof}For $p\geq 0$, set
 \begin{align*}G_p=C_p+\tilde C_p=C'_p+\tilde C'_p.
 \end{align*}
\begin{corollary}\label{c1}
For $p\geq 0$, we have 
$G_{2p}= A_p.$
\end{corollary}
\begin{proof}For $p\geq 0$,  by Lemma \ref{EQQ} (ii) and \ref{EQQ2} (ii), we have
$G_{2p}=C_{2p}+\tilde C_{2p}\subset A_p.$ 

If $p\geq 0$ is even, then by Lemma \ref{EQQ} (iii), we have $G_{2p}\supset C_{2p}= A_p.$

If $p\geq 0$ is odd, then by Lemma \ref{EQQ2} (iii), we have $G_{2p}\supset \tilde C_{2p}= A_p.$

Thus, we have the assertion.
\end{proof}
\begin{corollary}\label{c2}
The filtrations $\{A_p\}_{p\geq 0},\{C_p\}_{p\geq 0},\{\tilde C_p\}_{p\geq 0},\{G_p\}_{p\geq 0}$ are all cofinal with each other.
\end{corollary}
\subsection{Filtrations of $\uqe$ and $(\uqe)^{\otimes 2}$ with respect to $\dot Y$}
For $p\geq 0$, let $\mathcal{Y}_p$  be the two-sided ideal in $\uqe$ generated by the elements in 
\begin{align*}
&\sum_{p'\geq p}\dot Y(\uqz \e{p'} \otimes \uq)
, \ \sum_{p'\geq p}\dot Y(\uq \otimes \uqz \e{p'} ) 
, \ \sum_{p'\geq p}\dot Y(\uqz \f{p'} \otimes \uq)
, \ \sum_{p'\geq p}\dot Y(\uq \otimes \uqz \f{p'} ).
\end{align*}
\begin{lemma}\label{zf1}
For $p\geq 0$, we have $\mathcal{Y}_p \subset G_p.$
\end{lemma}
\begin{proof}
It is enough to prove that all the generators of $\mathcal{Y}_p$ are contained in $G_p$.

By (\ref{ad1}), (\ref{ad2}),  and (\ref{bibi}), we have
\begin{align*}
\dot Y(\uqz \e{p'}\otimes \uq)&\subset \sum_{i=0,1} \Big( \uqz\e{p'} \triangleright K^i\uqe \Big) K^i\uqe \subset C_p+ \tilde C_p \subset G_p,
\\
\dot Y(\uq\otimes \uqz \e{p'})&=\sum_{i=0,1} K^i\uqe\Big(S^{-1}(\uqz \e{p'}) \triangleright K^i\uqe \Big)
\\
&\subset \sum_{i=0,1} K^i\uqe\Big( \uqz \e{p'} \triangleright K^i\uqe \Big)\subset C_p+ \tilde C_p\subset  G_p,
\end{align*} 
for $p'\geq p$.
Similarly, we have
\begin{align*}
&\dot Y(\uqz \f{p'} \otimes \uq)\subset C_p'+ \tilde C_p'\subset G_p,\\
&\dot Y(\uq \otimes \uqz\f{p'} )\subset C_p'+ \tilde C_p'\subset G_p,
\end{align*}
for $p'\geq p$.
Hence we have the assertion.
\end{proof}
Let $(\mathcal{Y}^D)_p$ be the two-sided ideal in $(\uqe)^{\otimes 2} $ generated by the elements in
\begin{align*}
&\sum_{p'\geq p} \dot Y(\uqz \e{p'}\otimes \uqz D_{\pm}')\otimes \dot Y(\uq\otimes  \uqz D_{\pm}''), \  
\sum_{p'\geq p}  \dot Y(\uqz \f{p'}\otimes  \uqz D_{\pm}')\otimes \dot Y(\uq\otimes  \uqz D_{\pm}''),
\\
&\sum_{p'\geq p}  \dot Y(\uqz \e{p'} \otimes  \uqz D_{\pm}')\otimes \dot Y(\uqz D_{\pm}''\otimes\uq ), \ 
\sum _{p'\geq p}\dot Y(\uqz \f{p'}\otimes \uqz  D_{\pm}')\otimes \dot Y(\uqz D_{\pm}''\otimes\uq ),
\\
&\sum_{p'\geq p}  \dot Y(\uqz D_{\pm}'\otimes \uqz \e{p'})\otimes \dot Y(\uq \otimes \uqz D_{\pm}''), \ 
\sum_{p'\geq p}  \dot Y(\uqz D_{\pm}'\otimes \uqz \f{p'})\otimes \dot Y(\uq \otimes \uqz  D_{\pm}''),
\\
&\sum_{p'\geq p}  \dot Y( \uqz D_{\pm}'\otimes \uqz \e{p'})\otimes \dot Y( \uqz D_{\pm}''\otimes \uq), \ 
\sum_{p'\geq p}  \dot Y( \uqz D_{\pm}'\otimes \uqz \f{p'})\otimes \dot Y( \uqz D_{\pm}''\otimes \uq).
\end{align*}
Note that these sets are all contained in $(\uqe)^{\otimes 2}$ by Lemma \ref{qd}.
\subsection{Filtrations of $(\uqe)^{\otimes n}$}
For $n\geq 1$ and a filtration  $\{X_p\}_{p\geq 0}$ of $\uqe$, we define a filtration $\{X_p^{(n)}\}_{p\geq 0}$ of $(\bar U_q^{\ev})^{ \otimes n}$ by
\begin{align*}
X_p^{(n)}= \sum_{j=1}^n (\uqe)^{\otimes j-1} \otimes X_p \otimes (\uqe)^{\otimes n-j}.
\end{align*}
For $n\geq 1$, we define the filtration $\{(\mathcal{Y}^D)_p^{(n)}\}_{p\geq 0}$ of $(\bar U_q^{\ev})^{ \otimes n}$ by
\begin{align*}
(\mathcal{Y}^D)_p^{(n)}= \{ &\sum_{1\leq i<j\leq n} (\uqe)^{\otimes i-1} \otimes
 y' \otimes (\uqe)^{\otimes j-i-1}\otimes y'' \otimes(\uqe)^{\otimes n-j}\ | \ \sum y'\otimes y''\in (\mathcal{Y}^D)_p  \}
 \\
 +\{&\sum_{1\leq i<j\leq n} (\uqe)^{\otimes i-1} \otimes
 y'' \otimes (\uqe)^{\otimes j-i-1}\otimes y' \otimes(\uqe)^{\otimes n-j}\ | \ \sum y'\otimes y''\in (\mathcal{Y}^D)_p  \}
 \\
 +\{ &\sum_{k=1}^n (\uqe)^{\otimes k-1} \otimes 
 y'y''\otimes  (\uqe)^{\otimes n-k} \ | \ \sum y'\otimes y''\in (\mathcal{Y}^D)_p  \}
\\
+\{&\sum_{k=1}^n (\uqe)^{\otimes k-1} \otimes 
 y''y' \otimes (\uqe)^{\otimes n-k} \ | \ \sum y'\otimes y''\in(\mathcal{Y}^D)_p  \}.
\end{align*}
\begin{lemma}\label{zf2}
For $n\geq 1,p\geq 0$, we have $(\mathcal{Y}^D)_p^{(n)}\subset G_{\lfloor \frac{p}{2}\rfloor}^{(n)}.$
\end{lemma}
\begin{proof}
It is enough to prove that all the  generators of $(\mathcal{Y}^D)_p$ are contained in  $G_{\lfloor \frac{p}{2}\rfloor}^{(2)}$.

We prove 
\begin{align*}
\sum_{p'\geq p} \dot Y(\uqz \e{p'}\otimes \uqz D')\otimes \dot Y(\uqz D'' \otimes \uq )\subset G_{\lfloor \frac{p}{2}\rfloor}^{(2)}.
\end{align*}
Similarly for the other generators of $(\mathcal{Y}^D)_p$.
For $p\geq 0$, let us assume 
\begin{align}\label{yg}
\sum \dot Y(\uqz \e{p}\otimes D')\otimes \dot Y(D''\otimes \uq )\subset  G_{p}^{(2)}.
\end{align}
Then, similarly to (\ref{Dd}), for $p'\geq p,$ we have
\begin{align*}
&\sum \dot Y(\uqz \e{p'}\otimes \uqz D')\otimes \dot Y(\uqz D''\otimes  \uq)
\\
=&\sum \dot Y(\uqz \e{p'}\otimes D'\uqz )\otimes \dot  Y(\uqz D''\otimes  \uq)
\\
\subset &\sum\dot Y\big(\uqz (\e{p'})_{(1)}\otimes \uqz\big)\big(\dot Y\big(\uqz (\e{p'})_{(2)}\otimes D'\big)\triangleleft \uqz\big)\otimes 
\big(\uqz\triangleright \dot Y( D''\otimes \uq)\big)\dot Y( \uqz\otimes \uq)
\\
\subset &\sum\dot Y\big(\uqz (\e{p'})_{(1)}\otimes \uqz\big)\dot Y\big(\uqz (\e{p'})_{(2)}\otimes D'\big)\otimes 
 \dot Y( D''\otimes \uq)\dot Y( \uqz\otimes \uq)
\\
=&\sum _{p'_1+p'_2=p'}\dot Y(\uqz \e{p'_1}\otimes \uqz)\dot Y(\uqz \e{p'_2}\otimes  D')\otimes 
\dot Y( D''\otimes \uq)\dot Y( \uqz\otimes \uq)
\\
\subset &\sum_{p'_1+p'_2=p'}\Big(\mathcal{Y}_{p'_1}\cdot \dot Y(\uqz \e{p'_2}\otimes  D')\Big)\otimes 
\Big(\dot Y( D''\otimes \uq)\cdot \uqe\Big)
\\
\subset &\sum_{p'_1+p'_2=p'}\mathcal{Y}_{p'_1}^{(2)}\cdot \Big(\dot Y(\uqz \e{p'_2}\otimes  D')\otimes 
\dot Y( D''\otimes \uq)\cdot \uqe\Big)
\\
\subset &\sum_{p'_1+p'_2=p'}\mathcal{Y}_{p'_1}^{(2)}\cdot G_{p'_2}^{(2)}
\subset \sum_{p'_1+p'_2=p'}G_{p'_1}^{(2)}\cdot G_{p'_2}^{(2)}\subset  \sum_{p'_1+p'_2=p'} G_{\max( p_1',p'_2)}^{(2)}\subset G_{\lfloor \frac{p'}{2}\rfloor}^{(2)}\subset G_{\lfloor \frac{p}{2}\rfloor}^{(2)}.
\end{align*}

Now, we prove (\ref{yg}).
Similar to (\ref{ad3}), for $b\in \uq$ homogeneous, we have
\begin{align*}
\sum \dot Y(\uqz \e{ p}\otimes D')\otimes \dot Y(D'' \otimes b)= \sum (\uqz \e{ p}\triangleright K^{-|b_{(2)}|})K^{|b_{(2)}|}\otimes   S^{-1}(b_{(2)})b_{(1)},
\end{align*}
 with  $b_{(1)},b_{(2)}\in \uq$ homogeneous such that  $S^{-1}(b_{(1)})b_{(2)}\in \uqe$.  

If $|b_{(2)}|\in 2\mathbb{Z},$ then we have
\begin{align*}
(\uqz \e{ p}\triangleright K^{-|b_{(2)}|})K^{|b_{(2)}|}&
\subset C_{p}\subset G_{p}.
\end{align*}
If $|b_{(2)}|\in \mathbb{Z}\setminus 2\mathbb{Z},$ then we have
\begin{align*}
(\uqz \e{ p}\triangleright K^{-|b_{(2)}|})K^{|b_{(2)}|}\subset  \tilde C_{p}\subset G_{p}.
\end{align*}
Thus, we have
\begin{align*}
\sum (\uqz \e{ p}\triangleright K^{-|b_{(2)}|})K^{|b_{(2)}|}\otimes   S^{-1}(b_{(2)})b_{(1)}\subset  G_{p}  \otimes \uqe \subset G_{p}^{(2)}.
\end{align*}
Hence we have the assertion.
\end{proof}
\subsection{Completions}\label{Comp}
Let $(\bar U_q^{\ev})\,\hat  {}$  denote the completion of $\uqe$ in $U_h$ with respect to the filtration $\{G_p\}_{p\geq 0}$, i.e.,  
$(\bar U_q^{\ev})\,\hat  {}$ is the image of the map
\begin{align*}
\varprojlim _{p}  \big( \bar U_q^{\ev}/G_p\big) \rightarrow U_h
\end{align*}
induced by  the inclusion $\bar U_q^{\ev} \subset U_h$.
Since $G_{2p}=A_p\subset h^pU_h$, this map is well-defined.

Let $\uqenh{n}$ denote the completion of $ (\bar U_q^{\ev})^{ \otimes n}$  in $U_h^{\hat \otimes n}$ with respect to the filtration $\{G_p^{(n)}\}_{p\geq 0}$.
For $n=0$, it is natural to set 
\begin{align*}
G_p^{(0)}=\begin{cases}
\mathbb{Z}[q,q^{-1}] \quad \text{if} \ \ p=0,
\\
0 \quad  \quad \quad \quad  \text{if} \ \ p>0.
\end{cases}
\end{align*}
Thus, we have 
\begin{align*}
\uqenh{0}=\mathbb{Z}[q,q^{-1}].
\end{align*}
\subsection{Proof of Theorem \ref{1}}
Let $T\in BT_n$ be a boundary bottom tangle and $(\tilde T;g,g_1,\ldots,g_n)$ a  boundary data for $T$.
Let $C(\tilde T)=\{c_1,\ldots,c_{l}\}$ be the set  of crossings of the diagram of $\tilde T$ which we fix in the definition of $J_{\tilde T}$.
We fix these notations in this section.

By Proposition \ref{2}, we have 
\begin{align*}
\mu ^{[g_1,\ldots ,g_n]}\bar  Y^{\otimes g}(J_{\tilde T,s}) \in (\bar U_q^{\ev})^{\otimes n},
\end{align*}
for $s\in \mathcal{S}(\tilde T)$.
In this section, we prove Theorem \ref{1}, i.e., we prove
\begin{align*}
J_T=\sum_{s\in \mathcal{S}(\tilde T)}\mu ^{[g_1,\ldots ,g_n]}\bar  Y^{\otimes g}(J_{\tilde T,s}) \in \uqenh{n}.
\end{align*}
Since $\mu ^{[g_1,\ldots ,g_n]}\big(G_p^{(g)}\big)\subset G_p^{(n)}$ for $p\geq 0$, it is enough  to prove the following lemma.
\begin{lemma}\label{cm}
For each $p\geq 0$, there are only finitely many states $s\in \mathcal{S}(\tilde T)$ such that 
$\bar  Y^{\otimes g}(J_{\tilde T,s})\notin G_p^{(n)}$.
\end{lemma}

We use the setting in  in Section \ref{proof}, where we can consider a state $s\in \mathcal{S}(\tilde T)$ as a parameter.
\begin{lemma}\label{uun}
There is  a map  $B\co \mathcal{S}(\tilde T)\rightarrow \Gamma _g''$, $s\mapsto B^s$, satisfying the following  conditions:
For all $s\in \mathcal{S}(\tilde T)$,  we have  $\bar  Y^{\otimes g}(J_{\tilde T,s})\in \mathcal{F}'\big(B^s\big)(1)$, and  we have  $B^s =b_4\circ b_3\circ b_2\circ  b_1^s$ with $(b_1^s, b_2,b_3,b_4)\in \tilde \Gamma _g''$ such that  $b_2\circ b_1^s=\sigma \circ \tilde b_1^s$ with $\sigma \in \Hom(\mathcal{A}_{\mathfrak{S}})$ and $\tilde b_1^s$ as in (\ref{bdel3}) replacing $s_p$ with $s(c_p)$ for $p=1,\ldots,l_2$, where
any part of $B^s$ except $\Thetax{s(c_1)}\otimes \ldots\otimes \Thetax{s(c_{l_2})}$ in  $\tilde b_1^s$ does not depend on $s\in \mathcal{S}(\tilde T)$.
\end{lemma}
\begin{proof}
We can define $B$ as in Lemma \ref{uun} by constructing  $B^s$   as follows. First, we choose a  state $x \in \mathcal{S}(\tilde T)$,  and  construct $B^{x}_0\in \Gamma _g$ so that $\bar  Y^{\otimes g}(J_{\tilde T,x})\in \mathcal{F}\big(B^{x}_0\big)(1)$ as in the proof of Lemma \ref{A}.
By the definition, we have $\bar  Y^{\otimes g}(J_{\tilde T,s})\in \mathcal{F}\big(B^s_0\big)(1)$ for all state $s\in \mathcal{S}(\tilde T)$, where $B^s_0\in \Gamma _g$  is obtained  from $B^s_0$  by replacing $\Thetax{x(c_p)}$ with $\Thetax{s(c_p)}$ for $p=1,\ldots,l_2$.
Second, we transform $B^x_0$ into some $B^x\in \Gamma _g''$ by using the preorder $\preceq $ and $\preceq '$.
We have $\bar  Y^{\otimes g}(J_{\tilde T,x})\in \mathcal{F}'\big(B^x\big)(1)$ by Lemma \ref{co''} and \ref{co'}.
Here, since  $\preceq$ and $\preceq '$ each does not depend on any $\Thetax{x(c_p)}$, 
we can obtain the desired $B^s\in \Gamma _g''$ from $B^x$ by replacing $\Thetax{x(c_p)}$ with $\Thetax{s(c_p)}$ for $p=1,\ldots,l_2$.
\end{proof}

We take $B\co \mathcal{S}(\tilde T)\rightarrow \Gamma _g''$, $s\mapsto B^s$, as in Lemma \ref{uun}.
For $s\in \mathcal{S}(\tilde T)$, recall from (\ref{bdel}) the modification $\tilde{b}^s_1{(\mathbf{i}_1, \bar{ \mathbf{i}}_1, \ldots , \mathbf{i}_{l_2}, \bar{ \mathbf{i}}_{l_2})}$
of $\tilde{b}^s_1$  with respect to $\mathbf{i}_p\in \mathcal{I}(s(c_p),m_p)$ and $\bar {\mathbf{i}}_p\in \mathcal{I}(s(c_p),n_p)$ for $p=1,\ldots,l_2$. 

For $s\in \mathcal{S}(\tilde T)$, and 
\begin{align*}
\mathbf{i}_p=\big(i_{(p,0,1)},\ldots, i_{(p,0,m_p)}\big)\in \mathcal{I}(s(c_p),m_p), \quad \bar {\mathbf{i}}_p=\big(i_{(p,1,1)},\ldots, i_{(p,1,n_p)}\big)\in \mathcal{I}(s(c_p),n_p),
\end{align*}
for $p=1,\ldots,l_2$, set
\begin{align*}
N^s&(\mathbf{i}_1, \bar{ \mathbf{i}}_1, \ldots , \mathbf{i}_{l}, \bar{ \mathbf{i}}_{l})=\max\{ i_{(p,0,i)}, i_{(p,1,j)} \  | \  1\leq i\leq m_p, \  1\leq j\leq n_p,\  1\leq p\leq l_2 \},
\\
N^s&=\min\{ N^s(\mathbf{i}_1, \bar{ \mathbf{i}}_1, \ldots , \mathbf{i}_{l}, \bar{ \mathbf{i}}_{l}) \  | \  \mathbf{i}_p\in \mathcal{I}(s(c_p),m_p), \  \bar {\mathbf{i}}_p\in \mathcal{I}(s(c_p),n_p), \  1\leq p\leq l_2 \}.
\end{align*}
We use the following lemma.
\begin{lemma}\label{infs}
For $r\geq 0$, there are only finitely many states  $s\in \mathcal{S}(\tilde T)$
such that
$N^s\leq r.$
\end{lemma}
\begin{proof}
Note that for $\mathbf{i}=(i_1,\ldots,i_l)\in \mathcal{I}(k,l)$, $k\geq 0,l\geq 1$ we have 
\begin{align*}
\frac{k}{l}\leq \max(i_1,\ldots,i_l).
\end{align*}
Thus we have
\begin{align*}
w^s:=\max\Big\{\frac{s(c_p)}{m_p}, \frac{s(c_p)}{n_p}\ \Big| \  1\leq p\leq l_2 \Big\}\leq N^s(\mathbf{i}_1, \bar{ \mathbf{i}}_1, \ldots , \mathbf{i}_{l}, \bar{ \mathbf{i}}_{l}),
\end{align*}
for all $\mathbf{i}_p\in \mathcal{I}(s(c_p),m_p),$ $\bar {\mathbf{i}}_p\in \mathcal{I}(s(c_p),n_p),$ $p=1,\ldots, l_2.$
Hence we have
\begin{align}\label{wss}
w^s\leq N^s.
\end{align}
It is not difficult to prove that, for $r\geq 0$, there are only finitely many states  $s\in \mathcal{S}(\tilde T)$ such that
$w^s\leq r.$ This and (\ref{wss}) imply the assertion.
\end{proof}
Lemma \ref{cm} follows from Lemma \ref{infs} and  the following lemma.
\begin{lemma}\label{sa}
For $s\in \mathcal{S}(\tilde T)$ and  $r\geq 0$ such that
$N^s\geq  2r,$ we have
\begin{align*}
\mathcal{F}'\big(b_4\circ b_3\circ \sigma \circ \tilde b^s_1{(\mathbf{i}_1, \bar{ \mathbf{i}}_1, \ldots , \mathbf{i}_{l_2}, \bar{ \mathbf{i}}_{l_2})}\big)(1)\subset G_{r}^{(g)},
\end{align*}
for  $\mathbf{i}_p\in \mathcal{I}(s(c_p),m_p),$ $\bar {\mathbf{i}}_p\in \mathcal{I}(s(c_p),n_p),$ $p=1,\ldots, l_2.$
\end{lemma}
\begin{proof}
The proof is similar to that of Lemma \ref{B''}.
By replacing $s_p$ with $s(c_p)$ for $p=1,\ldots,l_2$, we use the notations and results in the proof of Lemma \ref{B''}.

Fix $\mathbf{i}_p\in \mathcal{I}(s(c_p),m_p)$ and $\bar {\mathbf{i}}_p\in \mathcal{I}(s(c_p),n_p),$ for  $p=1,\ldots, l_2.$
Recall that we color the output edges of $\tilde b^s_1(\mathbf{i}_1, \bar{ \mathbf{i}}_1, \ldots , \mathbf{i}_{l_2}, \bar{ \mathbf{i}}_{l_2})$ with the labels in $\mathcal{P}$ as in Figure \ref{fig:EF}.  Note that 
\begin{align*}
M:=N^s(\mathbf{i}_1, \bar{ \mathbf{i}}_1, \ldots , \mathbf{i}_{l}, \bar{ \mathbf{i}}_{l}) =\max\{ i_a \  | \  a\in \mathcal{P} \}\geq  N^s\geq 2r.
\end{align*}
Since the filtration $\{G_p\}_{p\geq 0}$ is  decreasing, it is enough to prove 
\begin{align}\begin{split}\label{maxa}
\mathcal{F}'\big(b_4\circ b_3\circ\sigma \circ  \tilde{b}^s_1{(\mathbf{i}_1, \bar{ \mathbf{i}}_1, \ldots , \mathbf{i}_{l_2}, \bar{ \mathbf{i}}_{l_2})}\big)(1)\subset G_{\lfloor M/2 \rfloor}^{(g)}.
\end{split}
\end{align}

We prove (\ref{maxa}).
Recall that $\mathcal{P}_{\mathrm{iso}}\subset \mathcal{P}$ denotes the set of isolated labels,
and $\mathcal{P}_{\mathrm{A}}^2$ denotes the set of  unordered pairs $\{a,b\}$ of mutually adjacent labels $a,b\in \mathcal{P}$.
Set 
\begin{align*}
\mathcal{P}_Y=\mathcal{P}\setminus \mathcal{P}_{\mathrm{iso}}=\bigcup \mathcal{P}_{\mathrm{A}}^2.
\end{align*}
Set $M_{\mathrm{iso}}=\max \{  i_a  | \ a\in \mathcal{P}_{\mathrm{iso}}\}$  and  $M_Y=\max\{ i_a | \ a\in \mathcal{P}_Y\}$.
It is enough to prove 
\begin{itemize}
\item[(i)]$\mathcal{F}'\big(b_4\circ b_3\circ \sigma \circ\tilde 
 b_1^s(\mathbf{i}_1, \bar{ \mathbf{i}}_1, \ldots , \mathbf{i}_{l_2}, \bar{ \mathbf{i}}_{l_2})\big)(1)\subset G_{M_{\mathrm{iso}}}^{(g)}$ ($\subset G_{\lfloor M_{\mathrm{iso}}/2 \rfloor}^{(g)}$), and 
\item[(ii)]$ \mathcal{F}'\big(b_4\circ b_3\circ \sigma \circ\tilde  b_1^s(\mathbf{i}_1, \bar{ \mathbf{i}}_1, \ldots , \mathbf{i}_{l_2}, \bar{ \mathbf{i}}_{l_2})\big)(1)\subset G_{\lfloor M_Y/2 \rfloor}^{(g)}$.
\end{itemize}

Let us prove (i).
Recall from (\ref{ab}) that
\begin{align*}
\mathcal{F}'\big( b_3\circ \sigma \circ \tilde b^s_1{(\mathbf{i}_1, \bar{ \mathbf{i}}_1, \ldots , \mathbf{i}_{l_2}, \bar{ \mathbf{i}}_{l_2})}\big)(1)\subset \uqzq^{\otimes l_4}\otimes (\uqe)^{\otimes l_5+\otimes l_6}.
\end{align*}
Thus, it is enough to  prove
\begin{align}
 \mathcal{F}'(b_4)\Big( \uqzq^{\otimes l_4}\otimes (\uqe)^{\otimes l_5+\otimes l_6}\Big) 
&\subset G_{M_{\mathrm{iso}}}^{(g)}.\label{aa5}
\end{align}

Recall that the first $l_4$ input edges of $b_4$ are connected to the left (resp. right) input edges of the $\adx$'s (resp. $\sadx$'s),
and the next  $l_5+l_6$ input edges of $b_4$  go down to the edges of the $\mux$'s and to the right (resp. left) input edges  of the $\adx$'s (resp. $\sadx$'s).
By the definition of $C_p$ and $C'_p$, we have
\begin{align}
\ad (U_{\mathbb{Z},q}\tilde E^{(p)}\otimes \bar U_q^{\ev}) \subset C_p\subset G_p,\label{aa3}
\\
\ad (U_{\mathbb{Z},q}\tilde F^{(p)}\otimes \bar U_q^{\ev}) \subset C_p'\subset G_p,
\end{align}
for $p\geq 0$. We also  have
\begin{align}\begin{split}
  \sad (\bar U_q^{\ev}\otimes U_{\mathbb{Z},q}\tilde E^{(p)})&\subset \sad (\bar U_q^{\ev}\otimes \tilde E^{(p)})
  \\
&\subset \ad \big(S^{-1}(\tilde E^{(p)})\otimes \bar U_q^{\ev}\big)\\
&\subset \ad (\uqz\tilde E^{(p)}\otimes \bar U_q^{\ev})\subset C_p\subset G_p,
\end{split}
\end{align}
for $p\geq 0$. Similarly, we have
\begin{align}
\sad (\bar U_q^{\ev}\otimes U_{\mathbb{Z},q}\tilde F^{(p)})&\subset C'_p\subset G_p,\label{aa4}
\end{align}
for $p\geq 0$. 
Thus,  (\ref{aa5}) follows from (\ref{aa3})--(\ref{aa4}) and  the inclusions
\begin{align}
\mu(\uqzq \otimes G_{p})&=\mu (G_{p}\otimes \uqzq )\subset G_{p},\label{aa2}
\\
\ad(\uqzq \otimes G_{p})&\subset G_{p},
\quad
\sad( G_{p}\otimes \uqzq )\subset G_{p}.\label{aa6}
\end{align}
for $p\geq 0$. We have finished the proof of  (i).

Let us prove (ii).
Recall from (\ref{b31}) that
\begin{align}\label{anl}
\mathcal{F}'\big(b_3\circ \sigma \circ\tilde  b^s_1{(\mathbf{i}_1, \bar{ \mathbf{i}}_1, \ldots , \mathbf{i}_{l_2}, \bar{ \mathbf{i}}_{l_2})}\big)(1)\in 
\Big(\prod_{p=1}^{l_2}\{s_p\}_q!\Big)\cdot \Big(\uqzq^{\otimes l_4}\otimes \mathcal{F}'(Z)(1)\otimes (\uqze)^{\otimes l_6}\Big),
\end{align}
where $Z\in \Yx^{\otimes l_5}\circ \Hom_{\mathcal{A}}(I,A^{\otimes 2l_5})$.
We study $\mathcal{F}'(Z)(1)$ by  using the following inclusions instead of (\ref{K1})--(\ref{K2}).

For $X_1,X_2\in \{E,F\}$ and $i_a,i_b\geq 0$ for $\{a,b\}\in \mathcal{P}_{\mathrm{A}}^2$, we have
\begin{align}
&\dot Y(\uqz \tilde X_1^{(i_a)} \otimes  \uqz \tilde X_2^{(i_b)})\subset  (\{\min(i,j)\}_q!)^{-1}\cdot \mathcal{Y}_{\max(i_a,i_b)}.\label{yd1}
\end{align}
For example, we have
\begin{align*}
\dot Y(\uqz \tilde E^{(2)} \otimes  \uqz \tilde F^{(3)})&=  (\{2\}_q!)^{-1}\dot Y(\uqz e^2 \otimes  \uqz \tilde F^{(3)})
\\
&\subset (\{2\}_q!)^{-1} \mathcal{Y}_{3}.
\end{align*}
We also have
\begin{align}
&\sum \dot Y(\uqz \tilde X_1^{(i_a)} \otimes \uqz D_{\pm}')\otimes \dot Y( \uqz \tilde X_2^{(i_b)} \otimes \uqz D_{\pm}'') \subset  (\{\min(i_a,i_b)\}_q!)^{-1}\cdot (\mathcal{Y}^D)_{\max(i_a,i_b)},
\\
&\sum \dot Y(\uqz \tilde X_1^{(i_a)}\otimes \uqz D_{\pm}' )\otimes \dot Y( \uqz D_{\pm}''\otimes \uqz \tilde X_2^{(i_b)})\subset  (\{\min(i_a,i_b)\}_q!)^{-1}\cdot 
(\mathcal{Y}^D)_{\max(i_a,i_b)}.
\\
&\sum \dot Y(\uqz D_{\pm}'\otimes \uqz \tilde X_1^{(i_a)} )\otimes \dot Y(\uqz \tilde X_2^{(i_b)} \otimes \uqz D_{\pm}'')\subset  (\{\min(i_a,i_b)\}_q!)^{-1}\cdot(\mathcal{Y}^D)_{\max(i_a,i_b)},
\\
&\sum \dot Y(\uqz D_{\pm}'\otimes \uqz \tilde X_1^{(i_a)})\otimes \dot Y(\uqz D_{\pm}''\otimes \uqz \tilde X_2^{(i_b)} )\subset    (\{\min(i_a,i_b)\}_q!)^{-1}\cdot
 (\mathcal{Y}^D)_{\max(i_a,i_b)}.\label{yd2}
\end{align}

By the above inclusions (\ref{yd1})--(\ref{yd2}), and by Lemmas \ref{zf1} and \ref{zf2}, we have
\begin{align}\label{cal}
\mathcal{F}'(Z)(1)  \in \prod_{\{a,b\}\in \mathcal{P}_{\mathrm{A}}^2}\big(\big\{\min (i_{a},i_{b})\big\}_q!\big)^{-1} \cdot G^{(l_5)}_{\lfloor M_Y/2\rfloor }.
\end{align}

Thus, by (\ref{anl}), (\ref{cal}) and   (\ref{u}),   we have
\begin{align*}
\mathcal{F}'\big(b_3\circ \sigma \circ \tilde b^s_1{(\mathbf{i}_1, \bar{ \mathbf{i}}_1, \ldots , \mathbf{i}_{l_2}, \bar{ \mathbf{i}}_{l_2})}\big)(1)& \subset \uqzq^{\otimes l_4}\otimes  G^{(l_5)}_{\lfloor M_Y/2\rfloor }\otimes (\uqze) ^{\otimes l_6}
\\
& \subset \uqzq^{\otimes l_4}\otimes  G^{(l_5+l_6)}_{\lfloor M_Y/2\rfloor }.
\end{align*}

For the proof of the claim, it is enough to prove the inclusion
\begin{align}
\mathcal{F}'(b_4)\Big( \uqzq^{\otimes l_4}\otimes  G^{(l_5+l_6)}_{\lfloor M_Y/2\rfloor }\Big)\subset  G^{(g)}_{\lfloor M_Y/2\rfloor },
\end{align}
which follows from (\ref{aa2}) and (\ref{aa6}).
This completes the proof.
\end{proof}

\begin{acknowledgments}
This work was partially suppoted by JSPS Research Fellowships for Young Scientists.
The author is deeply grateful to Professor Kazuo Habiro and Professor Tomotada Ohtsuki
for helpful advice and encouragement.
\end{acknowledgments}

\end{document}